\documentclass[11pt]{article}
\usepackage{epsfig, amsmath, amssymb, amsthm, times,color}

\usepackage[utf8]{inputenc}

\usepackage{amsmath,amssymb,amsmath}
\usepackage{mathabx}
\usepackage{graphicx}
\usepackage{subfigure}
\usepackage{cite}
\usepackage{mathrsfs}  
\usepackage{hyperref}

\usepackage{cleveref}
\usepackage{autonum} %numerotation intelligente

\numberwithin{equation}{section}

\usepackage{pgf,tikz}
\usetikzlibrary{arrows}
 \usepgflibrary{fpu}

\newcommand{\one}{{\rm 1\mskip-4mu l}}

\def\be{\begin{equation}}
\def\ee{\end{equation}}

\def\1{\mathbf{1}}

\DeclareMathOperator{\dist}{dist}

\DeclareMathOperator{\id}{id}

\newtheorem{theorem}{Theorem}[section]
\newtheorem{proposition}[theorem]{Proposition}
\newtheorem{lemma}[theorem]{Lemma}

\newtheorem{remark}[theorem]{Remark}

\theoremstyle{definition}
\newtheorem{example}[theorem]{Example}

\DeclareMathOperator{\vspan}{span}

\input{macros.sty}

\begin{document}

\title{Metastability in the stochastic nearest-neighbor\\ 
Kuramoto model of coupled phase oscillators}
% \date{December 19, 2024}
\date{Revised version, July 18, 2025}
\author{Nils Berglund, Georgi S. Medvedev, and Gideon Simpson}

\maketitle

\begin{abstract}

The Kuramoto model (KM) of $n$ coupled phase-oscillators is analyzed in this work. The KM on a Cayley graph possesses a family of steady state solutions called twisted states. Topologically distinct twisted states are distinguished by the winding number $q\in\Z$. It is known that for the KM on the nearest-neighbor graph, a $q$-twisted state is stable if $|q|<n/4$. In the presence of small noise, the KM exhibits metastable transitions between $q$--twisted states. Specifically, a typical trajectory remains in the basin of attraction of a given $q$-twisted state for an exponentially long time, but eventually transitions to the vicinity of another such state. In the course of this transition, it passes in close proximity of a saddle of Morse index $1$, called a relevant saddle. In this work, we provide an exhaustive analysis of metastable transitions in the stochastic KM with nearest-neighbor coupling.

We start by analyzing the equilibria and their stability. First, we identify all equilibria in this model. Using the discrete Fourier transform and eigenvalue estimates for rank--1 perturbations of symmetric matrices, we classify the equilibria by their Morse indices. In particular, we identify all stable equilibria and all relevant saddles involved in the metastable transitions. Further, we use Freidlin--Wentzell theory and the potential-theoretic approach to metastability to establish the metastable hierarchy and sharp estimates of Eyring--Kramers type for the transition times. The former determines the precise order, in which the metastable transitions occur, while the latter characterizes the times between successive transitions. The theoretical estimates are complemented by numerical simulations and a careful numerical verification of the transition times. Finally, we discuss the implications of this work for the KM with other coupling types including nonlocal coupling and the continuum limit as $n$ tends to infinity.
\end{abstract}

\leftline{\small 2020 {\it Mathematical Subject Classification.\/} 
60H10, %Stochastic ordinary differential equations (aspects of stochastic analysis)
34F05, %Ordinary differential equations and systems with randomness
60K35, %Interacting random processes; statistical mechanics type models; percolation theory
92B20. %Neural networks for/in biological studies, artificial life and related topics    
}
\noindent{\small{\it Keywords and phrases.\/}
Phase oscillators,
Kuramoto model, XY model,
noise, metastability, synchronization, twisted state.
}  
  
%%%%%%%%%%%%%%%%%%%%%%%%%%%%%%%%%%%%%%%%%%%%%%%%%%%%%%%%%%%%%%%
  
\section{Introduction}\label{sec.intro}

The Kuramoto model (KM) of coupled phase oscillators provides an important framework for studying collective behavior in diverse natural and man-made networks ranging from interacting particle models in statistical physics \cite{Kur75}, to neuronal networks and swarms of fireflies \cite{Str-Sync} in biology, to power grids in engineering \cite{DorBul2014}. It has been extremely useful in revealing new facets of well-known phenomena in coupled networks such as synchronization \cite{Kur75, strogatz_kuramoto_2000, ChiMed19a} and multistability \cite{WilStr06}, as well as in identifying new effects such as chimera states \cite{KurBat02, AbrStr06}. 

{The KM with small noise has been studied in the context of the transition from incoherence to synchronization \cite{StrMir91, StrMir92, strogatz_kuramoto_2000}, as well as the long-time behavior of solutions \cite{BGP2010, BGP2014, Giacomin_Lucon_Poquet14} in populations of mean-field coupled phase oscillators. It has also appeared in studies of nontrivial topological behavior in spin systems
such as the XY and $O(N)$ models\cite{Cosco_Shapira_21, COS24}. In this work, we use the stochastic KM to investigate metastability as a mechanism of pattern formation in dynamical networks subjected to small noise.}

{In a similar vein, metastability in the KM was studied in \cite{DeV12}, and metastability in excitable systems was investigated in \cite{BFG2007a, BFG2007b, MedZhu12}. For an extensive overview of metastability and its applications in physical and biological sciences, we refer the interested reader to the monographs \cite{OliVar-LargeDeviations, BovierHollander-Metastability}.}

{Metastability is commonly observed in dynamical systems that are subject to weak noise, 
or in statistical physics systems at low temperature. At equilibrium,
it is most likely to find such a system in the deepest energy well. However, if the system 
is started in a local minimum of the energy landscape that is different from its global 
minimum, it may take a very long time to reach its true equilibrium state. Examples 
of such situations include wrongly magnetized ferromagnets, supersaturated gases,
and supercooled liquids~\cite{denHollander04,OliVar-LargeDeviations}.
Metastability also plays a role in molecular dynamics~\cite{Aristoff_Lelievre} and in
many biological systems, see for instance~\cite{Creaser_Ashwin_Tsaneva-Atanasova20}. 
The aim of the theory of metastability is to identify metastable states, and to characterise 
the time it takes the system to reach its most stable state, as well as the typical path 
it takes to do so.}

We begin by introducing the KM with identical intrinsic frequencies (cf.~\cite{WilStr06}): 
\begin{equation}
\label{eq:Kuramoto} 
 \dot u_i = K \sum_{j\in S} \sin\left( 2\pi(u_{i+j} - u_i)\right)\;, 
\end{equation} 
where $K>0$ is the coupling strength, $i \in \Lambda = \Z/n\Z$, and $S$ is 
a  finite symmetric set subset of $\Lambda\setminus\set{0}$, i.e.,
 $S\subset\Lambda\setminus\set{0}$ and $-s\in S$ whenever $s\in S$. 
 In addition, let $r=|S|/2$ denote the range of coupling.
 For the most part, we  will deal with the nearest neighbor coupling $S=\{-1, 1\}$
 and $r=1$.

\begin{figure}[!t]
\centering
\subfigure[$q=1$]{\includegraphics[width=4.5cm]{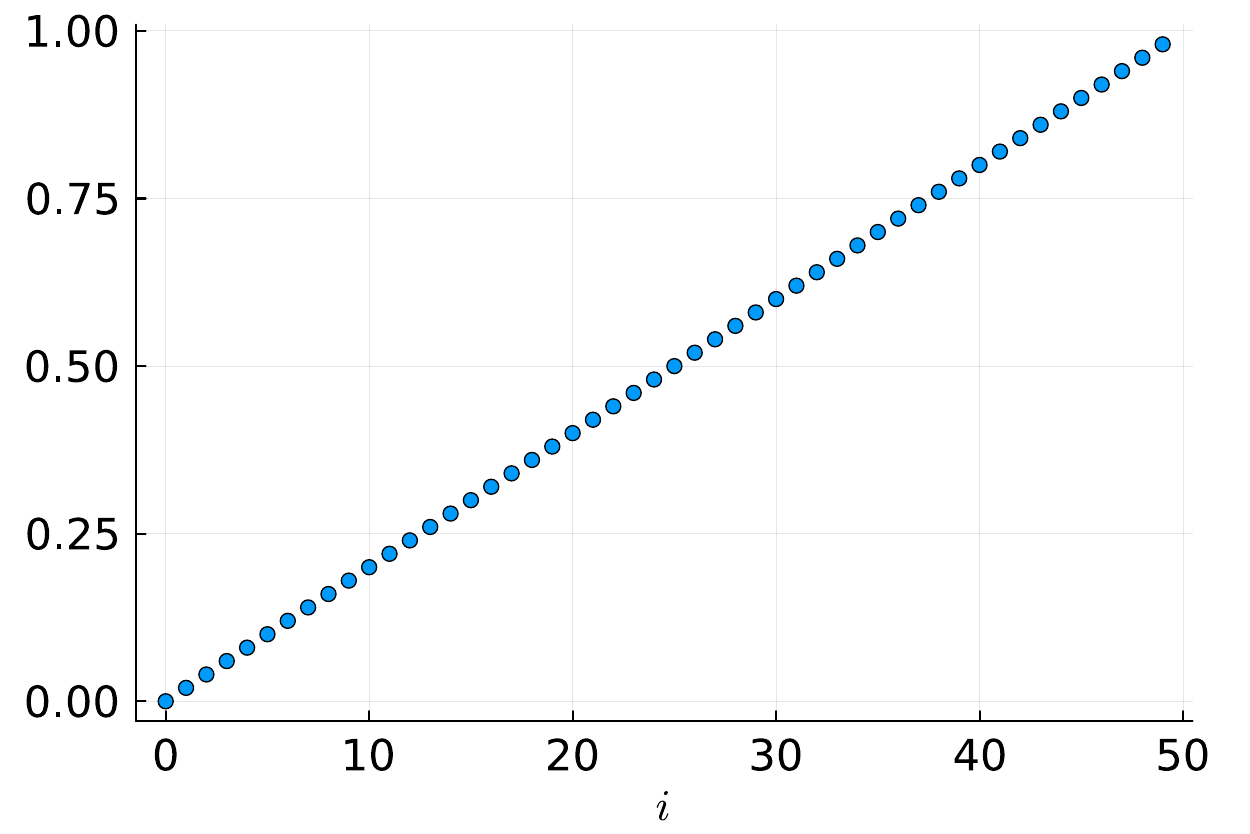}}
\subfigure[$q=2$]{\includegraphics[width=4.5cm]{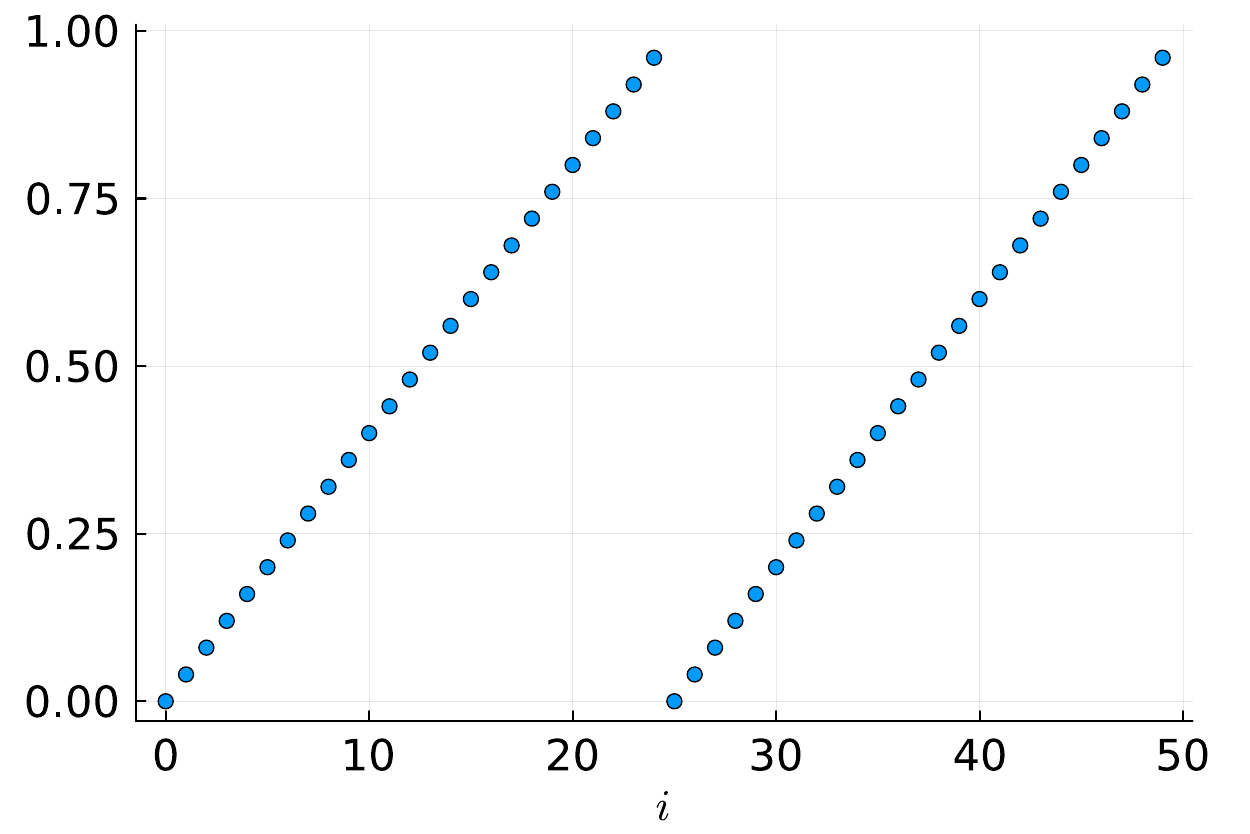}}
\subfigure[$q=3$]{\includegraphics[width=4.5cm]{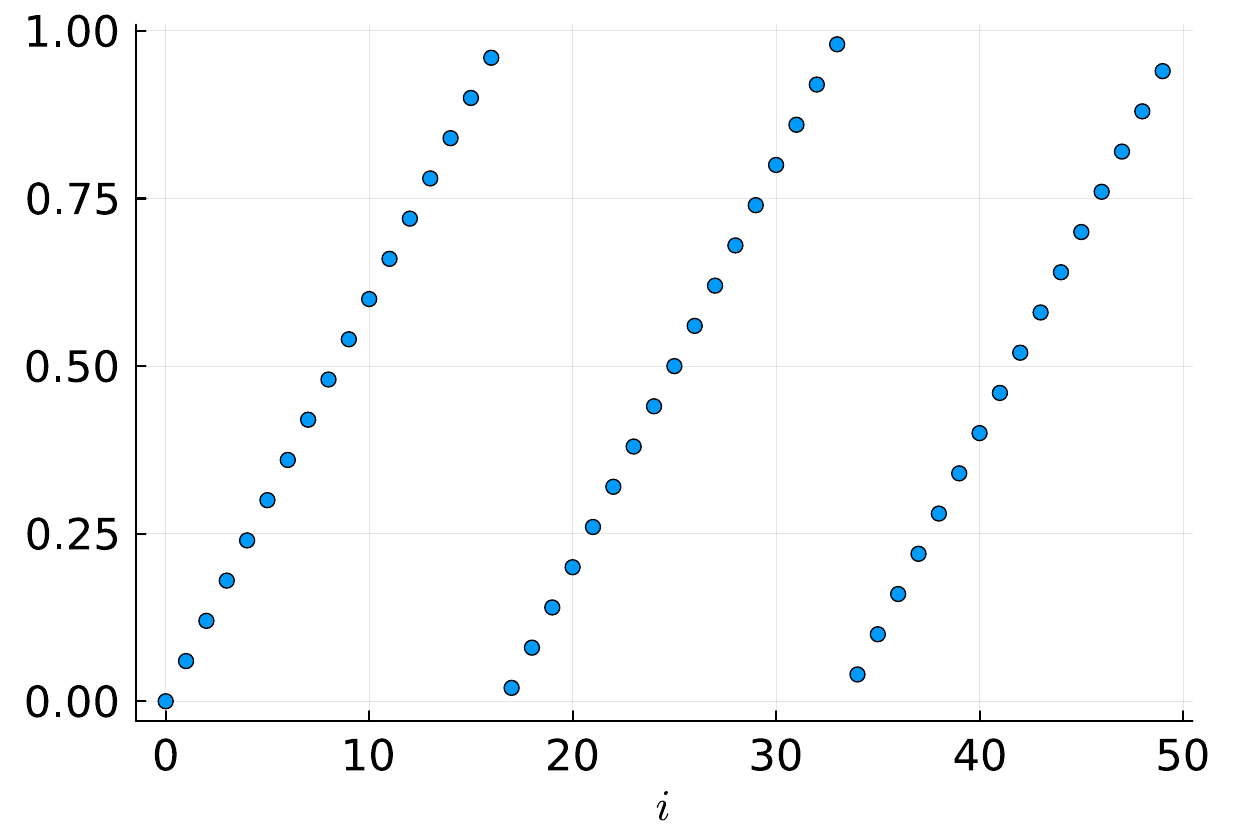}}
\label{f.twist}
\caption{Examples of $q$-twisted state equilibria with $n=50$. The plots represent 
$u^{(q,0)}_i$ as a function of $i$.}
\end{figure}

The KM \eqref{eq:Kuramoto} has a family of steady state solutions
$u^{(q,\ph)}$ of the form 
\begin{equation}
  \label{q-twist}
 u^{(q,\ph)}_i = \frac{qi}{n}+\ph\;, \qquad i\in\Lambda
\end{equation} 
for any integer $q$ such that $-\frac n2 < q \leqs \frac n2$ and $\ph\in [0, 1)$. 
See Figure~\ref{f.twist} for some examples. 
These are called \emph{$q$-twisted states} (cf.~\cite{WilStr06}). The stability of a $q$-twisted state can be determined from an explicit condition on $q, r$, and $n$
(cf.~\cite[Theorem~3.4]{MedTan15b}). In particular, it can be shown that for
$r\ll n$ there are multiple stable $q$-twisted states.

\begin{figure}[!t]
\centering
\subfigure[]{\includegraphics[width=6.5cm]{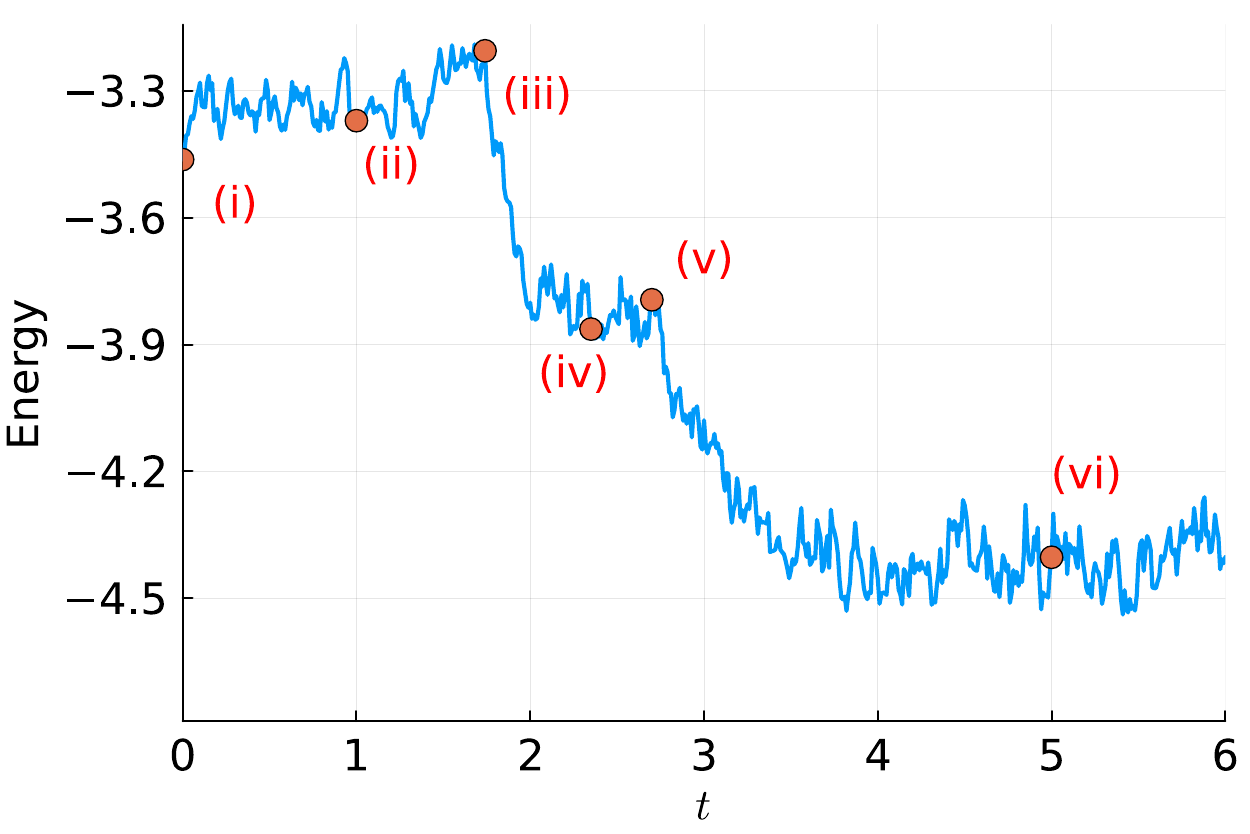}}     
\subfigure[]{\includegraphics[width=6.5cm]{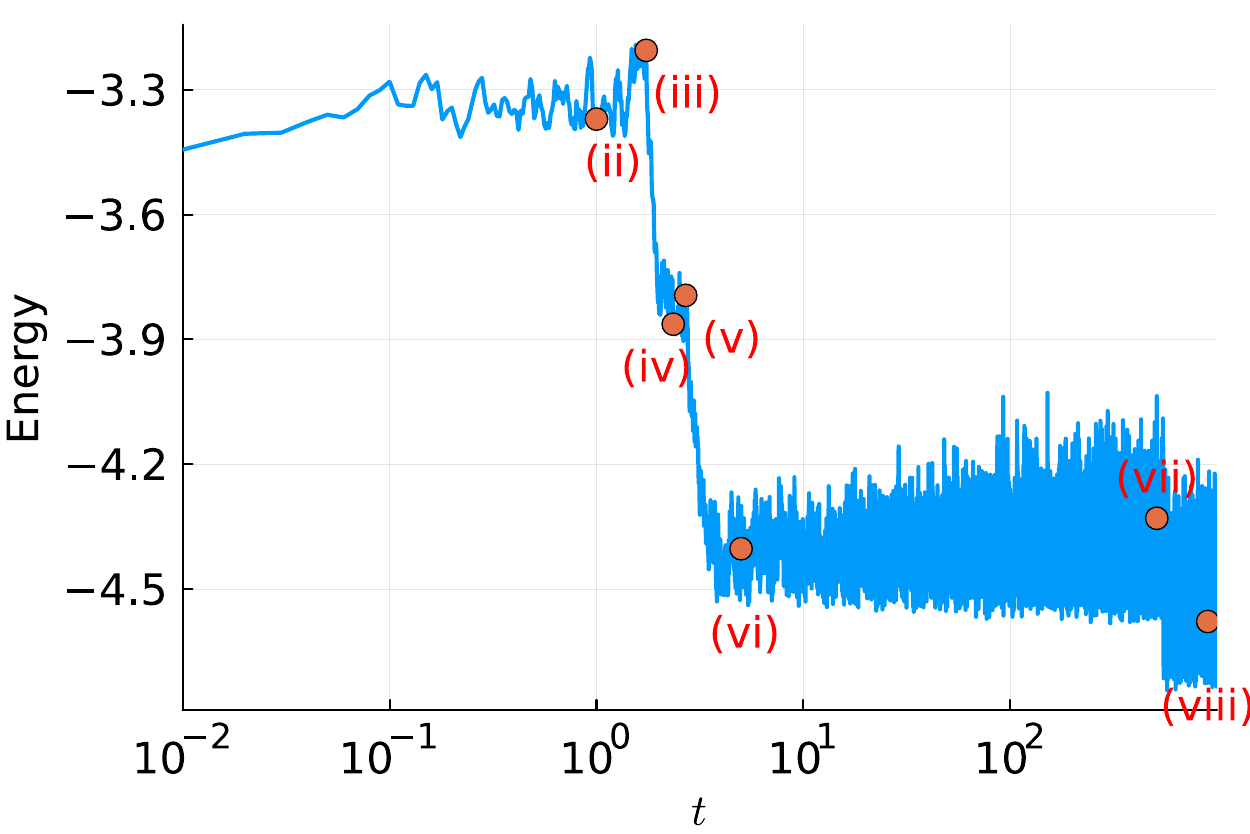}}     

    \caption{Energy of the Kuramoto system with $n=20$ as a function of time.  Labeled points correspond to approximate local minima and transition states that are shown in Figure \ref{fig:transitions-states}. Figure (b) is a continuation of (a) on a longer time scale.}
    \label{fig:eng_n20}
\end{figure}

In this paper, we study the behavior of solutions of \eqref{eq:Kuramoto} under weak
stochastic forcing. To this end, we rewrite it as a gradient system and add a stochastic forcing term on the right hand side, yielding the stochastic differential equation (SDE)
\begin{equation}
\label{eq:KM-forced}
 \6u_t = -\nabla U(u_t)\6t + \sqrt{2\eps}\6W_t\;.
\end{equation}
Here $U$ is the potential (or energy function) on $(\fS^1)^\Lambda$ given by 
\begin{equation}
\label{eq:KM-potential} 
 U(u) = -\frac{K}{4\pi} \sum_{i\in\Lambda} \sum_{j\in S} 
\cos\bigpar{2\pi(u_{i+j} - u_i)}\;,
\end{equation} 
$W_t$ is an $n$-dimensional standard Wiener process, and $\eps>0$ is a small parameter 
controlling the standard deviation of the noise. 

\begin{figure}[!t]
     \centering
     \subfigure[Point (i), $t=0$]{\includegraphics[width=4.6cm]{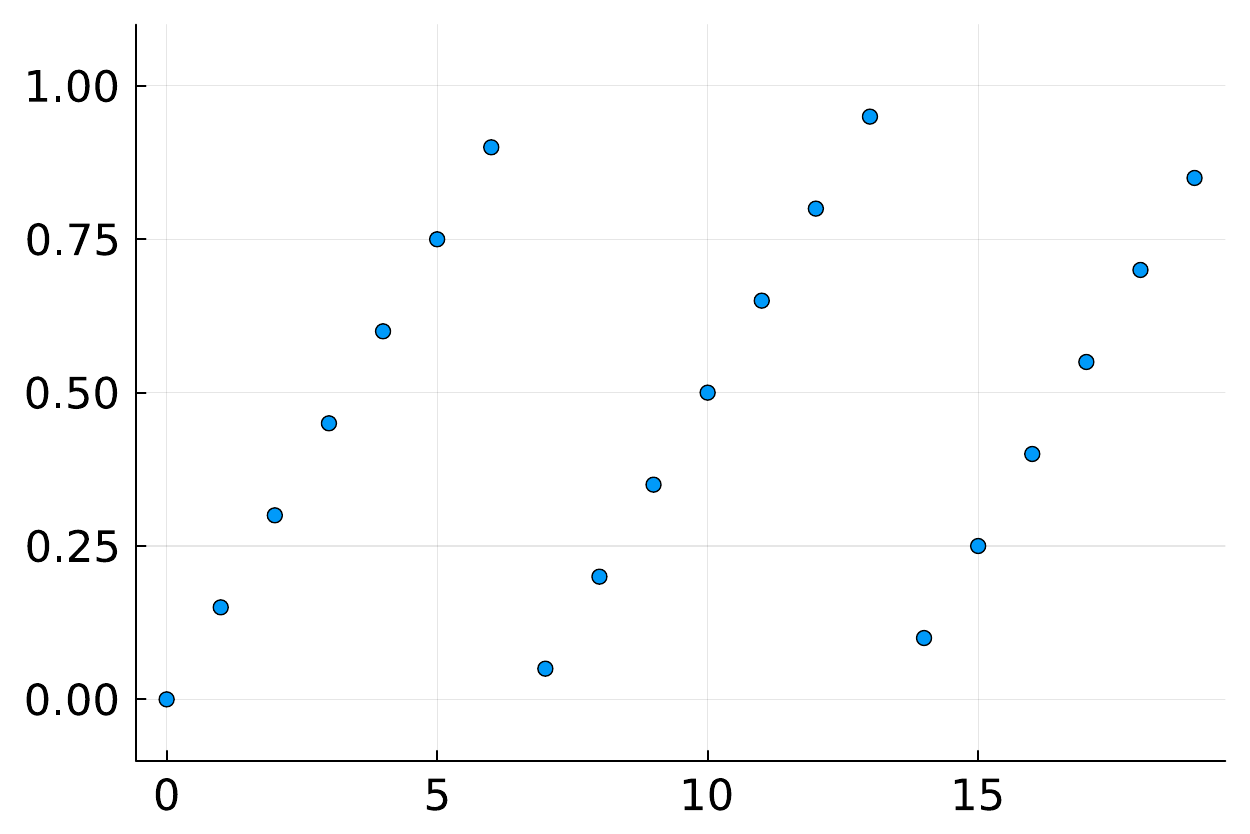}}
    \subfigure[Point (ii), $t=1$]{\includegraphics[width=4.6cm]{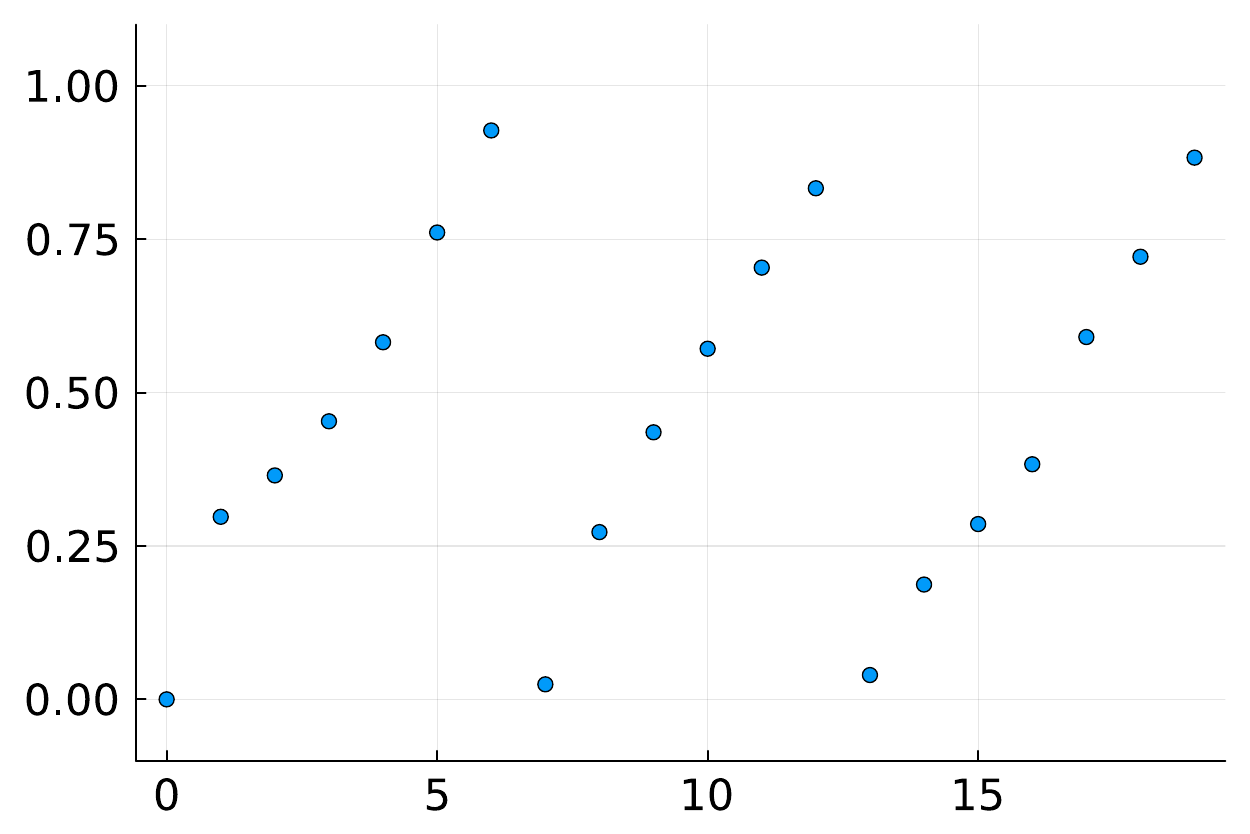}}
      \subfigure[Point (iii), $t=1.74$]{\includegraphics[width=4.6cm]{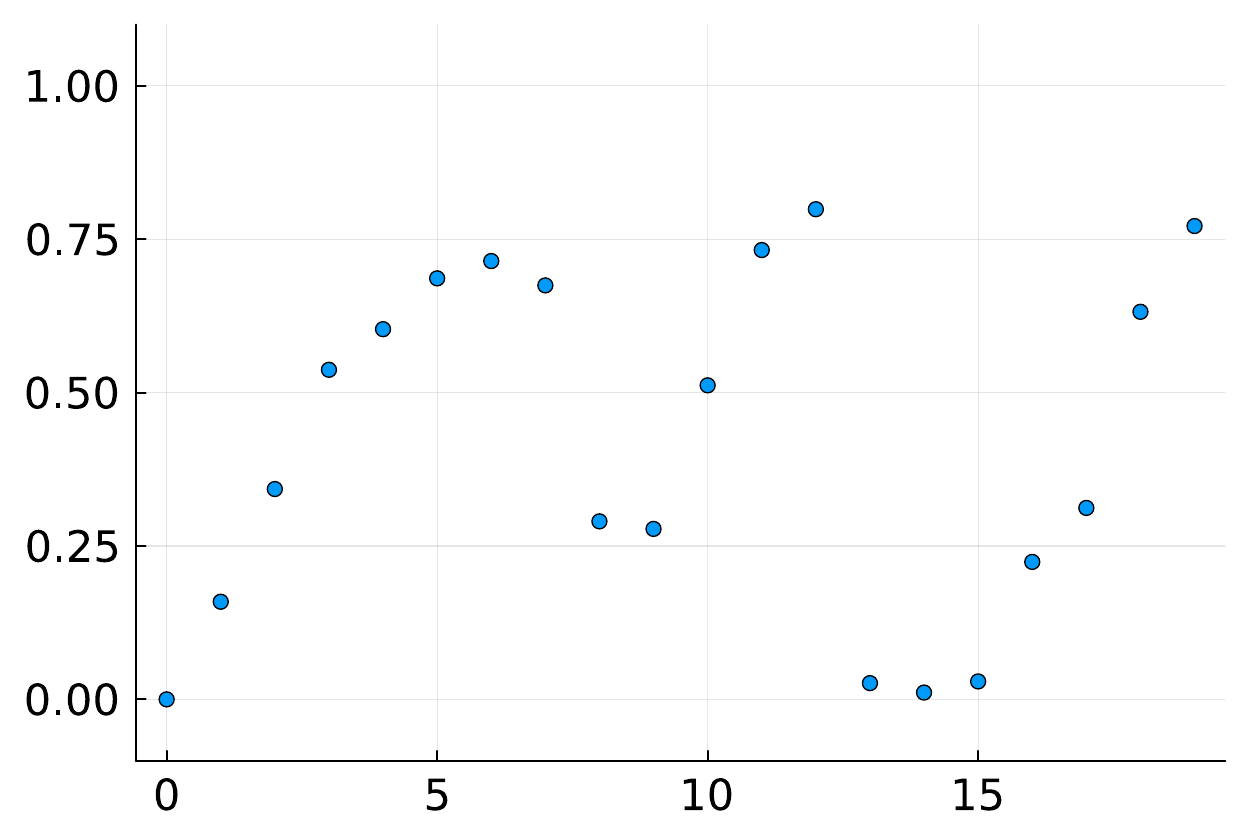}}\\
     \subfigure[Point (iv), $t=2.35$]{\includegraphics[width=4.6cm]{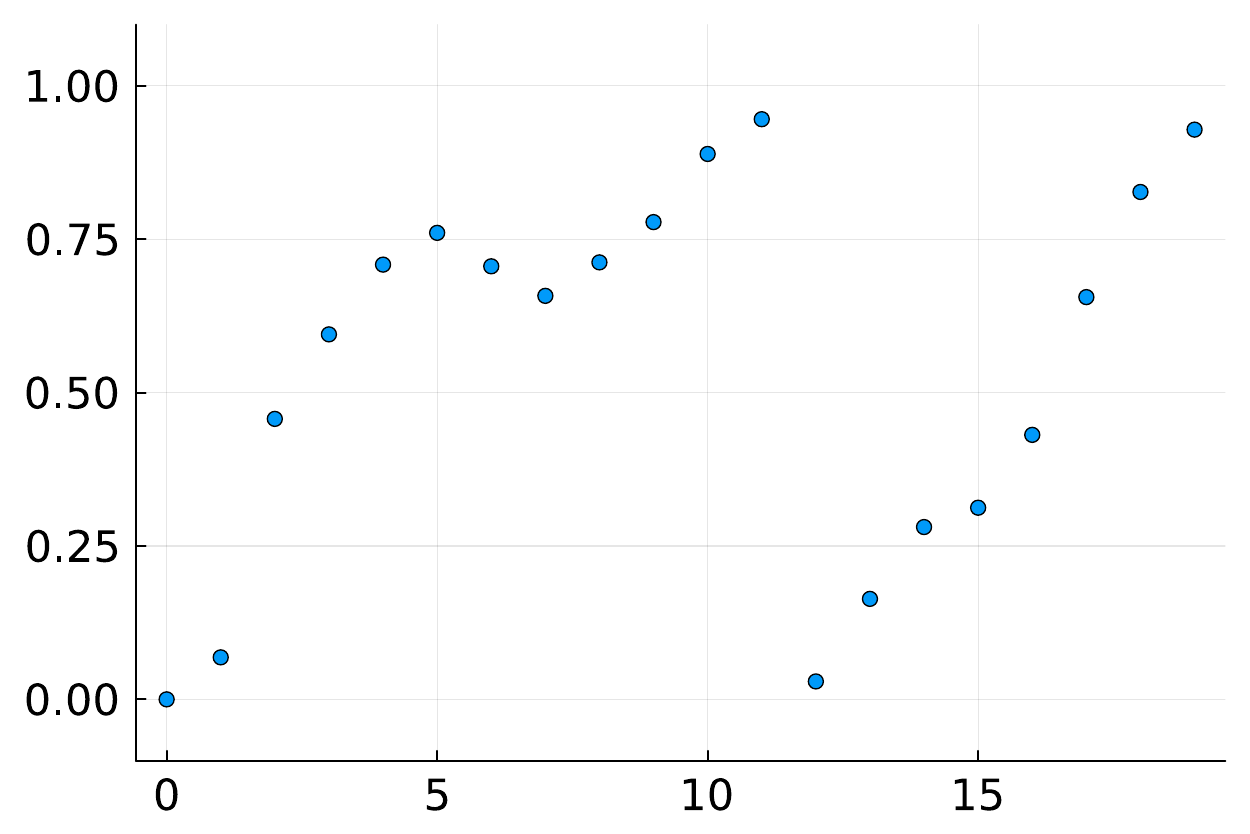}}
     \subfigure[Point (v), $t=2.7$]{\includegraphics[width=4.6cm]{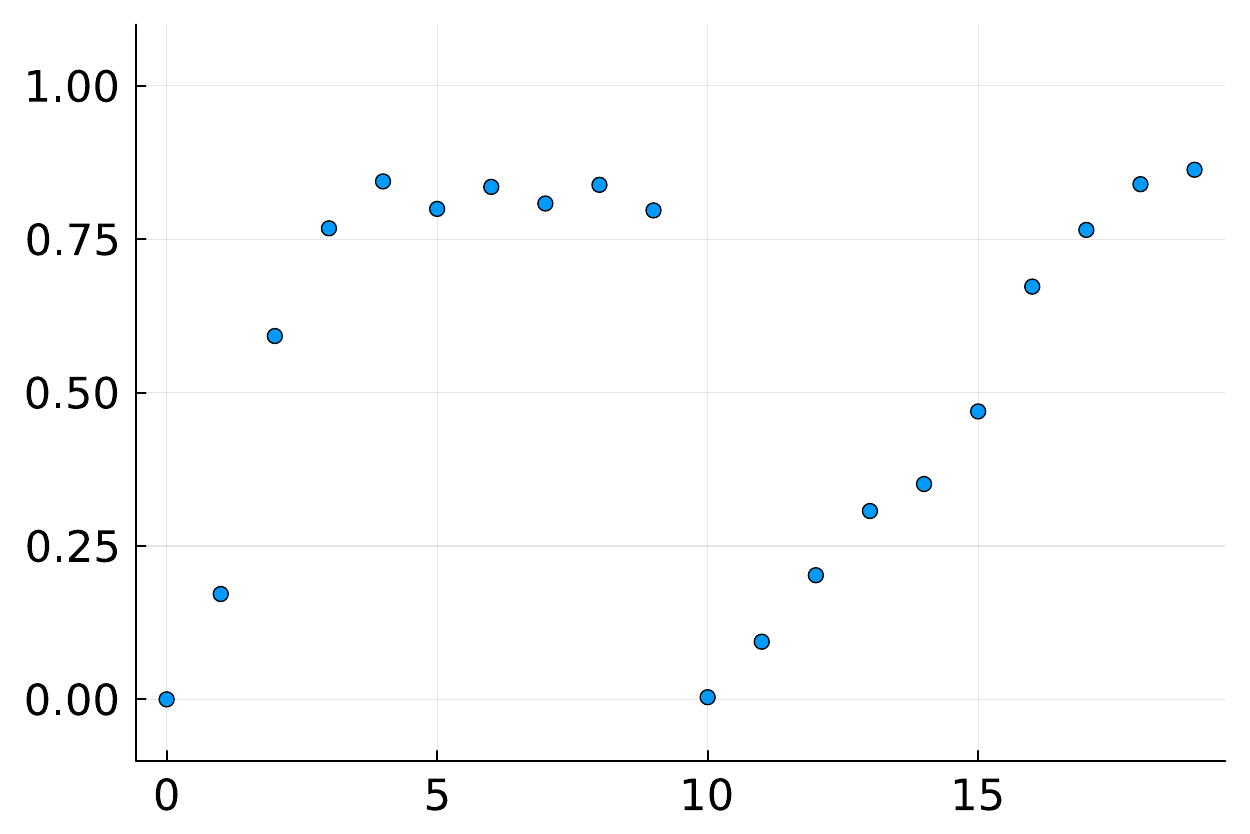}}
    \subfigure[Point (vi), $t=5$]{\includegraphics[width=4.6cm]{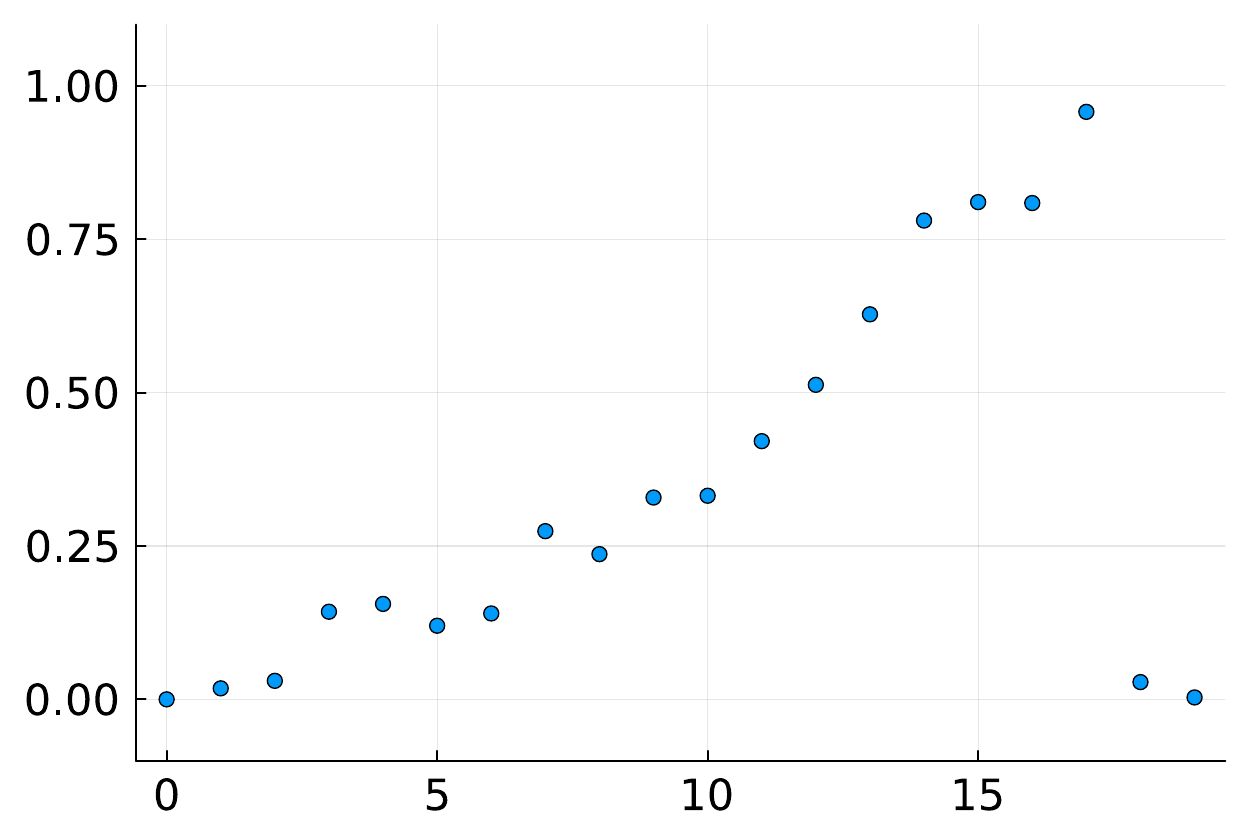}}\\  
     \subfigure[Point (vii), $t=510.65$]{\includegraphics[width=4.6cm]{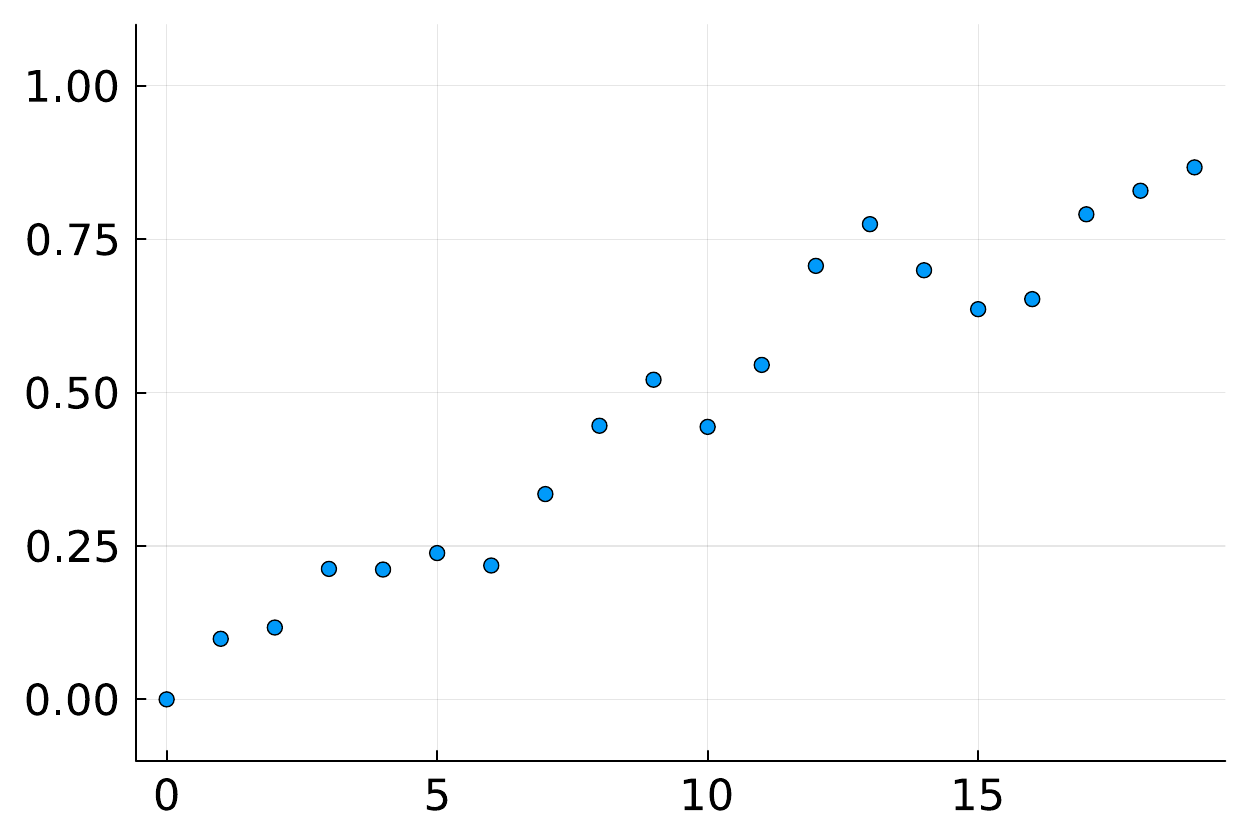}}
     \subfigure[Point (viii), $t=900$]{\includegraphics[width=4.6cm]{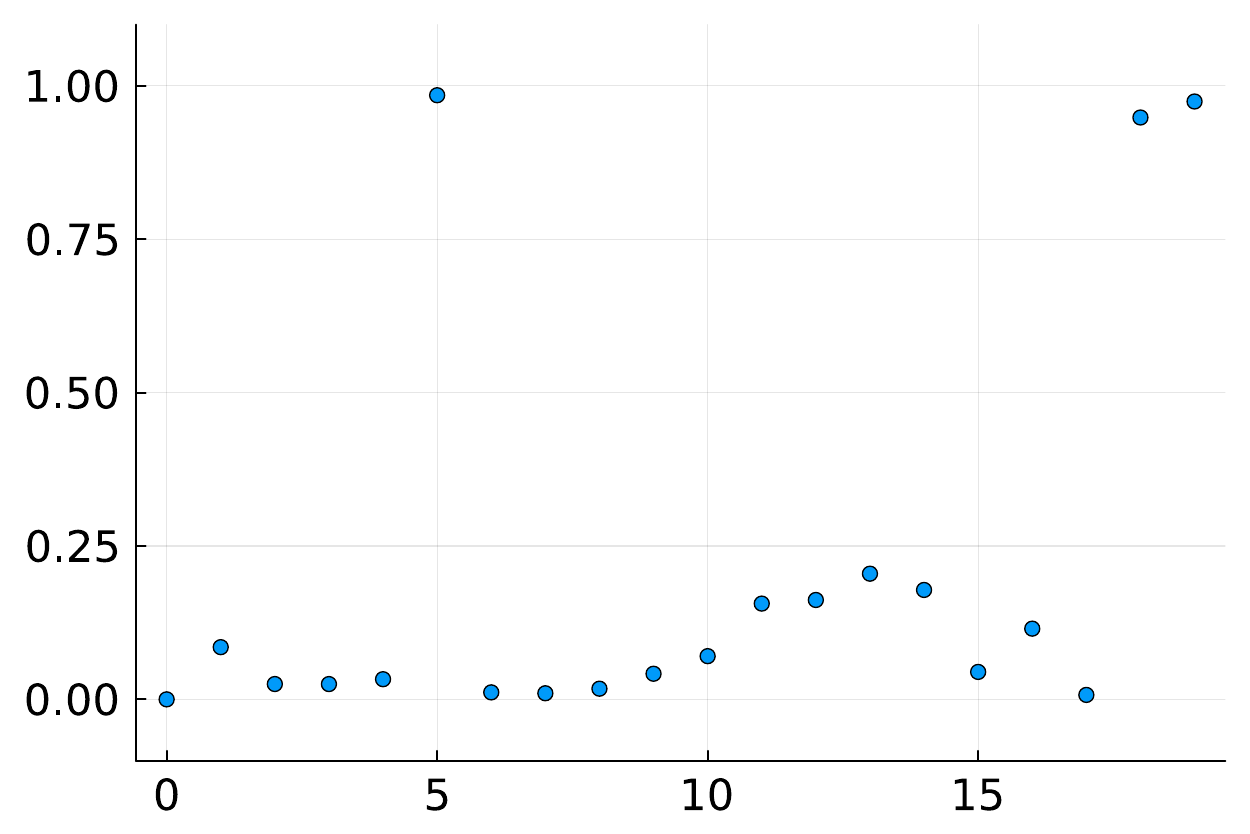}}    
     \caption{The characteristic states visited by the trajectory in the course of a transition from the
       $3$-twisted state to the $0$-twisted state shown in Figure~\ref{fig:eng_n20}.}
     \label{fig:transitions-states}
 \end{figure}

In the presence of noise, the KM exhibits {transitions between neighboring metastable twisted 
states, commonly referred to as \emph{metastable transitions} in the literature.}
For illustration, we discuss the example shown in Figures~\ref{fig:eng_n20} and \ref{fig:transitions-states}. The system initialized at a $3$-twisted state undergoes a series of metastable transitions that take it through the basins of attraction of $2$- and $1$-twisted states before reaching the basin of attraction of a $0$-twisted state. Figure~\ref{fig:transitions-states} shows a few representative snapshots, illustrating both the states that are close to the attractors of the KM and the states close to the boundaries between the basins of attraction of the different twisted states. The main objective of this work is to determine which of these transitions are the most likely, determine the most likely order of these transitions, and to quantify the random time intervals between successive transitions.

The analysis in this paper relies on the Freidlin--Wentzell theory of large deviations \cite{FW-3rd} and the potential-theoretic approach to metastability \cite{BEGK2004}, which allows for sharp asymptotics of the transition times. 
The description of metastability based on the Freidlin--Wentzell theory can be found in \cite[Chapter~6]{FW-3rd} (see also \cite{OliVar-LargeDeviations}).
For the exposition of the potential-theoretic approach and for many different applications of metastability, we refer to \cite{BovierHollander-Metastability}. Large deviations and metastability as the mechanisms driving dynamics of coupled networks have been studied, for instance, in \cite{BFG2007a, BFG2007b, DeV12, MedZhu12, Giacomin_Lucon_Poquet14, Lcon_Poquet17, Giacomin_Poquet_Shapira18}.  In this work, through the careful analysis of the attractors and relevant saddles for the model at hand, we are able to establish the precise hierarchy of metastable transitions and to obtain sharp estimates of the transition times. The results of this work show that coupled systems of simple phase oscillators with identical intrinsic velocities exhibit interesting metastable dynamics due to the complex energy landscape.

Owing to the presence of metastability in many physical systems, particularly those modeled with molecular dynamics, a variety of numerical methods have been developed to accelerate their simulation, \cite{lelievre_mathematical_2020,perez_chapter_2009,zamora_accelerated_2020,tiwary_review_2016,hamelberg_accelerated_2004}. These systems are multistable, analogous to the variety of stable twisted states found in the KM.  Thus, an additional benefit of this work is provide a benchmark problem for accelerated dynamics methods for which the mean first passage times are known analytically in the low temperature limit.

The remainder of this article is organized as follows. In Section~\ref{sec:stoch}, we recall some known properties of gradient systems perturbed by weak Gaussian noise, based on the Freidlin--Wentzell theory of large deviations~\cite{FW-3rd}, and the Eyring--Kramers law~\cite{BEGK2004,Berglund_Kramers}. This theory has to be adapted to the model at hand, owing to symmetries of the potential $U$. We give a description of this procedure in Section~\ref{sec:pot}. In Section~\ref{sec:equilibria}, we provide a detailed analysis of the equilibrium states of the system. In particular, we obtain a complete list of stable states and of all relevant saddles. Section~\ref{sec:EK} contains the main analytical result of this work, Theorem~\ref{thm:main}, which gives sharp Eyring--Kramers asymptotics for the expected transition time from less stable to more stable $q$-twisted states. This section also contains a discussion of the computation of some more general transition times. Section~\ref{sec:numerical} illustrates the results with numerical simulations, and Section~\ref{sec:discuss} provides concluding remarks and an outlook. 

%%%%%%%%%%%%%%%%%%%%%%%%%%%%%%%%%%%%%%%%%%%%%%%%%%%%%%%%%%%%%%%

\section{Stochastically perturbed systems}
\label{sec:stoch}

Consider a finite-dimensional SDE of the form 
\begin{equation}\label{generic}
 \6u_t = -\nabla V(u_t)\6t + \sqrt{2\eps}\6W_t
\end{equation} 
on $\R^n$, where $W_t$ is an $n$-dimensional standard Wiener process. 

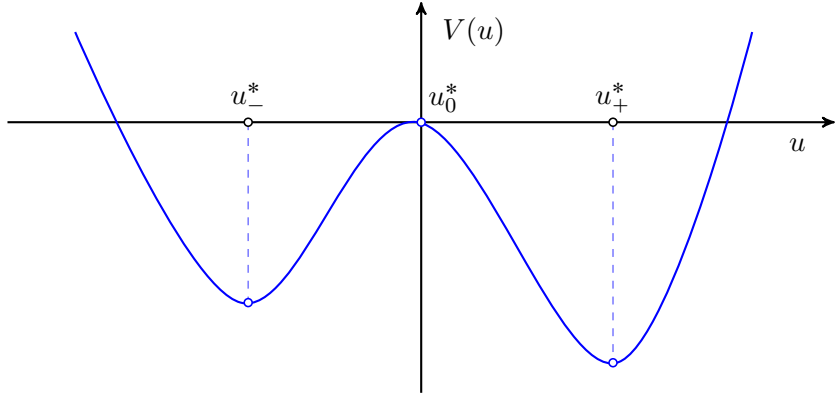
\begin{figure}[!t]
\begin{center}
%\vspace{-5mm}
\begin{tikzpicture}[>=stealth',main node/.style={draw,circle,fill=white,minimum
size=3pt,inner sep=0pt},scale=1,x=1cm,y=0.8cm
]

% grid to help positioning
%\draw[help lines] (-5,-4) grid (5,1);

% axes

\draw[->,thick] (-5.5,0) -> (5.5,0);
\draw[->,thick] (0,-4.5) -> (0,2.0);

% potential

\draw[blue,thick] plot[smooth,tension=.6]
  coordinates{(-4.6,1.5) (-2.4,-3) (-0.05,0) (2.6,-4) (4.4,1.5)};
 
% vertical lines

\draw[blue!50,semithick,dashed] (-2.3,0) -- (-2.3,-3);
\draw[blue!50,semithick,dashed] (2.55,0) -- (2.55,-4);

% local minima and maxima

\node[main node,blue,fill=white,semithick] at (0,0) {}; 
\node[main node,semithick] at (-2.3,0) {}; 
\node[main node,blue,fill=white,semithick] at (-2.3,-3) {}; 
\node[main node,semithick] at (2.55,0) {}; 
\node[main node,blue,fill=white,semithick] at (2.55,-4) {}; 
% \node[main node,semithick] at (-2.0,0) {}; 

\node[] at (-2.3,0.4) {$u^*_-$};
\node[] at (0.3,0.4) {$u^*_0$};
\node[] at (2.55,0.4) {$u^*_+$};

\node[] at (5.0,-0.37) {$u$};
\node[] at (0.7,1.5) {$V(u)$};

\end{tikzpicture}
\vspace{-4mm}
\end{center}
\caption[]{A one-dimensional double-well potential.  
}
\label{fig:double_well}
\end{figure}

Assume first that $V$ is a double-well potential, having local minima at $u^*_-$ and $u^*_+$ (Figure~\ref{fig:double_well}), and that the Hessian matrix of $V$ at $u^*_-$ has eigenvalues 
\begin{equation}
 0 < \lambda_1 \leqs \lambda_2 \leqs \dots \leqs \lambda_n\;.
\end{equation} 
Suppose $G\ni u^*_-$ is an open set contained in the basin of attraction of $u^*_-$. Consider a trajectory of \eqref{generic} starting at $u_0\in G$. After staying  in $G$ for a time of order  $\e^{C/\eps}$, for some $C>0$, it eventually leaves the basin of attraction of $u^*_-$ with probability $1$ (cf.~\cite{FW-3rd}). On its way, it passes with high probability through the vicinity of a saddle $u^*_0$, on the boundary of the basins of attraction of the two local minima of the double-well potential. Assume that this saddle is unique and the Hessian matrix of $V$ at the saddle has eigenvalues 
\begin{equation}
 \mu_1 < 0 < \mu_2 \leqs \mu_3 \leqs \dots \leqs \mu_n\;.
\end{equation} 
The Eyring--Kramers law states that if the system starts near $u^*_-$, and 
$\eps$ is small, then the expected time needed to reach a neighborhood of 
$u^*_+$ is given by
\begin{equation}
\begin{split} 
 \bigexpecin{u^*_-}{\tau_+} 
 &= \frac{2\pi}{\abs{\mu_1}}
 \sqrt{\frac{\bigabs{\det\bigbrak{\frac{\partial^2 V}{\partial^2u}(u^*_0)}}}
 {\det\bigbrak{\frac{\partial^2 V}{\partial^2u}(u^*_-)}}}
 \e^{[V(u^*_0) - V(u^*_-)]/\eps}
 \bigbrak{1 + R(\eps)} \\ 
  &= 2\pi \sqrt{\frac{\mu_2\dots\mu_n}{\abs{\mu_1}\lambda_1\dots\lambda_n}}
 \e^{[V(u^*_0) - V(u^*_-)]/\eps}
 \bigbrak{1 + R(\eps)}
\;,
\label{eq:EK} 
\end{split}
\end{equation} 
where $\frac{\partial^2 V}{\partial^2u}$ stands for the Hessian matrix of $V$, and 
$R(\eps)$ is a remainder going to $0$ as $\eps\to 0$ (one usually has 
information on how fast this happens). 

This result has been extended to potentials with more than two wells, provided one has a so-called \emph{metastable hierarchy}. Given two points $u, v\in\R^n$, we call \emph{path 
connecting $u$ to $v$} a continuous map $\gamma: [0,1]\to\R^n$ such that $\gamma(0) = u$ and $\gamma(1) = v$. In that case, we write $\gamma: u\to v$. The \emph{communication 
height from $u$ to $v$} is defined as 
\begin{equation}
 \overline V(u,v) 
 = \inf_{\gamma: u\to v} \sup_{t\in[0,1]} V(\gamma(t))\;.
 \label{e:cheight1}
\end{equation} 
This can be generalized to the communication height between two sets $A, B\subset \R^n$ 
by 
\begin{equation}
 \overline V(A,B) 
 = \inf_{u\in A, v\in B} \overline V(u,v) \;.
\end{equation} 
A \emph{minimal path from $A$ to $B$} is any path $\gamma: u\to v$, with $u\in A$ and $v\in B$, 
such that 
\begin{equation}
 \sup_{t\in[0,1]} V(\gamma(t)) = \overline V(A,B)\;.
\end{equation} 
{\cite[Proposition~2.1]{BG_MPRF10} shows that if $A$ and $B$ are included in the basins 
of attraction of two different local minima of $V$, then there exists a minimal path 
realising the supremum. In fact, the potential along any minimal path reaches its maximal value on a saddle.}

If there is a finite set of points $w_1, \dots, w_p\in\R^n$ such that 
$V(w_1) = \dots = V(w_p) = \overline V(A,B)$, and any minimal path $\gamma$ from $A$ to $B$
contains at least one of these points, then the $w_i$ are called \emph{relevant saddles} 
between $A$ and $B$. In fact, in the generic case the relevant saddle is unique, provided $A$ and $B$ are contained in the basins of attraction of two different local minima of the potential (see~\cite[Section~2.1]{BG_MPRF10} for more details). 
A related approach for determining the metastable hierarchy can be found in \cite[Chapter~6, \S~6]{FW-3rd}.

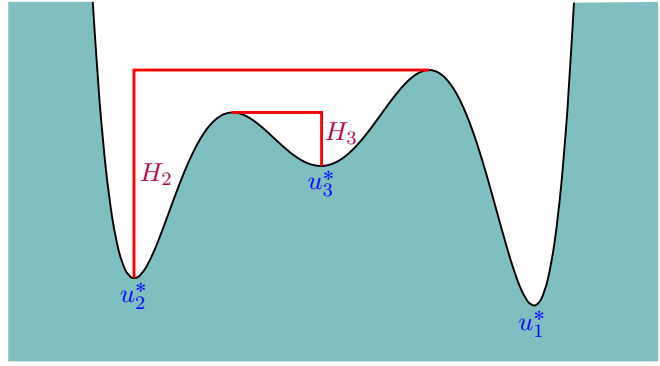
\begin{figure}[!t]
\begin{center}
%\vspace{-5mm}
\scalebox{0.9}{
\begin{tikzpicture}[>=stealth',main node/.style={circle,minimum
size=0.25cm,fill=blue!20,draw},x=1.2cm,y=0.65cm,
declare function={f(\x) = (\x-1)*(\x-5)*(67*\x^4-789*\x^3+2996*\x^2-4044*\x+864)/720;}]

\path[-,smooth,fill=teal!50,domain=0.04:5.964,samples=75,/pgf/fpu,
/pgf/fpu/output
format=fixed] plot (\x, {f(\x)}) -- 
(7,{f(5.964)}) -- (7,-3.5) -- (-1,-3.5) -- (-1,{f(0.04)});

\draw[thick,-,smooth,domain=0.04:5.964,samples=75,/pgf/fpu,/pgf/fpu/output
format=fixed] plot (\x, {
(\x-1)*(\x-5)*(67*\x^4-789*\x^3+2996*\x^2-4044*\x+864)/720 });

\draw[red,very thick] (2.857,0.943) -- (2.857,2.16) -- (1.75,2.16);
\draw[red,very thick] (0.548,-1.61) -- node[right,red!70!blue]
{$\!H_2$} (0.548,3.125) -- (4.18,3.125);

\node[blue] at (0.548,-2.01) {$u^*_2$};
\node[blue] at (2.857,0.543) {$u^*_3$};
\node[blue] at (5.452,-2.648) {$u^*_1$};

\node[red!70!blue] at (3.1,1.7) {$H_3$};
\end{tikzpicture}
}
\vspace{-4mm}
\end{center}
\caption[]{Example of metastable hierarchy. The relevant relative communication heights are 
$H_3 = \overline{V}(u^*_3,\set{u^*_1,u^*_2}) - V(u^*_3)$ and 
$H_2 = \overline{V}(u^*_2,\set{u^*_1}) - V(u^*_2)$.
}
\label{fig:meta_hierarchy}
\end{figure}

With this terminology in place, we can now define the notion of metastable hierarchy. Assume 
that the potential $V$ has a finite set $\set{u^*_1, \dots u^*_{N_0}}$ of local minima. For 
any $k\in\set{1,\dots,N_0}$, we define the metastable set $\cM_k = \set{u^*_1,\dots,u^*_k}$. 
For each $k\in\set{2,\dots,N_0}$, we define the \emph{relative communication height} 
\begin{equation}
 H_k = \overline{V}(u^*_k, \cM_{k-1}) - V(u^*_k)\;.
\end{equation} 
The quantity $H_k$ yields the barrier in the potential landscape that one needs to overcome 
to reach $\cM_{k-1}$ from $u^*_k$. 
We say that the $u^*_i$ are in metastable order, and write $u^*_1 \prec u^*_2 \prec \dots u^*_{N_0}$,  
if there exists $\theta > 0$ such that 
\begin{equation}
 H_k \leqs \min_{i\leqs k} \bigbrak{\overline V(u^*_i, \cM_{k-1}) - V(u^*_i)} - \theta 
\end{equation} 
holds for $k\in\set{2,\dots,N_0}$. Intuitively, this means that the minima are arranged 
in the order of difficulty for escape. The easiest transition is from 
$u^*_{N_0}$ to $\cM_{N_0-1}$, and the hardest transition is from $u^*_2$ to $\cM_1 = \set{u^*_1}$. 
Figure~\ref{fig:meta_hierarchy} gives an example with $N_0 = 3$. 

Assume that all local minima are non-degenerate, and that for each $k\in\set{2,\dots,N_0}$, there is a unique relevant saddle $\hat u^*_k$ between $u^*_k$ and $\cM_{k-1}$, which is also non-degenerate. Then Theorem~3.2 in~\cite{BEGK2004} shows that the first hitting time $\tau_{k-1}$ of $\cM_{k-1}$, starting from $u^*_k$, satisfies \begin{equation}
 \bigexpecin{u^*_k}{\tau_{k-1}} 
 = \frac{2\pi}{\abs{\mu_1(k)}}
 \sqrt{\frac{\bigabs{\det\bigbrak{\frac{\partial^2 V}{\partial^2u}(\hat u^*_k)}}}
 {\det\bigbrak{\frac{\partial^2 V}{\partial^2u}(u^*_k)}}}
 \e^{H_k/\eps}
 \bigbrak{1 + R(\eps)} \;,
\end{equation} 
where $\mu_1(k)$ is the unique negative eigenvalue of the Hessian at $\hat u^*_k$. 
Note that the only difference with~\eqref{eq:EK} lies in the points where the Hessians are evaluated. 

Note that~\cite[Theorem~3.2]{BEGK2004} does not make any statement on other metastable transitions 
than the ones from $u^*_k$ to $\cM_{k-1}$. However, results in~\cite{Rezakhanlou_Seo_2021scaling,Seo2020,B2023} show that the system can be approximated, 
in a suitable sense, by a continuous-time Markov chain on the set of local minima. This allows 
estimating other expected transition times, in a way we illustrate in Section~\ref{ssec:EK_general}. 

{Let us point out that while the above results apply to the expected transition time, 
some information is also available on the distribution of this time. For instance, 
results in~\cite[Chapter 6, \S 5]{FW-3rd} imply the concentration result 
\begin{equation}
 \lim_{\eps\to0} \bigprobin{u^*_k}{\e^{(H_k-\eta)/\eps} < \tau_{k-1} < \e^{(H_k+\eta)/\eps)}}
 =  1
\end{equation}
for all $\eta > 0$. In other words, with a probability going to $1$ as $\eps\to0$, 
the transition time will belong to a window $[\e^{(H_k-\eta)/\eps},\e^{(H_k+\eta)/\eps}]$
containing the expected value. Since $\e^{\eta/\eps}$ blows up as $\eps\to0$ for 
fixed $\eta>0$, this concentration result is not very precise. However, in~\cite{Day1}, 
Day has shown that the transition time distribution is exponentially asymptotic, 
in the sense that 
\begin{equation}
\label{eq:Day} 
 \lim_{\eps\to0} \bigprobin{u^*_k}{\tau_{k-1} > s\bigexpecin{u^*_k}{\tau_{k-1}}} 
 = \e^{-s}
\end{equation}
for any $s>0$. See also~\cite[Theorem~1.4]{BGK2005} for a statement with sharper 
error bounds valid for gradient systems. These results imply in particular 
that the expectation of the logarithm of the transition time behaves like $H_k/\eps$, 
see Remark~\ref{rem:concentration} below. 
}

%%%%%%%%%%%%%%%%%%%%%%%%%%%%%%%%%%%%%%%%%%%%%%%%%%%%%%%%%%%%%%%

\section{The potential landscape}
\label{sec:pot} 

In this section, we provide an analysis of symmetries of the potential landscape of the 
Kuramoto model~\eqref{eq:Kuramoto}. In particular, we construct a fundamental domain that accounts for the angular nature of the variables $u_i$, and show how the degeneracy of the potential under 
global phase shifts can be dealt with by a simple change of variables. 

We will focus on the nearest-neighbor coupling case $S = \set{-1,1}$. Then the 
Kuramoto potential~\eqref{eq:KM-potential} can be written as  
\begin{equation}
\label{eq:U_nn} 
 U(u) = -\frac{K}{2\pi} \sum_{i\in\Lambda} \cos(2\pi(u_{i+1}-u_i))\;,
\end{equation} 
which can be viewed as a function from $\R^\Lambda$ to $\R$ (or, equivalently, from 
$\R^n$ to $\R$). This potential has a large symmetry group, which has important implications 
{for} the analysis. 

%%%%%%%%%%%%%%%%%%%%%%%%%%%%%%%%%%%%%%%%%%%%%%%%%%%%%%%%%%%%%%%

\subsection{Symmetries}
\label{ssec:sym}

A \emph{symmetry} is a map $g:\R^\Lambda\to\R^\Lambda$ such that 
\begin{equation}
 U(g(u)) = U(u) 
 \qquad
 \forall u\in\R^\Lambda\;.
\end{equation} 
The potential~\eqref{eq:U_nn} has the following symmetries: 

\begin{enumerate}
\item   \textbf{Integer translations:}
\begin{equation}
 T_k(u_0, \dots, u_{n-1}) = (u_0 + k_0, \dots, u_{n-1} + k_{n-1})\;, 
 \qquad 
 k\in\Z^\Lambda\simeq\Z^n\;.
\end{equation} 
This is due to the $2\pi$-periodicity of $\sin$. 

\item   \textbf{Global phase shift:}
\begin{equation}
 S_\ph(u_0, \dots, u_{n-1}) = (u_0 + \ph, \dots, u_{n-1} + \ph)\;, 
 \qquad
 \ph\in\R\;.
\end{equation} 
This is due to the translation invariance of the KM. 
 
\item   \textbf{Cyclic permutation of components:}
\begin{equation}
\label{eq:Cp} 
 {C_j(u_0, \dots, u_{n-1}) = (u_{j}, \dots, u_{n-1+j})}\;, 
 \qquad 
 {j \in \Lambda}\;.
\end{equation} 
This is due to nearest neighbors having the same interaction. 

\item   \textbf{Inversion:}
\begin{equation}
 I(u_0, \dots, u_{n-1}) = (-u_0,\dots,-u_{n-1})\;.
\end{equation} 
This follows from the fact that the interaction with the left and right neighbor 
are the same. 
\end{enumerate}

It will be convenient to take as phase space a fundamental domain with respect to the first two 
types of symmetries. This means that we consider two points $u,v\in\R^\Lambda$ to be equivalent if, 
and only if, there exist $k\in\Z^\Lambda$ and $\ph\in\R$ such that $v = T_kS_\ph u$. A \emph{fundamental 
domain} (or \emph{unit cell}) is then a set $\cD$ of representatives of the equivalence classes defined by this equivalence relation.

If we only considered integer translations, a natural choice of fundamental domain would be 
the torus $\T^\Lambda = \R^\Lambda/\Z^\Lambda$ (or, equivalently, $\T^n = \R^n/\Z^n$). However, 
since we also consider global phase shifts, we proceed differently. We first define the hyperplane 
\begin{equation}
 \Sigma = \set{u\in\R^\Lambda \colon u_0 + \dots + u_{n-1} = 0}\;.
\end{equation} 
This hyperplane corresponds to modding out global phase shifts, the representative of a point 
$u\in\R^\Lambda$ being simply its orthogonal projection on $\Sigma$.

We now want to account for integer translations as well. 
If $k = (k_0,\dots,k_{n-1})^\top\in\Z^\Lambda$ is a point with integer coordinates, its orthogonal 
projection on $\Sigma$ has coordinates 
\begin{align}
 k - \frac1n\pscal{k}{\1}\1
 &= \frac1n {\bigpar{(n-1)k_0-k_1-\dots-k_{n-1}, -k_0+(n-1)k_1-k_2-\dots-k_{n-1},\dots)}^\top} \\
 &= \sum_{i=0}^{n-1} k_i v_i\;,
\label{eq:proj_lattice} 
\end{align} 
where 
\begin{equation}
\label{eq:def1} 
 \1 = \bigpar{1,\dots,1}^\top \in\R^\Lambda\;,
\end{equation} 
and
\begin{align}
 v_0 &= \bigpar{1 -\tfrac{1}{n}, -\tfrac{1}{n}, \dots, -\tfrac{1}{n}}^\top\;, \\
 v_1 &= \bigpar{-\tfrac{1}{n}, 1-\tfrac{1}{n}, \dots, -\tfrac{1}{n}}^\top\;, \\
 &\dots \\ 
 v_{n-1} &= \bigpar{-\tfrac{1}{n}, \dots, -\tfrac{1}{n}, 1-\tfrac{1}{n}}^\top\;.
\end{align}
The set of vectors $(v_0, \dots, v_{n-2})$ forms a (non-orthonormal) basis of $\Sigma$.
The projections~\eqref{eq:proj_lattice} form a lattice given by integer linear combinations 
of these vectors. 

This motivates the choice of fundamental domain 
\begin{equation}
 \cD = \bigset{y_0 v_0 + \dots + y_{n-2} v_{n-2} \colon -\tfrac12 \leqs y_i < \tfrac12, 
 i = 0,\dots,n-2}\;.
\end{equation} 

\begin{figure}[!t]
 \centerline{
 \includegraphics[height=80mm]{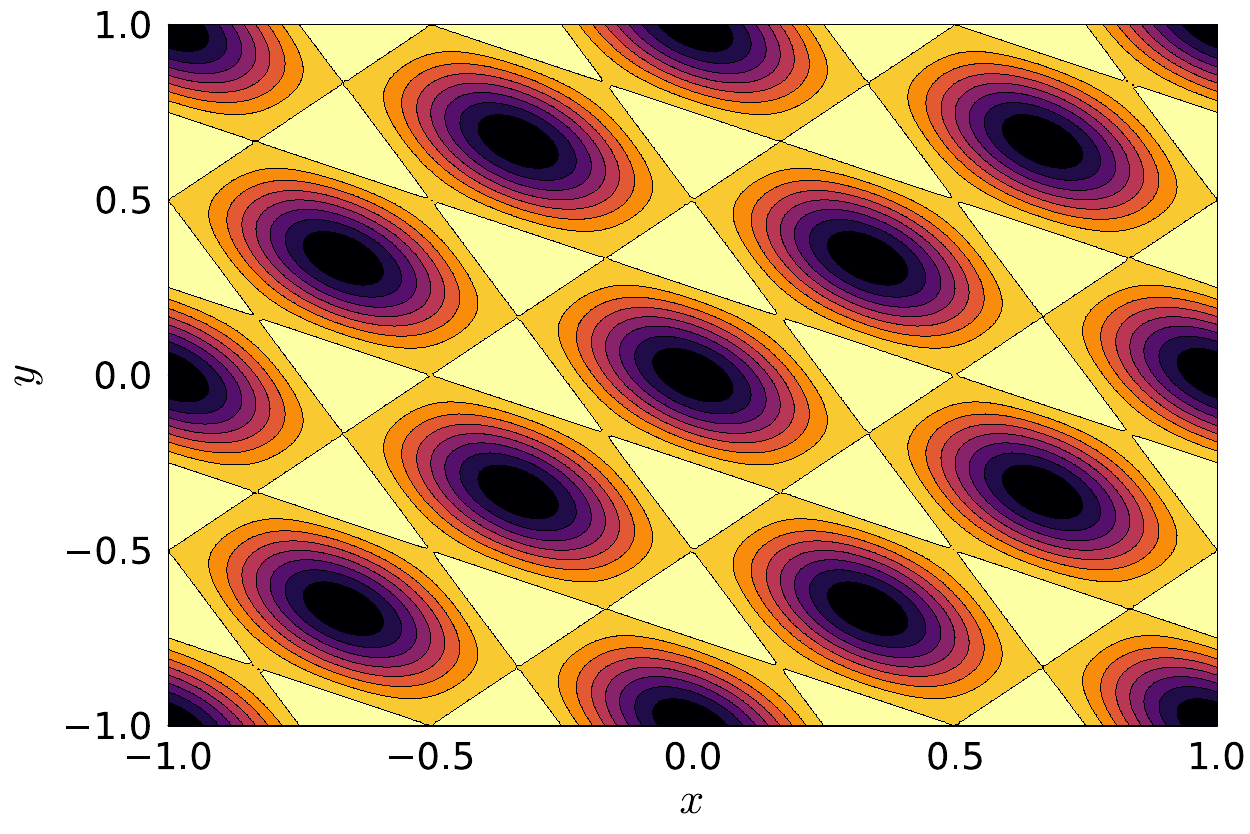}}
\caption[]{Contour plot of the potential~\eqref{eq:U_n3}, obtained by restricting 
the Kuramoto potential~\eqref{eq:KM-potential} to the hyperplane $\Sigma$, for $n=3$.
The hexagons are basins of attraction of the $0$-twisted state $u^{(0)}$ and its 
translates, while the triangles contain local maxima of the potential.}
\label{fig:potential_3d} 
\end{figure}

\begin{example}
\label{ex:n3} 
Consider the case $n=3$. Using the fact that $u_2 = -u_1-u_0$ if $u\in\Sigma$, we obtain 
\begin{equation}
\label{eq:U_n3} 
 U(u) = -\frac{K}{2\pi}
 \Bigbrak{\cos\bigpar{2\pi(u_1-u_0)} + \cos\bigpar{2\pi(u_0+2u_1)}
 + \cos\bigpar{2\pi(2u_0+u_1)}}\;.
\end{equation} 
Figure~\ref{fig:potential_3d} shows a contour plot of this function. The fundamental 
domain $\cD$ is a parallelogram with vertices $\pm(\frac16,\frac16)^\top$ and  
$\pm(\frac12,-\frac12)^\top$.
In $y$-coordinates, the expression for the potential becomes 
\begin{equation}
\label{eq:U_n3y} 
 U(y_0v_0 + y_1v_1) = -\frac{K}{2\pi}
 \Bigbrak{\cos\bigpar{2\pi(y_1-y_0)} + \cos\bigpar{2\pi y_1}
 + \cos\bigpar{2\pi y_0}}\;.
\end{equation} 
\end{example}

The following result gives an explicit expression for the coordinate transformation 
between the $u_i$ and the $y_j$.

\begin{lemma}
If $u = \sum_{i=0}^{n-2} y_i v_i$, then for any $i\in\set{0,\dots,n-2}$ one has 
\begin{equation}
\label{eq:u_to_y} 
 y_i \equiv u_i + \sum_{j=0}^{n-2} u_j\;, \qquad 
 u_i \equiv y_i - \frac1n \sum_{j=0}^{n-2} y_j\;,
\end{equation} 
where $\equiv$ stands for equality modulo $1$. 
\end{lemma}
\begin{proof}
The transformation from $y$ to $u$ can be written $u = My$, where 
$M = \one_{n-1} - \frac1n \1_{n-1}\1_{n-1}^\top$. Here $\1_{n-1}$ denotes the vector of dimension $n-1$
having all components equal to $1$, and $\one_{n-1}$ is the identity matrix of dimension 
$n-1$. Using the fact that $\1_{n-1}^\top\1_{n-1} = n-1$, one easily checks that $M^{-1} = \one_{n-1} + \1_{n-1}\1_{n-1}^\top$.
\end{proof}

%%%%%%%%%%%%%%%%%%%%%%%%%%%%%%%%%%%%%%%%%%%%%%%%%%%%%%%%%%%%%%%

\subsection{Changing coordinates}
\label{subsec:coordinates}

The results on metastability outlined in Section~\ref{sec:stoch} do not 
apply directly to the stochastically forced KM~\eqref{eq:KM-forced}, 
because the Hessian matrix at any equilibrium point has a zero eigenvalue,
as a consequence of the translation invariance of the potential. 
This can be remedied by a change of variables compatible with our choice of 
fundamental domain. A similar system of coordinates was used in the analysis of synchronization in \cite[\S~5.1]{MedZhu12}.

Let
\begin{equation}
 q_0=\frac{1}{\sqrt{n}} \1\;, 
\end{equation} 
where $\1$ has been defined in~\eqref{eq:def1}, 
and choose $q_1, q_2, \dots, q_{n-1}\in \R^\Lambda$ such that $q_0,q_1, q_2, \dots, q_{n-1}$
form an orthonormal basis in $\R^\Lambda \simeq \R^n$. Let
\begin{equation}
 Q = (q_1,q_2, \dots, q_{n-1})\in \R^{n\times (n-1)}\;.
\end{equation} 
Then the change of variables 
\begin{equation}\label{change}
 u_t = \bar u_t q_0 + Qv_t, \qquad \bar u_t \in\R\;, \quad  v_t\in\R^{n-1}
\end{equation}
yields the system 
\begin{align}
 \6\bar u_t &= \sqrt{2\eps} \6\widehat W^{0}_t\;, \\
 \6v_t &= Q^\top \nabla U(Qv - \bar u q_0) + \sqrt{2\eps} \6\widehat W_t\;, 
\end{align}
where $\widehat W^{0}_t$ and $\widehat W_t$ are independent standard Brownian motions, 
of respective dimensions $1$ and $n-1$. 
Here we have used $\1^\top \nabla U = 0$, due to the translation invariance of $U$, 
the facts that $Q^\top Q = \one_{n-1}$, and 
$Q^T q_0 = 0$ (both because the change of variables is orthogonal), and the invariance 
of Brownian motion under rotations. 
Furthermore
\begin{equation}
Q^\top \nabla U(Qv - \bar u q_0)
= Q^\top \nabla U(Qv)
= \nabla_v V(v)\;,
\end{equation}
where $V(v) = U(Qv)$, and we have again used translation invariance of $U$. 
It follows that the dynamics can be described by the set of equations 
\begin{align} 
\label{new-diag}
\6\bar u_t &= \sqrt{2\eps} \6\widehat W^{0}_t\;, \\
 \6v_t &= \nabla_v V(v_t) + \sqrt{2\eps} \6\widehat W_t\;.  
\end{align}
In this system, the evolutions of $\bar u_t$ and $v_t$ are completely decoupled. The component 
$\bar u_t$ performs a simple Brownian motion, of variance $2\eps t$, while $v_t$ obeys an $(n-1)$-dimensional SDE in the hyperplane $\Sigma$, which will in general 
be non-degenerate. The dynamics of $v_t$ can be mapped to an SDE 
on $\cD$ with periodic boundary conditions. 

\begin{remark}
\label{rem:Sigma} 
There is one difference between the equation on $\Sigma$ and its projection on $\cD$. 
While the latter admits an invariant probability measure, with density proportional to 
$\e^{-V(v)/\eps}$, the former does not admit such a measure, because $\Sigma$ is unbounded. 
This is because on large scales, the behavior of the process in $\Sigma$ is closer to that 
of a random walk. 
\end{remark}

In what follows, it will be more convenient to work in variables $u$ instead of $v$, because 
this simplifies the computation of equilibrium points and Hessian matrices. The two points of 
view are, however, equivalent, if one disregards the zero eigenvalues of the Hessians in the 
direction $\1$. 

%%%%%%%%%%%%%%%%%%%%%%%%%%%%%%%%%%%%%%%%%%%%%%%%%%%%%%%%%%%%%%%

\section{The equilibria}
\label{sec:equilibria} 

In this section, we analyze the equilibria of~\eqref{eq:Kuramoto}. These are critical points 
of the potential $U$, meaning that they satisfy $\nabla U(u) = 0$. The Hessian matrix of $U$ 
at a critical point has real eigenvalues, one of which is $0$ due to the translation invariance 
of $U$. The~\emph{Morse index} of a critical point is the number of strictly 
negative eigenvalues of the Hessian. 

We will be particularly interested in two types of equilibria, which are important for
the description of the metastable transitions in the stochastically forced model \eqref{eq:KM-forced}.
The first type consists in critical points of Morse index $0$, which are local minima of $U$. 
Because of the vanishing eigenvalue, they are neutrally stable in the $n$-dimensional phase 
space, but asymptotically stable for the dynamics restricted to the hyperplane $\Sigma$.
For that reason, we are going to call them \emph{sinks}. The second type of important critical 
points are those having Morse index $1$. We will call them \emph{$1$-saddles} for brevity. 
The important feature of the potential landscape in relation to understanding stochastic dynamics is how sinks are connected by minimal paths, in particular, 
which $1$-saddles lie on these paths. 

We restrict to the nearest-neighbor coupling case $S=\{-1,1\}$, for which we can present a more complete picture.

%%%%%%%%%%%%%%%%%%%%%%%%%%%%%%%%%%%%%%%%%%%%%%%%%%%%%%%%%%%%%%%

\subsection{Classification of equilibria}
\label{ssec:classification} 

An equilibrium of~\eqref{eq:Kuramoto} has to satisfy the equations 
\begin{equation}
\label{eq:equil} 
 \sin\bigpar{2\pi(u_{i+1} - u_i)}
 = \sin\bigpar{2\pi(u_i - u_{i-1})}\;, 
 \qquad 
 i \in\Lambda\;.
\end{equation} 
Let us write 
\begin{equation}
 a_i = (u_{i+1} - u_i) \pmod{1}\;.
\end{equation} 
An equilibrium point is uniquely determined by the tuple $(a_0,\dots,a_{n-1})$. 
Note that due to the periodic boundary conditions, the sum $\omega = a_0 + \dots + a_{n-1}$ 
is necessarily an integer, which must belong to $\set{0,1,\dots,n-1}$. This integer 
can be interpreted as a \emph{winding number}. For instance, the winding number of a 
$q$-twisted state is $q \pmod{n}$. 

Relation~\eqref{eq:equil} implies that for any $i\in\Lambda$, there are two options:
\begin{itemize}
\item   either $a_{i+1} = a_i$, 
\item   or $a_{i+1} = \frac12 - a_i \pmod{1}$.  
\end{itemize}

As a consequence, one of the following two cases holds: 

\begin{enumerate}
\item   All $a_i$ have the same value $a$. Then the winding number 
is $\omega = na$. This corresponds to a $q$-twisted state, with $\omega = na = q\pmod{n}$. 
For future reference, we set $p = n$. 

\item   The $a_i$ take two different values $a$ and $\hat a$, related by 
$\hat a = \frac12 - a \pmod{1}$. Notice that $\cos(2\pi a)$ and $\cos(2\pi \hat a)$ 
have opposite sign. If $a\notin\set{\frac14,\frac34}$, we may assume that $\cos(2\pi a) > 0$, 
while $\cos(2\pi \hat a) < 0$.
Then we set
\begin{equation}
 p = \#\set{i\colon a_i = a}
 = \#\set{i\colon\cos(2\pi a_i) > 0}\;.
\end{equation} 
Since $a \neq \hat a$, we necessarily have $p\in\set{1,\dots,n-1}$, and it has to satisfy 
the condition 
\begin{equation}
\label{eq:condition_omega} 
 pa + (n-p)\hat a = \omega\;.
\end{equation} 
Note that $\hat a = \frac12 - a$ if $a < \frac12$, and $\hat a = \frac32 - a$ otherwise. 
\end{enumerate}

In both cases, since we assume $S = \set{-1,1}$, the potential can be written
\begin{equation}
\label{eq:pot_nn} 
 U(u) = -\frac{K}{2\pi} \sum_{i\in\Lambda} \cos(2\pi(u_{i+1}-u_i))
 = -\frac{K}{2\pi} \sum_{i\in\Lambda} \cos(2\pi a_i)\;.
\end{equation} 
Note that 
\begin{equation}
 \cos\bigpar{2\pi a_i} 
 = \sigma_i \cos(2\pi a) 
 \qquad 
 \text{where } 
 \sigma_i = 
 \begin{cases}
  1 & \text{if $a_i = a$\;,}\\
  -1 & \text{if $a_i = \hat a$\;.}
 \end{cases}
\end{equation} 
Therefore, we have 
\begin{equation}
\label{eq:pot_nn2} 
 U(u) = -\frac{K}{2\pi} \cos(2\pi a) \sum_{i\in\Lambda} \sigma_i
 = -\frac{K}{2\pi}  (2p - n)\cos(2\pi a)\;.
\end{equation} 
The stability of an equilibrium $u$ is determined by the Hessian matrix of the potential 
at $u$, which can be written 
\begin{equation}
\label{eq:Hessian} 
\frac{\partial^2U}{\partial u^2}
= 2\pi K \cos\bigpar{2\pi a} M(\sigma)\;,
\end{equation} 
where 
\begin{equation}
M(\sigma) = 
\begin{pmatrix}
\sigma_{n-1} + \sigma_0 & -\sigma_0 & 0 & \dots & \dots & 0 & -\sigma_{n-1} \\
-\sigma_0 & \sigma_0 + \sigma_1 & -\sigma_1 & \ddots &  &  & 0 \\
0 & -\sigma_1 & \sigma_1 + \sigma_2 & \ddots & \ddots &  & \vdots \\  
\vdots & \ddots & \ddots & \ddots & \ddots & \ddots & \vdots \\
\vdots & & \ddots & \ddots & \ddots & -\sigma_{n-3} & 0 \\
0 &  &  & \ddots & -\sigma_{n-3} & \sigma_{n-3} + \sigma_{n-2} & -\sigma_{n-2} \\
-\sigma_{n-1} & 0 & \dots & \dots & 0 & -\sigma_{n-2} & \sigma_{n-2} + \sigma_{n-1}
\end{pmatrix}\;.
\end{equation} 

%%%%%%%%%%%%%%%%%%%%%%%%%%%%%%%%%%%%%%%%%%%%%%%%%%%%%%%%%%%%%%%

\subsection{The sinks}

For a $q$-twisted state $u^{(q)}$, all $\sigma_i$ are equal to $1$, and therefore $p = n$. 
The value~\eqref{eq:pot_nn2} of the potential reduces to 
\begin{equation}
\label{eq:potential_sink} 
 U(u^{(q)}) = -\frac{K}{2\pi} n \cos\biggpar{\frac{2\pi q}{n}}\;.
\end{equation} 

\begin{lemma} 
\label{lem:sinks} 
A $q$-twisted state is stable provided $\abs{q} < \frac n4$.
That is, the Hessian of $U$ at $u^{(q)}$ has zero as a simple eigenvalue, while all other
eigenvalues are strictly positive. 
\end{lemma}
\begin{proof}
It follows directly from~\eqref{eq:Hessian} that 
\begin{equation}
 \frac{\partial^2U}{\partial u^2}
= - 2\pi K \cos\bigpar{2\pi a} L\;,
\end{equation} 
where
\begin{equation}
\label{eq:L} 
 L = 
 \begin{pmatrix}
  -2 & 1 & 0 & \dots & \dots & 0 & 1 \\
  1 & -2 & 1 & 0 & \dots & \dots & 0 \\
  0 & 1 & -2 & 1 & 0 & \dots & 0 \\
  \vdots & \ddots & \ddots & \ddots & \ddots & \ddots & \vdots \\
  0 & \dots & 0 & 1 & -2 & 1 & 0 \\
  0 & \dots & \dots & 0 & 1 & -2 & 1 \\
  1 & 0 & \dots & \dots & 0 & 1 & -2
 \end{pmatrix} 
\end{equation} 
is the matrix of the discrete Laplacian with periodic boundary conditions. 
This is a circulant matrix, whose eigenvalues can be computed by discrete Fourier 
transform. 
{Namely, for $k\in\Z/n\Z$ let $v_k$ be the vector with components $v_{k,j} = 
\e^{2\pi\icx kj/n}$. Then 
\begin{align}
 \bigpar{Lv_k}_j 
 &= \e^{2\pi\icx k(j-1)/n} -2\e^{2\pi\icx kj/n}  + \e^{2\pi\icx k(j+1)/n} \\
 &= \bigpar{\e^{-2\pi\icx k/n} - 2 + \e^{2\pi\icx k/n}} \e^{2\pi\icx kj/n} \\
 &= 2 \biggpar{\cos\biggpar{\frac{2\pi k}{n}}-1} v_{k,j}
 = -4\sin^2\biggpar{\frac{\pi k}{n}}v_{k,j}\;.
\end{align} 
See also~\cite[Theorem~3.4]{MedTan15b} for a similar computation in the more 
general setting where the coupling goes beyond nearest neighbors. 
The eigenvalues of $L$ are thus given by 
}
\begin{equation}
\label{eq:ev_lambdak0} 
 -\lambda_k^0 = -4\sin^2 \biggpar{\frac{\pi k}{n}}\;, \qquad 
 k = 0, \dots, n-1\;.
\end{equation} 
The eigenvalues of the Hessian are thus given by 
\begin{equation}
\label{eq:lambda_sink} 
 -\lambda_k = 
 -8\pi K \cos\biggpar{\frac{2\pi q}{n}} 
 \sin^2 \biggpar{\frac{\pi k}{n}}\;.
\end{equation} 
If $\abs{q} < n/4$, all $\lambda_k$ except $\lambda_0 = 0$ are strictly positive. 
\end{proof}

The number of stable $q$-twisted states is equal to 
\begin{equation}
\label{eq:N0} 
 N_0 = 1 + 2 \max\Bigset{i\in\N_0\colon i < \frac n4}
 = 2\Bigl\lceil \frac n4 \Bigr\rceil - 1\;.
\end{equation} 

\begin{remark}
The proof shows that the case $\abs{q} = \frac n4$ is degenerate, since the Hessian 
matrix is the zero matrix, while in the case $\abs{q} > \frac n4$, 
the equilibrium is a local maximum of the potential, since all nonzero eigenvalues 
of the Hessian matrix are strictly negative. 
\end{remark}

We still have to examine the effect of the symmetries ${C_j}$ and $I$, which are not taken 
into account by the fundamental domain $\cD$. 

\begin{lemma}
For each $q\in\Z/n\Z$, there is exactly one $q$-twisted state in the fundamental domain $\cD$.
\end{lemma}
\begin{proof}
A representative of $u^{(q)}$ in the hyperplane $\Sigma$ is given by 
\begin{equation}
u^{(q)} = 
 \frac1{2n} \bigpar{-(n-1)q, -(n-3)q, \dots, (n-3)q, (n-1)q}^\top\;.
\end{equation} 
Note that this is invariant under the inversion $I$. Furthermore, since 
the sum of the first $n-1$ components is $-\frac{(n-1)q}{2n}$, it follows from~\eqref{eq:u_to_y}
that the corresponding $y$-coordinates are 
\begin{align}
 y^{(q)} &= \frac1n \bigpar{-(n-1)q, -(n-2)q, \dots, -q}^\top \\
 &= \frac1n \bigpar{q, 2q, \dots, (n-1)q}^\top 
 + (q,q,\dots,q)^\top\;.
\end{align} 
Consider now the cyclic permutation 
\begin{equation}
 C_1u^{(q)} = 
 \frac1{2n} \bigpar{-(n-3)q, -(n-5)q, \dots, (n-3)q, (n-1)q, -(n-1)q}^\top\;.
\end{equation} 
The sum of its $n-1$ first components is $\frac{n-1}{2n}q$. Using~\eqref{eq:u_to_y}, we obtain that its $y$-coordinates are 
\begin{equation}
 C_1y^{(q)} = \frac1n \bigpar{q, 2q, \dots, (n-1)q} ^\top
 \sim y^{(q)}\;.
\end{equation}
Since $y^{(q)}$ and $C_1y^{(q)}$ are related by an integer translation, they 
correspond to the same point in the fundamental domain $\cD$. Cyclic permutations and the inversion 
being the only candidates for possible other $q$-twisted states in $\cD$, the claim follows. 
\end{proof}

%%%%%%%%%%%%%%%%%%%%%%%%%%%%%%%%%%%%%%%%%%%%%%%%%%%%%%%%%%%%%%%%%%%%%%%%%%%%%%

\subsection{The saddles} 
\label{ssec:saddles}

Consider now the case where $p = n-1$ of the $\sigma_i$ are equal to $1$, 
say $\sigma = (1, \dots, 1, -1)$. These states have a \lq\lq jump\rq\rq\ 
between $i=n-1$ and $i=0$. Other equilibria with the same Morse index can 
be obtained by cyclic permutation of the coordinates, which amounts to 
moving the location of the jump. 

\begin{lemma}
If $n\geqs 5$, equilibria with $\sigma = (1, \dots, 1, -1)$ are of the form 
\begin{equation}
\label{eq:qhat} 
 u_i^{(r)} 
 = \frac{\hat q i}{n}\;,
 \quad i\in\Lambda\;, 
 \qquad \hat q = \frac{rn}{n-2}\;, 
\end{equation} 
where $r \in \Z+\frac12$ is any half integer satisfying 
\begin{equation}
 -\frac n4 + \frac12 < r < \frac n4 - \frac12\;.
\end{equation} 
\end{lemma}
\begin{proof}
According to the discussion of Section~\ref{ssec:classification}, there are two cases 
to consider, depending on the value of $a$. Recall that $a\in[0,1)$ satisfies $\cos(2\pi a) > 0$.

\begin{itemize}
\item   The first case occurs if $0\leqs a < \frac14$. Then $\hat a = \frac12 - a$, and 
Condition~\eqref{eq:condition_omega} with $p = n-1$ becomes 
\begin{equation}
 (n-2)a = \omega - \frac12
 \quad \Rightarrow \quad 
 a = \frac{2\omega - 1}{2(n - 2)}\;.
\end{equation} 
Setting $r = \omega - \frac12$ yields the expression~\eqref{eq:qhat} for the equilibrium. 
The condition $0\leqs a < \frac14$ becomes 
\begin{equation}
 \frac12 \leqs r < \frac n4 - \frac12\;.
\end{equation} 

\item   The second case occurs if $\frac34 < a < 1$. Then $\hat a  = \frac32 - a$, and 
Condition~\eqref{eq:condition_omega} with $p = n-1$ becomes 
\begin{equation}
 (n-2)a = \omega - \frac32
 \quad \Rightarrow \quad 
 a = \frac{2\omega - 3}{2(n - 2)}\;.
\end{equation} 
Here it is more convenient to view the state as having a \lq\lq negative slope\rq\rq. 
This amounts to replacing $a$ by $a-1$, and setting $r = \omega - n + \frac12$ with 
\begin{equation}
 -\frac n4 + \frac12 < r \leqs -\frac12\;,
\end{equation} 
which yields again the expression~\eqref{eq:qhat} for the equilibrium. 
\qed
\end{itemize}
\renewcommand{\qed}{}
\end{proof}

\begin{remark}
In the case $n=3$, there is no admissible equilibrium with $\sigma=(1,1,-1)$. However, the 
r\^ole of $1$-saddle is played by the equilibrium with $\sigma=(1,-1,-1)$, which is also of the form~\eqref{eq:qhat} with $r=\frac12$ and $\hat q=\frac32$. It corresponds to $a=0$ and $\hat a = \frac12$. That state is a $1$-saddle because a direct computation shows that the eigenvalues 
of $M(\sigma)$ are $-1$, $0$ and $3$. 
The case $n=4$ is degenerate, as for the sinks, because then the Hessian matrix is identically zero.
\end{remark}

The value of the potential at the equilibria $u^{(r)}$ is given, according to~\eqref{eq:pot_nn2}, by 
\begin{equation}
\label{eq:potential_saddle} 
 U(u^{(r)})
 = -\frac{K}{2\pi} (n-2)\cos\biggpar{2\pi\frac{r}{n-2}}\;.
\end{equation} 

\begin{proposition}
\label{thm:saddle}
For $-\frac n4 + \frac12 < r < \frac{n}{4} - \frac12$, $u^{(r)}$ is a $1$-saddle of \eqref{eq:Kuramoto}, i.e., the Hessian  $\frac{\partial^2U}{\partial u^2}\bigpar{u^{(r)}}$ 
has one negative eigenvalue, the zero eigenvalue, and $n-2$ positive eigenvalues.
\end{proposition}

To prove the proposition, we observe that~\eqref{eq:Hessian} implies that 
the Hessian matrix of $U$ at these states is given by 
\begin{equation}
\label{eq:Hessian_saddle} 
 \frac{\partial^2U}{\partial u^2}\bigpar{u^{(r)}} 
 = -2\pi K \cos\biggpar{\frac{2\pi\hat q}{n}} M\;, \qquad 
 M = D + N\;, 
\end{equation} 
where $D$ is the matrix of the discrete Laplacian with Neumann boundary 
conditions 
\begin{equation}
 D = 
 \begin{pmatrix}
  -1 & 1 & 0 & \dots & \dots & \dots & 0 \\
  1 & -2 & 1 & 0 & \dots & \dots & 0 \\
  0 & 1 & -2 & 1 & 0 & \dots & 0 \\
  \vdots & \ddots & \ddots & \ddots & \ddots & \ddots & \vdots \\
  0 & \dots & 0 & 1 & -2 & 1 & 0 \\
  0 & \dots & \dots & 0 & 1 & -2 & 1 \\
  0 & \dots & \dots & \dots & 0 & 1 & -1
 \end{pmatrix}\;,
\end{equation} 
and $N$ is the rank $1$ matrix 
\begin{equation}
 N = 
 \begin{pmatrix}
  1 & 0 & \dots & 0 & -1 \\
  0 & 0 & \dots & 0 & 0 \\
  \vdots & & & & \vdots \\
  0 & 0 & \dots & 0 & 0 \\
  -1 & 0 & \dots & 0 & 1 
 \end{pmatrix}
 = \psi\psi^\top\;, 
 \qquad 
 \psi = 
 \begin{pmatrix}
 1 \\ 0 \\ \vdots \\ 0 \\ -1
 \end{pmatrix}\;.
\end{equation} 
The eigenvalues of $-D$ are known to have the form 
\begin{equation}
 \nu_k^0 = 4\sin^2 \biggpar{\frac{\pi k}{2n}}\;, 
 \qquad k = 0, \dots, n-1\;.
\end{equation} 
The corresponding eigenvectors have components 
\begin{equation}
 \ph_k^0(i) 
 = \sqrt{\frac2n} \cos\biggpar{\frac{\pi k (i+\frac12)}{n}}\;, 
 \qquad 
 i = 0, \dots, n-1
\end{equation} 
for $k \neq 0$, while $\ph_0^0$ is a constant vector. One easily checks that 
\begin{equation}
 \pscal{\psi}{\ph_k^0} 
 = 
 \begin{cases}
  0 & \text{if $k$ is even\;,} \\
  \displaystyle 
  2 \sqrt{\frac2n} \cos\biggpar{\frac{\pi k }{2n}}
  & \text{if $k$ is odd\;.}
 \end{cases}
\end{equation} 

The following result is well known in the theory of perturbations by rank $1$ 
linear operators. 

\begin{lemma}
\label{lem:rank_one_ev} 
Denote the eigenvalues of $-M$ by $\nu_k$, $k=0, \dots, n-1$. Then for even 
$k$, we have $\nu_k = \nu^0_k$. The odd-numbered eigenvalues of $-M$ satisfy 
the equation 
\begin{equation}
\label{eq:Fnu} 
 F(\nu) := 
 \sum_{k \text{ odd}} \frac{\pscal{\psi}{\ph_k^0}^2}{\nu_k^0 - \nu} = 1\;.
\end{equation}
\end{lemma}
\begin{proof}
If $k$ is even, since $\pscal{\psi}{\ph_k^0} = 0$, $\ph_k^0$ is an 
eigenvector of $-M$ for the same eigenvalue. If $\nu$ is an eigenvalue of 
$-M$ which is not an eigenvalue of $-D$, then the eigenvalue equation can be 
written  
\begin{equation}
\label{eq:ev_ph} 
 \bigpar{-D-\nu} \ph - \pscal{\psi}{\ph}\psi = 0\;,
\end{equation} 
so that 
\begin{equation}
 \ph = \pscal{\psi}{\ph} \bigpar{-D-\nu}^{-1}\psi\;.
\end{equation} 
It follows that 
\begin{equation}
 \pscal{\psi}{\ph} 
 = \pscal{\psi}{\ph} \pscal{\psi}{\bigpar{-D-\nu}^{-1}\psi}\;.
\end{equation}
We must have $\pscal{\psi}{\ph} \neq 0$, since otherwise, \eqref{eq:ev_ph} 
would imply that $\nu$ is an eigenvalue of $-D$. (Note also that $\psi$ lies in 
the subspace orthogonal to the span of the $\ph_k^0$ with even $k$, which is 
invariant under $D$.) Therefore, we obtain 
\begin{align}
1 &= \pscal{\psi}{\bigpar{-D-\nu}^{-1}\psi} \\
&= \sum_{k \text{ odd}} \pscal{\psi}{\ph_k^0} 
\pscal{\ph_k^0}{\bigpar{-D-\nu}^{-1}\psi} \\
&= \sum_{k \text{ odd}} \pscal{\psi}{\ph_k^0} 
\bigpar{\nu_k^0-\nu}^{-1}\pscal{\ph_k^0}{\psi} \\
&= \sum_{k \text{ odd}} \frac{\pscal{\psi}{\ph_k^0}^2}{\nu_k^0-\nu}\;, 
\end{align}
from which the claim follows. 
\end{proof}

\begin{proof}[Proof of Proposition~\ref{thm:saddle}]
The function $F$ has poles on the eigenvalues $\nu_k^0$ of $-D$ ($k$ odd), and 
is otherwise strictly increasing. Furthermore, it converges to $0$ as 
$\nu\to\pm\infty$. It follows that~\eqref{eq:Fnu} has exactly one solution in 
each interval $(\nu_{2\ell+1}^0, \nu_{2\ell+3}^0)$ (this is called 
interlacing), and an additional solution that is smaller than $\nu_1^0$. 

Note that
\begin{align}
 \psi^\top N \psi &= 2\psi^\top\psi = 4\;, \\
 \psi^\top D \psi &= -2\;,
\end{align}
showing that $M$ has at least one strictly positive eigenvalue. By the 
interlacing result, there is exactly one such eigenvalue, showing that the 
smallest eigenvalue $\nu_1$ of $-M$ is strictly negative. The state $u^{(r)}$ 
is thus indeed a saddle of index $1$. To summarize, the 
odd-numbered eigenvalues satisfy
\begin{equation}
 \nu_1 < 0 = \nu_0^0 < \nu_1^0 < \nu_3 < \nu_3^0 < \nu_5 < \nu_5^0 < \dots\;.
\end{equation} 
This shows in particular that the $u^{(r)}$ are indeed $1$-saddles.
\end{proof}

One can check that the condition $-\frac n4 + \frac12 < r < \frac n4 - \frac12$ implies that the 
number $N_1$ of $1$-saddles $u^{(r)}$, having a jump between $i=n-1$ and $i=0$,
is given by 
\begin{equation}
 N_1 = N_0 - 1\;,
\end{equation} 
where $N_0$ is the number of stable $q$-twisted states, see~\eqref{eq:N0}.

As before, we also examine the effect of the symmetries ${C_j}$ and $I$. 

\begin{lemma}
If $n\neq4$, then the representatives of $u^{(r)}$, $C_1u^{(r)}$, \dots, 
$C_{n-1}u^{(r)}$ in the fundamental domain are all different. Therefore, 
there are exactly $n$ $1$-saddles for each admissible $r$ in the fundamental 
domain.
\end{lemma}
\begin{proof}
A representative of $u^{(r)}$ in the hyperplane $\Sigma$ is given by 
\begin{equation}
u^{(r)} = 
 \frac1{2n} \bigpar{-(n-1)\hat q, -(n-3)\hat q, \dots, (n-3)\hat q, (n-1)\hat q}^\top\;.
\end{equation} 
Note that this is invariant under the inversion $I$. 
The sum of the first $n-1$ components being $-\frac{(n-1)\hat q}{2n}$,
the corresponding $y$-coordinates are 
\begin{equation}
 y^{(r)} = \frac1n \bigpar{-(n-1)\hat q, -(n-2)\hat q, \dots, -\hat q}^\top\;.
\end{equation} 
Applying the cyclic permutation, we find 
\begin{equation}
C_1 u^{(r)} = 
 \frac1{2n} \bigpar{-(n-3)\hat q, \dots, (n-3)\hat q, (n-1)\hat q, -(n-1)\hat q}^\top\;.
\end{equation} 
The sum of the first $n-1$ components is now $-\frac{(n-1)\hat q}{2n}$, 
and the corresponding $y$-coordinates are 
\begin{equation}
 C_1y^{(r)} = \frac1n \bigpar{\hat q, 2\hat q, \dots, (n-1)\hat q}^\top\;.
\end{equation} 
In a similar manner, we find 
\begin{align}
C_2y^{(r)} 
&= \frac1n \bigpar{\hat q, 2\hat q, \dots, (n-2)\hat q, -\hat q}^\top\;, \\
C_3y^{(r)} 
&= \frac1n \bigpar{\hat q, 2\hat q, \dots, (n-3)\hat q, -2\hat q, -\hat q}^\top\;, \\
&\dots  \\
C_{n-1}y^{(r)} 
&= \frac1n \bigpar{\hat q, -(n-2)\hat q, \dots, -\hat q}^\top\;.
\end{align}
We observe that 
\begin{equation}
{C_j}y^{(r)} 
= y^{(r)} + \bigpar{\underbrace{\hat q, \dots, \hat q}_{{n-j}\text{ times}}, 
\underbrace{0, \dots, 0}_{{j-1}\text{ times}}}^\top\;.
\end{equation} 
We conclude that all $n$ saddles obtained by cyclic permutation of $u^{(r)}$ 
have different representatives in the fundamental domain, unless $\hat q$ is an 
integer. Since 
\begin{equation}
 \hat q = r \biggpar{1 + \frac{2}{n-2}}\;,
\end{equation} 
this is the case if and only if $\frac{2}{n-2}$ is an odd integer, which is the case
if and only if $n=4$. 
\end{proof}

\begin{remark}
As remarked before, the case $n=4$ is degenerate. We already know that the Hessian matrix 
at the $\pm1$-twisted states is identically zero in this case. The state 
$u^{(1/2)}$ is actually identical with $u^{(1)}$, because in the notations of 
Section~\ref{ssec:classification}, it corresponds to $a = \hat a = \frac14$. 
It may be the case that there are additional constants of motion in this 
situation. 
\end{remark}

{All results given in what follows are restricted to the case $n\geqs5$. Note that the particular 
case $n=3$ has been discussed in Example~\ref{ex:n3} above. The case $n=4$ would require a more 
detailed normal form analysis around critical points.}

\begin{table}
\begin{center}
 \begin{tabular}{|c|c|c|}
\hline
\vrule height 12pt depth 0pt width 0pt
Equilibrium & $(u_0,u_1)$ & $(y_0,y_1)$ \\
\hline 
\vrule height 13pt depth 0pt width 0pt
$u^{(0)}$ & $(0,0)$ & $(0,0)$ \\
\vrule height 13pt depth 0pt width 0pt
$u^{(1)}$ & $(\frac13,-\frac13)$ & $(\frac13,-\frac13)$ \\
\vrule height 13pt depth 0pt width 0pt
$u^{(-1)}$ & $(-\frac13,\frac13)$ & $(-\frac13,\frac13)$ \\
\vrule height 13pt depth 0pt width 0pt
$u^{(1/2)}$ & $(\frac16,-\frac13)$ & $(0,-\frac12)$ \\
\vrule height 13pt depth 0pt width 0pt
$C_1 u^{(1/2)}$ & $(-\frac13,\frac16)$ & $(-\frac12,0)$ \\
\vrule height 13pt depth 8pt width 0pt
$C_2 u^{(1/2)}$ & $(-\frac16,-\frac16)$ & $(-\frac12,-\frac12)$ \\
\hline
\end{tabular}
\end{center}
\caption[]{List of all equilibria in the fundamental domain $\cD$ in the case $n=3$, 
with their $u$ and $y$-coordinates.}
\label{table:n3} 
\end{table}

\begin{figure}[!t]
\begin{center}
\scalebox{0.75}{
\begin{tikzpicture}[>=stealth',main node/.style={circle,inner sep=0.05cm,fill=white,draw},x=7cm,y=7cm]

% axes and graduations

\draw[->,thick] (-0.6,0) -> (0.6,0);
\draw[->,thick] (0,-0.6) -> (0,0.65);

\draw[semithick] (-0.01,1/2) -- (0.01,1/2);
\draw[semithick] (1/2,-0.01) -- (1/2,0.01);

% stable/unstable manifolds of saddles 
\draw[semithick,violet] (1/6,1/6) -- (-1/2,1/2) -- (-1/6,-1/6) -- (1/2,-1/2) -- cycle;
\draw[semithick,violet] (-1/6,1/3) -- (-1/3,1/6);
\draw[semithick,violet] (1/6,-1/3) -- (1/3,-1/6);

% sinks 
\node[main node,blue] at (0,0) {};

% maxima
\node[main node,red] at (1/3,-1/3) {};
\node[main node,red] at (-1/3,1/3) {};

% 1-saddles
\node[main node,semithick,violet,fill=white] at (1/6,1/6) {};
\node[main node,semithick,violet,fill=white] at (-1/6,1/3) {};
\node[main node,semithick,violet,fill=white] at (-1/2,1/2) {};
\node[main node,semithick,violet,fill=white] at (-1/3,1/6) {};
\node[main node,semithick,violet,fill=white] at (-1/6,-1/6) {};
\node[main node,semithick,violet,fill=white] at (1/6,-1/3) {};
\node[main node,semithick,violet,fill=white] at (1/2,-1/2) {};
\node[main node,semithick,violet,fill=white] at (1/3,-1/6) {};

\node[blue] at (-0.05,-0.04) {$u^{(0)}$};
\node[red] at ({-1/3 + 0.02},{1/3-0.05}) {$u^{(-1)}$};
\node[red] at ({1/3-0.05},{-1/3 + 0.02}) {$u^{(1)}$};
\node[violet] at ({-1/6+0.05},{1/3+0.05}) {$u^{(1/2)}$};
\node[violet] at ({1/6-0.05},{-1/3-0.05}) {$u^{(1/2)}$};
\node[violet] at ({-1/3-0.09},{1/6-0.05}) {$C_1u^{(1/2)}$};
\node[violet] at ({1/3+0.09},{-1/6+0.05}) {$C_1u^{(1/2)}$};
\node[violet] at ({1/6+0.07},{1/6+0.05}) {$C_2u^{(1/2)}$};
\node[violet] at ({-1/6-0.07},{-1/6-0.05}) {$C_2u^{(1/2)}$};
\node[violet] at ({1/2+0.07},{-1/2-0.05}) {$C_2u^{(1/2)}$};
\node[violet] at ({-1/2-0.03},{1/2+0.05}) {$C_2u^{(1/2)}$};

\node[] at (1/2,0.06) {$\frac12$};
\node[] at (0.04,1/2) {$\frac12$};
\node[] at (0.55,-0.03) {$u_0$};
\node[] at (-0.04,0.6) {$u_1$};
\end{tikzpicture}
}
\scalebox{0.75}{
\begin{tikzpicture}[>=stealth',main node/.style={circle,inner sep=0.05cm,fill=white,draw},x=7cm,y=7cm]

% axes and graduations

\draw[->,thick] (-0.6,0) -> (0.6,0);
\draw[->,thick] (0,-0.6) -> (0,0.65);

\draw[semithick] (-0.01,1/2) -- (0.01,1/2);
\draw[semithick] (1/2,-0.01) -- (1/2,0.01);

% stable/unstable manifolds of saddles 
\draw[semithick,violet] (1/2,1/2) -- (-1/2,1/2) -- (-1/2,-1/2) -- (1/2,-1/2) -- cycle;
\draw[semithick,violet] (0,1/2) -- (-1/2,0);
\draw[semithick,violet] (1/2,0) -- (0,-1/2);

% sinks 
\node[main node,blue] at (0,0) {};

% maxima
\node[main node,red] at (1/3,-1/3) {};
\node[main node,red] at (-1/3,1/3) {};

% 1-saddles
\node[main node,semithick,violet,fill=white] at (1/2,1/2) {};
\node[main node,semithick,violet,fill=white] at (-1/2,0) {};
\node[main node,semithick,violet,fill=white] at (1/2,0) {};
\node[main node,semithick,violet,fill=white] at (-1/2,-1/2) {};
\node[main node,semithick,violet,fill=white] at (0,-1/2) {};
\node[main node,semithick,violet,fill=white] at (0,1/2) {};
\node[main node,semithick,violet,fill=white] at (1/2,-1/2) {};
\node[main node,semithick,violet,fill=white] at (-1/2,1/2) {};

\node[blue] at (-0.05,-0.04) {$u^{(0)}$};
\node[red] at ({-1/3+0.02},{1/3+0.05}) {$u^{(-1)}$};
\node[red] at ({1/3+0.07},{-1/3-0.02}) {$u^{(1)}$};
\node[violet] at ({+0.09},{1/2+0.05}) {$u^{(1/2)}$};
\node[violet] at ({-0.07},{-1/2-0.05}) {$u^{(1/2)}$};
\node[violet] at ({-1/2-0.11},{-0.05}) {$C_1u^{(1/2)}$};
\node[violet] at ({1/2+0.11},{+0.05}) {$C_1u^{(1/2)}$};
\node[violet] at ({1/2+0.07},{1/2+0.05}) {$C_2u^{(1/2)}$};
\node[violet] at ({-1/2-0.07},{-1/2-0.05}) {$C_2u^{(1/2)}$};
\node[violet] at ({1/2+0.07},{-1/2-0.05}) {$C_2u^{(1/2)}$};
\node[violet] at ({-1/2-0.07},{1/2+0.05}) {$C_2u^{(1/2)}$};

% \node[] at (1/2,-0.06) {$\frac12$};
% \node[] at (0.04,1/2) {$\frac12$};
\node[] at (0.55,-0.03) {$y_0$};
\node[] at (-0.04,0.6) {$y_1$};
\end{tikzpicture}
}
% \vspace{-3mm}
\end{center}
\caption{Equilibria in the fundamental domain, in the case $n=3$. 
They are shown in $u$ and $y$-coordinates, with 
sinks in \textcolor{blue}{blue}, 
local maxima in \textcolor{red}{red}, 
and $1$-saddles in \textcolor{violet}{violet}.
Points on opposite sides and corners of the fundamental domain are identified. 
The \textcolor{violet}{violet} lines indicate the 
stable/unstable manifolds of the $1$-saddles (it follows from~\eqref{eq:U_n3y} that the 
potential is constant on these lines), and part of them are domain boundaries.}
\label{fig:fundamental_domain}
\end{figure}
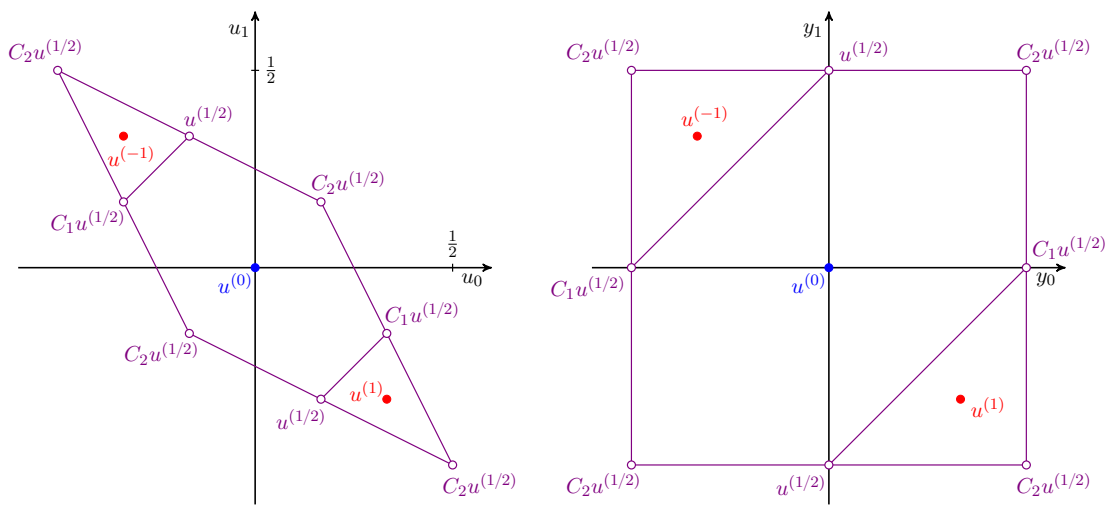

\begin{example}
Returning to the case $n=3$, the list of all equilibria is given in Table~\ref{table:n3}.
They are also shown in Figure~\ref{fig:fundamental_domain}, to be compared 
with the contour plot of the potential in Figure~\ref{fig:potential_3d}. 
Note that the $0$-twisted state $u^{(0)}$ communicates with six copies 
of itself, via six $1$-saddles, which are given by two copies each for 
the three states $u^{(1/2)}$, $C_1 u^{(1/2)}$ and $C_2 u^{(1/2)}$.

In this situation, if the dynamics is described in the fundamental domain, metastable transitions correspond to the system leaving a neighborhood of $u^{(0)}$, and returning to it after passing 
near one of the $1$-saddles. By contrast, if the system is viewed as evolving in the whole 
hyperplane $\Sigma$, the different copies of $u^{(0)}$ are no longer considered to be all 
the same state, and the dynamics resembles a random walk between these copies (see also 
Remark~\ref{rem:Sigma}). 
\end{example}

%%%%%%%%%%%%%%%%%%%%%%%%%%%%%%%%%%%%%%%%%%%%%%%%%%%%%%%%%%%%%%%

\subsection{Other equilibria}

In order to determine the metastable hierarchy, it will be useful to know exactly 
how many equilibria have Morse index $0$ or $1$. Recall that we have introduced 
the integer 
\begin{equation}
 p = \#\set{i\in\Lambda \colon a_i = a}
 = \#\set{i\in\Lambda \colon \sigma_i = 1} 
 \in\set{1,\dots,n}\;.
\end{equation} 
The case $p = n$ corresponds to $q$-twisted states, which have either Morse index $0$ (sinks), 
or Morse index $n$ (local maxima of the potential), or are degenerate, in the special case 
$\abs{q} = \frac n4$. The case $p = n-1$ corresponds to $1$-saddles. 

It thus remains to consider the cases where $1 \leqs p \leqs n-2$. The following result 
shows that these equilibria cannot have Morse index $0$ or $1$, and therefore are not 
important for the metastable dynamics of the system. 

\begin{lemma}
Assume $n\geqs 5$. Then any equilibrium point with $1 \leqs p \leqs n-2$, that is, with at least two $\sigma_i$ equal to $-1$, is a saddle of Morse index $2$ at least. 
\end{lemma}
\begin{proof}
Since $p \leqs n-2$, we may assume that $\sigma_0 = 1$. 
Let $(e_0,\dots,e_{n-1})$ be the canonical basis of $\R^\Lambda$. 
Our aim is to construct vectors $\hat e_1$ and $\hat e_2$, orthogonal to each other, and such that the restriction of $M(\sigma)$ to the subspace spanned by these vectors is negative definite. 

Let $i$ be the smallest $i\geqs1$ such that $\sigma_i = -1$. 
We first consider the case where $\sigma_{i+1} = -1$. 
Then the restriction of 
$M(\sigma)$ to $\vspan(e_{i-1},e_i,e_{i+1},e_{i+2})$ is given by
\begin{equation}
\widehat{M} = 
\begin{pmatrix}
\sigma_{i-2} - 1 & -1 & 0 & 0 \\
-1 & 0 & 1 & 0 \\
0 & 1 & -2 & 1 \\
0 & 0 & 1 & \sigma_{i+2} - 1
\end{pmatrix}\;.
\end{equation} 
Consider the vectors 
\begin{equation}
 \hat e_1 = e_{i-1} + 2e_i\;, \qquad 
 \hat e_2 = -e_{i+1} + e_{i+2}\;.
\end{equation} 
Computing $\pscal{\hat e_p}{\widehat{M}\hat e_q}$ for $p, q\in\set{1,2}$, one obtains 
that the restriction of $M(\sigma)$ to $\vspan(\hat e_1, \hat e_2)$ is given by
\begin{equation}
 \begin{pmatrix}
  \sigma_{i-2} - 3 & -2 \\
  -2 & \sigma_{i+2} - 5
 \end{pmatrix}\;.
\end{equation} 
One easily checks that this matrix is negative definite, for any values of 
$\sigma_{i-2}, \sigma_{i+2} \in\set{-1, 1}$. 

Consider now the case $\sigma_{i+1} = 1$. Let $j \geqs i+2$ be the smallest $j$ such 
that $\sigma_{j-1} = 1$ and $\sigma_j = -1$. Then the restriction of 
$M(\sigma)$ to $\vspan(e_i,e_{i+1},e_j,e_{j+1})$ is given by
\begin{equation}
\widehat{M} = 
\begin{pmatrix}
0 & 1 & 0 & 0 \\
1 & 0 & M_{i+1,j} & 0 \\
0 & M_{i+1,j} & 0 & 1 \\
0 & 0 & 1 & \sigma_{j+1} - 1
\end{pmatrix}\;,
\end{equation} 
where $M_{i+1,j} = -1$ if $j = i+2$, and $0$ otherwise. 
Consider the vectors 
\begin{equation}
 \hat e_1 = e_i - e_{i+1}\;, \qquad 
 \hat e_2 = e_j + e_{j+1}\;.
\end{equation} 
The restriction of $M(\sigma)$ to $\vspan(\hat e_1, \hat e_2)$ is given by
\begin{equation}
 \begin{pmatrix}
  -2 & -M_{i+1,j} \\
  -M_{i+1,j} & \sigma_{j+1} - 3
 \end{pmatrix}\;.
\end{equation} 
One again easily checks that this matrix is negative definite, for any values of 
$\sigma_{j+1}\in\set{-1, 1}$ and $M_{i+1,j} \in\set{-1,0}$. 
\end{proof}

%%%%%%%%%%%%%%%%%%%%%%%%%%%%%%%%%%%%%%%%%%%%%%%%%%%%%%%%%%%%%%%

\subsection{Relevant saddles and metastable hierarchy}

In this subsection, we identify the relevant saddles between the $q$-twisted states, 
as defined in Section~\ref{sec:stoch}.  

\begin{figure}[!t]
\begin{center}
\scalebox{1.0}{
\begin{tikzpicture}[>=stealth',main node/.style={circle,minimum
size=0.25cm,fill=blue!20,draw},x=1.5cm,y=0.5cm, 
declare function={U(\n,\s) = -(\n-1)*cos(360*\s/\n) - cos(360*\s*(1-1/\n));
sq(\q,\r) = (\q+0.5)*\r;
}]

\newcommand*{\nn}{18}
\newcommand*{\mm}{4}
\newcommand*{\rr}{\nn/(\nn-2)}

% axes and graduations

\draw[->,thick] ({-\mm - 0.5},0) -> (\mm + 0.5,0);
\draw[->,thick] (0,-\nn-1) -> (0,2);

\foreach \x in {-\mm,...,\mm}
\draw[semithick] ({\x},-0.2) -- ({\x},0.2);

\foreach \x in {1,...,\mm}
\node[] at (\x,0.75) {{\small $\x$}};

\foreach \x in {1,...,\mm}
\node[] at (-\x-0.1,0.75) {{\small $-\x$}};

\foreach \y in {-\nn,...,-1}
\draw[semithick] (-0.07,{\y}) -- (0.07,{\y});

% plot 

\draw[blue,very thick,-,smooth,domain={-\mm}:{\mm},samples=100,/pgf/fpu,
/pgf/fpu/output format=fixed] plot (\x, {U(\nn,\x)});

% labels of extrema

\newcommand*{\sqa}{0.5*\rr}
\newcommand*{\sqb}{1.5*\rr}
\newcommand*{\sqc}{2.5*\rr}
\newcommand*{\sqd}{3.5*\rr}

\node[] at ({\sqa}, {U(\nn,\sqa) + 0.6}) {$u^{(1/2)}$};
\node[] at ({\sqb}, {U(\nn,\sqb) + 0.6}) {$u^{(3/2)}$};
\node[] at ({\sqc}, {U(\nn,\sqc) + 0.6}) {$u^{(5/2)}$};
\node[] at ({\sqd}, {U(\nn,\sqd) + 0.6}) {$u^{(7/2)}$};

\node[] at ({0.1-\sqa}, {U(\nn,\sqa) + 0.6}) {$u^{(-1/2)}$};
\node[] at ({0.2-\sqb}, {U(\nn,\sqb) + 0.6}) {$u^{(-3/2)}$};
\node[] at ({0.3-\sqc}, {U(\nn,\sqc) + 0.6}) {$ u^{(-5/2)}$};
\node[] at ({0.3-\sqd}, {U(\nn,\sqd) + 0.6}) {$u^{(-7/2)}$};

\foreach \x in {1,...,\mm} \node[] at ({\x + 0.2},{U(\nn,\x) - 0.6}) {$u^{(\x)}$};
\foreach \x in {1,...,\mm} \node[] at ({-\x - 0.1},{U(\nn,\x) - 0.6}) {$u^{(-\x)}$};

\node[] at (0.3, {U(\nn,0) - 0.6}) {$u^{(0)}$};

\node[] at (\mm + 0.3,0.5) {$s$};
% \node[blue] at (0.6,1.2) {$U(u(s))$};
\end{tikzpicture}
}
% \vspace{-3mm}
\end{center}
\caption{The potential $U(u(s))$ for $n = 18$ and $K = 2\pi$. The local minima are 
located at the $q$-twisted states $u(s) = u^{(q)}$, while the local maxima are 
on the $1$-saddles $u(\hat s_q) = u^{(q+1/2)}$.}
\label{f.potential}
\end{figure}
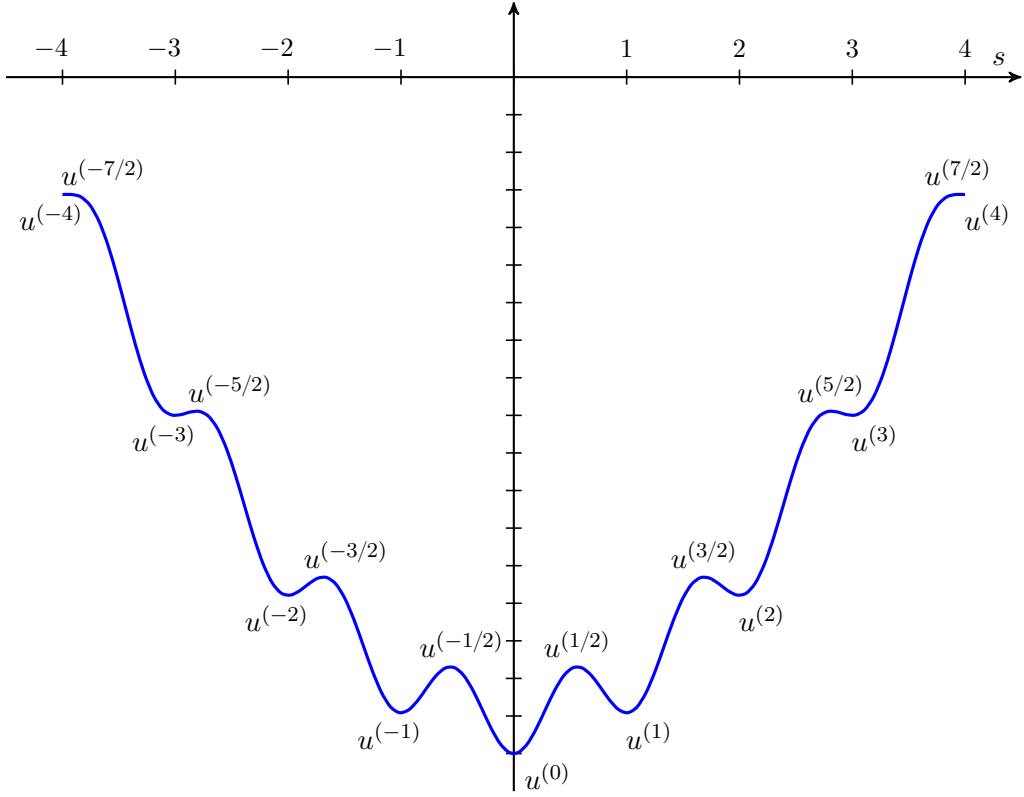

\begin{proposition}
For any $q$ such that $1 \leqs \abs{q} < \frac n4$, 
the set of relevant saddles between the sink $u^{(q)}$ and the set 
$\set{u^{(-\abs{q}+1)}, u^{(-\abs{q}+2)},\dots, u^{(\abs{q}-1)}}$ is 
\begin{equation}
 \bigset{{C_j}u^{(\abs{q}-1/2)}, {C_j} u^{(-\abs{q}+1/2)} \colon 0 \leqs p \leqs n-1}\;,
\end{equation} 
% $\set{C_pu^{(\abs{q}-1/2)}, C_p u^{(-\abs{q}+1/2)} \colon 0 \leqs p \leqs n-1}$, 
where we recall that ${C_j}$ denotes the cyclic permutation defined in~\eqref{eq:Cp}. 
\end{proposition}
\begin{proof}
We want to show that any minimal path from $u^{(q)}$ to $\set{u^{(-\abs{q}+1)}, u^{(-\abs{q}+2)}, 
\dots, u^{(\abs{q}-1)}}$, as defined in Section~\ref{sec:stoch} above, contains at least 
one of the points of the given list as its hightest point. 
Note that~\eqref{eq:potential_saddle} implies that the value of $U$ is the same at 
all points in the list. 

Let $m = \max\set{q\in\N_0\colon q < \frac n4}$, so that the sinks are parametrized by 
$q \in\set{-m,\dots,m}$. We define a curve $u:[-m,m] \to (\R/\Z)^\Lambda$ by 
\begin{equation}
 u_i(s) = \frac{s}{n} i\;, 
 \qquad s\in[-m,m]\;, 
 \qquad i\in\Lambda\;.
\end{equation} 
Note that for any $q\in[-m,m]$, we have 
\begin{equation}
 u(s) = u^{(q)} \quad \Leftrightarrow \quad s = q\;,
\end{equation} 
while 
\begin{equation}
 u(s) = u^{(q+\frac12)} 
 \quad \Leftrightarrow \quad 
 s = \hat s_q := \biggpar{q + \frac12} \frac{n}{n-2}\;. 
\end{equation} 
In other terms, the curve $\set{u(s)}_{s\in[-m,m]}$ visits the equilibria 
$u^{(-m)}$, $u^{(-m+1/2)}$, $u^{(-m+1)}, \dots$ up to $u^{(m)}$ in this order.
By~\eqref{eq:pot_nn}, the potential along this curve is given by 
\begin{equation}
 U(u(s)) = -\frac{K}{2\pi}
 \biggbrak{(n-1)\cos\biggpar{2\pi\frac{s}{n}} + \cos\biggpar{2\pi \frac{n-1}{n}}s}\;.
\end{equation} 
The derivative of this function is 
\begin{align}
 \frac{\6}{\6s} U(u(s))
 &= K \frac{n-1}{n} \biggbrak{\sin\biggpar{2\pi\frac sn}
 + \sin\biggpar{2\pi\frac{n-1}{n}s}} \\
 &= 2K \frac{n-1}{n} \sin(2\pi s)
 \cos\biggpar{2\pi\frac{n-2}{n}s}\;,
\end{align}
where we have used the sum-product formula $\sin\alpha + \sin\beta 
= 2 \sin(\frac{\alpha+\beta}{2})\cos(\frac{\alpha-\beta}{2})$.
It is straightforward to check that this derivative in positive on the intervals $(q,\hat s_q)$,
and negative on the intervals $(\hat s_q, q+1)$.

The proof now follows by induction on $\abs{q}$. By symmetry, we may restrict the discussion to the case $q>0$ and $p=0$. In the base case $q = 1$, the restricted curve $\set{u(1-s)}_{s\in[0,1]}$
provides a path from $u^{(1)}$ to $u^{(0)}$, containing the saddle $u^{(1/2)}$ 
as its highest point. Furthermore, it is known that since all critical points are 
non-degenerate, any minimal path from $u^{(1)}$ to $u^{(0)}$ has to contain a saddle of index $1$. Since $u^{(1/2)}$ and $u^{(-1/2)}$ are the lowest $1$-saddles, the path is minimal, 
and the claim follows for $q = 1$. 

To proceed, we need some topological properties of the potential landscape (see for instance~\cite[Section~2.1]{BG_MPRF10}). Let $\hat u$ be a $1$-saddle, and define its~\emph{closed valley} and~\emph{open valley}, respectively, by  
\begin{align}
 \cC\cV(\hat u) &= \bigset{u \colon \overline V(u,\hat u) \leqs U(\hat u)}\;, \\
 \cO\cV(\hat u) &= \bigset{u\in\cC\cV(\hat u) \colon U(u) < U(\hat u)}\;.
\end{align} 
If the potential is of class $\cC^2$ and $\hat u$ is non-degenerate, then $\cO\cV(\hat u)$ 
has at most two connected components, containing the two parts of the one-dimensional unstable manifold of $\hat u$. Furthermore, $\cC\cV(\hat u)$ is the closure of $\cO\cV(\hat u)$, and is path-connected. Intuitively, if $\cO\cV(\hat u)$ has two connected components, these components 
meet at the saddle point $\hat u$, and only there. 

\begin{figure}[!t]
\begin{center}
%\vspace{-5mm}
\scalebox{0.85}{
\begin{tikzpicture}[>=stealth',main node/.style={circle,blue,inner sep=0.04cm,fill=white,draw},x=1.2cm,y=0.8cm,
declare function={f(\x) = (\x-1)*(\x-5)*(67*\x^4-789*\x^3+2996*\x^2-4044*\x+864)/720;}]

\draw[blue,thick,fill=blue!25] plot[smooth,tension=.9]
  coordinates{(4.9,0) (6.2,-1.1) (6.9,0) (6.2,1.1) (4.9,0)};

\draw[blue,thick,fill=blue!25] plot[smooth,tension=.9]
  coordinates{(-4.9,0) (-6.2,-1.1) (-6.9,0) (-6.2,1.1) (-4.9,0)};
  
\draw[blue,thick,fill=blue!25] plot[smooth,tension=.6]
   coordinates{(4.9,0) (2.5,2.3) (0,3) (-2.5,2.3) (-4.9,0)};
  
\draw[blue,thick,fill=blue!25] plot[smooth,tension=.6]
   coordinates{(4.9,0) (2.5,-2.3) (0,-3) (-2.5,-2.3) (-4.9,0)};

\draw[blue,thick,fill=blue!50] plot[smooth,tension=.9]
  coordinates{(1.4,0) (2.7,-1) (3.4,0) (2.7,1) (1.4,0)};

\draw[blue,thick,fill=blue!50] plot[smooth,tension=.9]
  coordinates{(-1.4,0) (-2.7,-1) (-3.4,0) (-2.7,1) (-1.4,0)};
  
\draw[blue,thick,fill=blue!50] plot[smooth,tension=.9]
  coordinates{(1.4,0) (0,1) (-1.4,0)};

\draw[blue,thick,fill=blue!50] plot[smooth,tension=.9]
  coordinates{(1.4,0) (0,-1) (-1.4,0)};
  
\node[main node,semithick] at (0,0) {};
\node[main node,semithick] at (1.4,0) {};
\node[main node,semithick] at (-1.4,0) {};
\node[main node,semithick] at (2.7,0) {};
\node[main node,semithick] at (-2.7,0) {};

\node[] at (0.3,0.4) {$u^{(0)}$};
\node[] at (2.8,0.4) {$u^{(1)}$};
\node[] at (-2.45,0.4) {$u^{(-1)}$};

\node[] at (1.95,0.1) {$u^{(1/2)}$};
\node[] at (-0.75,0.1) {$u^{(-1/2)}$};

\node[main node,semithick] at (4.9,0) {};
\node[main node,semithick] at (6.3,0) {};
\node[main node,semithick] at (-4.9,0) {};
\node[main node,semithick] at (-6.3,0) {};

\node[] at (6.4,0.4) {$u^{(2)}$};
\node[] at (-6.1,0.4) {$u^{(-2)}$};

\node[] at (5.45,0.1) {$u^{(3/2)}$};
\node[] at (-4.25,0.1) {$u^{(-3/2)}$};

\node[blue] at (0,-1.4) {$\cV^{(0)}_0$};
\node[blue] at (2.5,-1.4) {$\cV^{(0)}_+$};
\node[blue] at (-2.5,-1.4) {$\cV^{(0)}_-$};
\node[blue] at (0,3.4) {$\cV^{(1)}_0$};
\node[blue] at (-6.3,1.5) {$\cV^{(1)}_-$};
\node[blue] at (6.3,1.5) {$\cV^{(1)}_+$};

\end{tikzpicture}
}
\vspace{-4mm}
\end{center}
\caption[]{Construction of the metastable hierarchy.
The set $\cV^{(0)} = \cV^{(0)}_- \cup \cV^{(0)}_0 \cup \cV^{(0)}_+$ is the 
union of the open valleys $\cO\cV(u^{(1/2)})$ and $\cO\cV(u^{(-1/2)})$, which overlap 
in the central component. Similarly, the set $\cV^{(1)}$ is the union of 
the open valleys $\cO\cV(u^{(3/2)})$ and $\cO\cV(u^{(-3/2)})$. 
}
\label{fig:meta_hierarchy2}
\end{figure}
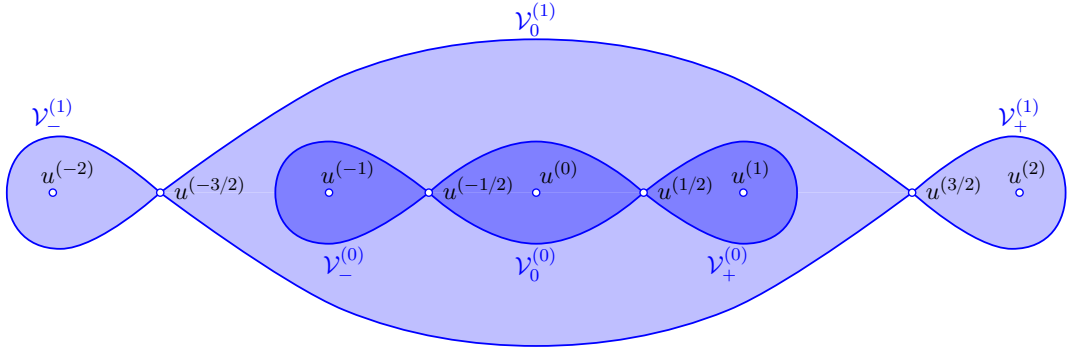

Returning to our particular problem, we set 
\begin{equation}
 \cV^{(0)} = \cO\cV(u^{(1/2)}) \bigcup \cO\cV(u^{(-1/2)})\;.
\end{equation} 
This set has exactly three connected components: the component $\cV^{(0)}_+$ containing $u^{(1)}$, the component $\smash{\cV^{(0)}_-}$ containing $u^{(-1)}$, and the component $\smash{\cV^{(0)}_0}$ containing $u^{(0)}$, which is also the intersection of the open valleys $\cO\cV(u^{(1/2)})$ and $\cO\cV(u^{(-1/2)})$, see Figure~\ref{fig:meta_hierarchy2}. 
We claim that each component contains only one sink. Indeed, if one component contained two different sinks, it would have to contain a $1$-saddle as well. But this is not possible, since all $1$-saddles have height $U(u^{(1/2)})$ at least, while points in the components are lower by construction. 

We now assume by induction that the claim is true for a given $q < \frac n4-1$, and that the set 
\begin{equation}
\label{eq:valley_decomp} 
 \cV^{(q-1)} = \cO\cV(u^{(q-1/2)}) \bigcup \cO\cV(u^{(-q+1/2)})
\end{equation} 
has exactly three connected component. The components containing $u^{(q)}$ and 
$u^{(-q)}$ contain each exactly one sink, while the remaining component $\smash{\cV^{(q-1)}_0}$ contains the sinks $u^{(-q+1)}, \dots u^{(q-1)}$ and the saddles $u^{(-q+3/2)}, \dots, u^{(q-3/2)}$, and no other sinks or $1$-saddles. 

The induction step proceeds as follows. The restriction $\set{u(q+1-s)}_{s\in[0,1]}$
provides a path from $u^{(q+1)}$ to $u^{(q)}$. Now assume by contradiction 
that there is another path from $u^{(q+1)}$ to a sink $u^{(q')}$, 
with $-q\leqs q'\leqs q$, whose highest point is below $U(u^{(q+1/2)})$. 
This highest point must be one of the saddles $u^{(-q+1/2)}, \dots, u^{(q-1/2)}$, 
since these are the only saddles lower than $u^{(q+1/2)}$. Call this saddle 
$u^{(r)}$. But this would mean that $u^{(q+1)}$ belongs to the open valley
of $u^{(r)}$, which lies in $\smash{\cV^{(q-1)}}$. This contradicts the fact 
that $\cV^{(q-1)}$ does not contain $u^{(q+1)}$. We conclude that the above path 
from $u^{(q+1)}$ to $u^{(q)}$ is minimal. 

Furthermore, this shows that the open valley $\cO\cV(u^{(q+1/2)})$ has two connected components, 
one of them containing $u^{(q+1)}$, and the other one containing $u^{(q)}$. The latter will in 
fact contain all sinks $u^{(-q)}, \dots, u^{(q)}$ by the induction hypothesis. A similar argument holds for $u^{(-q-1/2)}$, which implies that $\cV^{(q)}$ has a similar decomposition as~\eqref{eq:valley_decomp}. The induction step is thus complete. 
\end{proof}

The next result provides information on the metastable hierarchy.

\begin{proposition}
The potential difference 
\begin{equation}
 \Delta U(q) = U(u^{(q+1/2)}) - U(u^{(q+1)})
\end{equation} 
is a decreasing function of $q$ for $0\leqs q \leqs \frac{n}{4} - 1$. 
\end{proposition}
\begin{proof}
Using~\eqref{eq:potential_sink} and~\eqref{eq:potential_saddle}, we get 
\begin{align}
 \frac{2\pi}{K} \Delta U(q) 
 &= n \cos\biggpar{\frac{2\pi(q+1)}{n}} - (n-2)\cos\biggpar{\pi\frac{2q+1}{n-2}} \\
 &=: n\cos(\beta) - (n-2)\cos(\alpha)\;.
\end{align} 
This expression is well-defined for any $q\in\R$. We can thus compute its derivative 
with respect to $q$, to obtain 
\begin{equation}
 \frac{1}{K} \frac{\6}{\6q} \Delta U(q) 
 = \sin\alpha - \sin\beta
 = 2 \sin\frac{\alpha-\beta}{2} \cos \frac{\alpha+\beta}{2}\;.
\end{equation}
The difference 
\begin{equation}
 \alpha - \beta = -\pi\frac{n - 4(q+1)}{n(n-2)}
\end{equation} 
belongs to $(-\pi,0)$ for $q+1 < 4n$, while the sum $\alpha+\beta$ is strictly positive. 
Therefore, the $q$-derivative of $\Delta U(q)$ is strictly negative, that is, $q\mapsto \Delta U(q)$ is decreasing. 
\end{proof}

This result shows that if only sinks with $q \geqs 0$ were present, we would have a clear 
metastable hierarchy $u^{(0)} \prec u^{(1)} \prec u^{(2)} \prec \dots$. The existence of the $q\mapsto -q$ symmetry makes the situation slightly degenerate, however, and we will have to examine the consequences of that. 

%%%%%%%%%%%%%%%%%%%%%%%%%%%%%%%%%%%%%%%%%%%%%%%%%%%%%%%%%%%%%%%%%%%%%%%%%%%%%%

\section{Transition times} 
\label{sec:EK}

In this section, we estimate the time separating metastable transitions between various $q$-twisted states, or sets of $q$-twisted states. Certain expected transition times, between a $q$-twisted state and the set $\set{u^{(-\abs{q}+1)}, \dots, u^{(\abs{q}-1)}}$ of more stable states is directly obtained by applying the Eyring--Kramers relation~\eqref{eq:EK} to our situation. For more general transitions, computations are more involved. 

%%%%%%%%%%%%%%%%%%%%%%%%%%%%%%%%%%%%%%%%%%%%%%%%%%%%%%%%%%%%%%%%%%%%%%%%%%%%%%

\subsection{Preliminary computations}
\label{ssec:EK_prelim} 

A central quantity governing expected transition times is the potential difference 
\begin{equation}
 H_q = U(u^{(q-1/2)}) - U(u^{(q)})\;, 
 \qquad 1 \leqs q < \frac n4\;.
\end{equation} 
For negative $q$, we set $H_q = H_{-q}$. The potential difference
\begin{equation}
 \bar H_q = U(u^{(q+1/2)}) - U(u^{(q)})\;, 
 \qquad 0 \leqs q < \frac n4
\end{equation} 
plays a role in transitions to less stable states. 
Again, for negative $q$, we set $\bar H_q = \bar H_{-q}$. 
While~\eqref{eq:potential_sink} and~\eqref{eq:potential_saddle} provide explicit 
expressions for these quantities, it is useful to understand their behavior for 
$\abs{q} \ll n$, as described by the following estimate.

\begin{lemma}
\label{lem:Hq} 
For fixed $q$ and $n$ large, one has 
\begin{align}
\label{e:Hdiff1}
 H_q &= \frac{K}{\pi} - \biggpar{q - \frac14} \frac{K\pi}{n} 
 + \mathcal{O}\bigpar{n^{-2}}\;, 
 \qquad 1\leqs q \ll n\;,\\
\label{e:Hdiff2}
 \bar H_q &= \frac{K}{\pi} + \biggpar{q + \frac14} \frac{K\pi}{n} 
 + \mathcal{O}\bigpar{n^{-2}}\;,
 \qquad  0\leqs q \ll n\;. 
\end{align}
\end{lemma}
\begin{proof}
The values of $\bar H_0$ and $H_1$ are obtained by a direct Taylor expansion, using~\eqref{eq:potential_sink} and~\eqref{eq:potential_saddle}.
For any $0\leqs q < \frac n4$, \eqref{eq:potential_sink}, \eqref{eq:potential_saddle}, and sum-product formulas yield
\begin{align}
 U(u^{(q+1)}) - U(u^{(q)}) 
 &= \frac{Kn}{\pi} \sin\biggpar{\frac\pi n}\sin\biggpar{\frac{\pi(2q+1)}{n}}
 = (2q + 1) \frac{K\pi}{n} + \mathcal{O}\bigpar{n^{-3}}\;,\\
 U(u^{(q+3/2)}) - U(u^{(q+1/2)}) 
 &= \frac{K(n-2)}{\pi} \sin\biggpar{\frac\pi{n-2}}\sin\biggpar{\frac{2\pi(q+1)}{n-2}}\\
 &= 2(q + 1) \frac{K\pi}{n} + \mathcal{O}\bigpar{n^{-2}}\;.
\end{align} 
The result then follows by induction on $q$, using 
\begin{equation}
 \bar H_{q+1} 
 = \bar H_q + U(u^{(q+1)}) - U(u^{(q)}) - \bigbrak{U(u^{(q+1)}) - U(u^{(q)})}\;, 
\end{equation} 
and a similar relation between $H_{q+1}$ and $H_q$. 
\end{proof}
Another important quantity is the product of eigenvalues appearing in the prefactor in the Eyring--Kramers law~\eqref{eq:EK}. 
The eigenvalues $\lambda_k$ at the sink $u^{(q)}$ can be written, by~\eqref{eq:lambda_sink}, as 
\begin{equation}
\label{e:lam}
 \lambda_k = 2\pi K \cos\biggpar{\frac{2\pi (q+1)}{n}} \lambda^0_k\;, \qquad  
 \lambda_k^0 = 4\sin^2 \biggpar{\frac{\pi k}{n}}\;.
\end{equation} 
Recall from~\eqref{eq:ev_lambdak0} that the $\lambda_k^0$ are also the eigenvalues of the discrete Laplacian with periodic boundary conditions. 
The eigenvalues $\mu_k$ at the saddles are given  by 
\begin{equation}
\label{e:mu}
 \mu_k = 2\pi K \cos\biggpar{\frac{2\pi \hat q}{n}}\nu_k\;, 
 \qquad 
 \hat q = \biggpar{q + \frac12}\frac{n}{n-2}\;,
\end{equation} 
where the $\nu_k$ are the eigenvalues of $-M$ (cf.~\eqref{eq:Hessian_saddle}). 
Combining \eqref{e:lam} and \eqref{e:mu}, we get
\begin{equation}
\label{eq:prefactor} 
 \frac{\abs{\mu_1}\mu_2\dots\mu_{n-1}}{\lambda_1\dots\lambda_{n-1}}
 = \Biggpar{\frac{\cos\bigpar{\frac{2\pi\hat q}{n}}}{\cos\bigpar{\frac{2\pi(q+1)}{n}}}}^n 
 \frac{\abs{\nu_1}\dots\nu_{n-1}}{\lambda_1^0\dots\lambda_{n-1}^0}\;.
\end{equation} 
Taking logarithms and expanding, we have 
\begin{equation}
\begin{split}
 \log \Biggpar{\frac
 {\cos\bigpar{\frac{2\pi\hat q}{n}}}{\cos\bigpar{\frac{2\pi(q+1)}{n}}}}^n 
 &= n\log \frac{1 - \frac{2\pi^2\hat q^2}{n^2} + \Order{n^{-4}}}{1 - \frac{2\pi^2(q+1)^2}{n^2} + \Order{n^{-4}}} \\
 &= \frac{2\pi^2}{n} ( (q+1)^2 - \hat{q}^2)  + \bigOrder{n^{3}} \\
 &= \frac{2\pi^2}{n} \biggpar{q + \frac34} + \bigOrder{n^{-2}}\;.
\end{split}
\end{equation}
We conclude that
\begin{equation}
\label{e:cosratio}
 \Biggpar{{\frac{\cos\bigpar{\frac{2\pi\hat q}{n}}}{\cos\bigpar{\frac{2\pi(q+1)}{n}}}}}^n 
 = 1 + \frac{\pi^2(4q+3)}{2n} + \bigOrder{n^{-2}}\;.
\end{equation} 
In fact, the error for large but finite $n$ has order $1/n$, and has been computed above.
On the other hand, the second factor in~\eqref{eq:prefactor} is described 
by the following result. 

\begin{proposition}
We have the exact relation 
\begin{equation}
\label{e:nulam0ratio}
 \frac{\nu_1\dots\nu_{n-1}}{\lambda_1^0\dots\lambda_{n-1}^0}
 = -1 + \frac{2}{n}\;.
\end{equation} 
\end{proposition}
\begin{proof}
Let $L$ be the matrix~\eqref{eq:L} of the discrete Laplacian 
with periodic boundary conditions. 
We observe that $D = L + N$. The $\lambda_k^0$ are eigenvalues of $-L$, while the 
$\nu_k$ are eigenvalues of $-M = -D-N = -L-2N$. 

To avoid problems due to the zero eigenvalues, we introduce for $\eps\in\R$ the quantity
\begin{equation}
 r(\eps) = \frac{\det(\eps\one -M)}{\det(\eps\one -L)}
 = \frac{(\eps+\nu_0)(\eps+\nu_1)\dots(\eps+\nu_{n-1})}{(\eps+\lambda_0^0)(\eps+\lambda_1^0)\dots(\eps+\lambda_{n-1}^0)}\;.
\end{equation} 
This function is well-defined whenever $\eps$ is not an eigenvalue of $L$. Moreover, 
since $\nu_0 = \lambda_0^0 = 0$, we have 
\begin{equation}
\label{eq:proof_ratio_ev1} 
 \frac{\nu_1\dots\nu_{n-1}}{\lambda_1^0\dots\lambda_{n-1}^0}
 = \lim_{\eps\to0} r(\eps)\;.
\end{equation} 
Now we observe that we can write $r(\eps)$ as a Fredholm determinant 
\begin{align}
r(\eps) 
&= \det\bigbrak{(\eps\one -M)(\eps\one-L)^{-1}} \\
&= \det\bigbrak{\one - 2N(\eps\one-L)^{-1}}\;.
\end{align}
The inverse of $\eps\one-L$ can be written as 
\begin{equation}
 (\eps\one-L)^{-1}
 = \sum_{k=0}^{n-1} \frac{1}{\eps+\lambda_k^0} \Pi_k\;,
\end{equation} 
where $\Pi_k$ is the projector on the eigenspace associated with $\lambda_k^0$, given by 
\begin{equation}
 \Pi_k = \eta_k \bar\eta_k^\top\;, \qquad 
 \eta_k^\top = \frac{1}{\sqrt{n}}
 \begin{pmatrix}
  1 & \varpi^k & \varpi^{2k} & \dots & \varpi^{(n-1)k}\;
 \end{pmatrix}, \qquad 
  \varpi = \e^{\icx 2\pi/n}\;.
\end{equation} 
Note that all $\lambda_k^0$ except $\lambda_0^0 = 0$ and $\lambda_{n/2}^0 = 4$ (if $n$ is even)
have multiplicity $2$, so that the choice of the $\Pi_k$ is not unique, but the above choice is 
a valid one. Since 
\begin{equation}
 \pscal{\psi}{\eta_k} = 
 \frac{1 - \varpi^{-k}}{\sqrt{n}}\;, 
\end{equation} 
it follows that 
\begin{equation}
 N(\eps\one-L)^{-1} 
 = \sum_{k=0}^{n-1} \frac{1}{\eps+\lambda_k^0} \pscal{\psi}{\eta_k} \psi \eta_k^\top
 = 
 \begin{pmatrix}
  a_1 & a_2 & \dots & a_{n} \\
  0 & 0 & \dots & 0 \\
  \vdots & & & \vdots \\
  0 & 0 & \dots & 0 \\
  -a_1 & -a_2 & \dots & -a_{n} 
 \end{pmatrix}\;,
\end{equation} 
where
\begin{equation}
 a_j = \frac{1}{n} \sum_{k=0}^{n-1} \frac{1}{\eps+\lambda_k^0} 
 (1-\varpi^k)\varpi^{k(1-j)}\;.
\end{equation} 
In particular, for $j=1$, grouping the summands $k$ and $n-k$ one obtains 
\begin{equation}
 a_1 = \frac{1}{2n} \sum_{k=1}^{n-1} \frac{\lambda_k^0}{\eps+\lambda_k^0}\;.
\end{equation} 
A similar computation shows that $a_n = -a_1$. Now we note that 
\begin{align}
 r(\eps) = \det\bigbrak{\one - 2N(\eps\one-L)^{-1}} 
 &= 
 \begin{vmatrix}
  1-2a_1 & -2a_2 & \dots & -2a_{n-1} & -2 a_n \\
  0 & 1 & & 0& 0\\
  \vdots & & \ddots & & \vdots \\
  0 & 0 & & 1 & 0 \\
  2a_1 & 2a_2 & \dots & 2a_{n-1} & 1 + 2 a_n 
 \end{vmatrix} \\
 &= (1-2a_1)(1+2a_n) - 4a_1a_n \\
 &= 1 - 2a_1 + 2 a_n \\
 &= 1 - 4a_1\;.
\end{align} 
As $\eps$ tends to $0$, $a_1$ converges to $\frac{n-1}{2n}$. Substituting in~\eqref{eq:proof_ratio_ev1}
yields the result. 
\end{proof}

The last missing piece in order to determine the prefactor in the Eyring--Kramers law 
is the behavior of $\nu_1$. This is described by the following result.

\begin{proposition}
The negative eigenvalue $\nu_1$ of $-M$ satisfies 
\begin{equation}
\label{e:nu1}
 -\frac43 \leqs \nu_1 \leqs -\frac43 + \frac{1}{3^{n-3}}\;.
\end{equation} 
\end{proposition}
\begin{proof}
Consider the subspace $E = \set{u\in(\R/\Z)^{\Lambda} \colon u_i = -u_{n-1-i} \forall i\in\set{0,\dots,n-1}}$. This subspace has dimension $m=n/2$ if $n$ is even, 
and $m=(n-1)/2$ if $n$ is odd. One easily checks that it is invariant under 
the matrix $M$. We claim that the eigenvector $v_1$ corresponding to 
$\nu_1$ belongs to $E$. 

A basis of $E$ is given by $(\hat e_0, \dots, \hat e_{m-1}) = (e_0 - e_{n-1}, e_1 - e_{n-2}, \dots, e_{m-1} - e_{n-m})$.
The form $\widehat{M}$ of $M$ in this basis is slightly different depending on the parity of 
$n$. We consider first the case where $n$ is even. Then 
\begin{equation}
 \widehat{M} = 
 \begin{pmatrix}
  1 & 1 & 0 & \dots & \dots & 0 & 0 \\
  1 & -2 & 1 & 0 & \dots & \dots & 0 \\
  0 & 1 & -2 & 1 & 0 & \dots & 0 \\
  \vdots & \ddots & \ddots & \ddots & \ddots & \ddots & \vdots \\
  0 & \dots & 0 & 1 & -2 & 1 & 0 \\
  0 & \dots & \dots & 0 & 1 & -2 & 1 \\
  0 & 0 & \dots & \dots & 0 & 1 & -3
 \end{pmatrix}\;. 
\end{equation} 
Observe that the vector
\begin{equation}
 v = \Bigpar{1, \frac13, \frac19, \dots, \frac{1}{3^{m-1}}}^\top
\end{equation} 
satisfies 
\begin{equation}
 \widehat{M} v = \frac43 \Bigpar{v - \frac{1}{3^{m-1}}\hat e_{m-1}}\;.
\end{equation} 
As $m\to\infty$, this reduces to $\widehat{M} v = \frac43 v$, suggesting that $-\widehat{M}$ has 
an eigenvalue close to $-\frac43$. 
To prove this, we consider the matrix 
\begin{equation}
 \widetilde{M} = \widehat{M} + \frac43 \hat e_{m-1} \hat e_{m-1}^\top\;.
\end{equation} 
This matrix differs from $\widehat{M}$ only in the bottom right element, which has value $-\frac53$ 
instead of $-3$. It satisfies $\widetilde{M} v = \frac43 v$, and therefore admits $\frac43$ as an 
eigenvalue. Observe that $\widetilde{M}$ is a rank-$1$ perturbation of the discrete Laplacian 
with periodic boundary conditions $L$, namely 
\begin{equation}
 \widetilde{M} = L + \tilde\psi \tilde\psi^\top\;, 
 \qquad 
 \tilde\psi^\top = 
 \begin{pmatrix}
  \sqrt{3} & 0 & \dots & 0 & -\frac{1}{\sqrt{3}}
 \end{pmatrix}\;.
\end{equation} 
We claim that $-\frac43$ is the only negative eigenvalue of $-\widetilde{M}$.
To show this, we proceed similarly to the proof of Lemma~\ref{lem:rank_one_ev}. 
Let $\tilde\mu$ be an eigenvalue of $-\widetilde{M}$. This means that there exists a 
vector $\tilde\ph \neq 0$ such that 
\begin{equation}
\bigpar{L + \tilde\psi \tilde\psi^\top +\tilde\mu} \tilde\ph = 0\;,
\end{equation} 
and therefore 
\begin{equation}
\label{eq:ev_Ltildemu} 
 \bigpar{L + \tilde \mu} \tilde\ph = -\pscal{\tilde\psi}{\tilde\ph}\tilde\psi\;.
\end{equation} 
If $\tilde\mu$ is an eigenvalue of $-L$, then $\tilde\mu\geqs0$, and the claim is 
proved. We thus assume that $\tilde\mu$ is not an eigenvalue of $-L$. Then 
\begin{equation}
\label{eq:phtilde} 
 \tilde\ph = -\pscal{\tilde\psi}{\tilde\ph} \bigpar{L+\tilde\mu}^{-1} \tilde\psi\;.
\end{equation} 
We must have $\pscal{\tilde\psi}{\tilde\ph} \neq 0$, since otherwise~\eqref{eq:ev_Ltildemu} would imply that $\tilde\mu$ is an eigenvalue of $-L$. Taking the inner product of~\eqref{eq:phtilde} with 
$\tilde\psi$, and dividing by $\pscal{\tilde\psi}{\tilde\ph}$, we obtain 
\begin{equation}
 -1 = \pscal{\tilde\psi}{(L+\tilde\mu)^{-1}\tilde\psi}
 = \sum_k \frac{\pscal{\tilde\psi}{\ph^0_k}^2}{\tilde\mu - \lambda^0_k} =: F_1(\tilde\mu)\;.
\end{equation} 
Here the sum ranges over the eigenvalues $\lambda^0_k$ of $-L$, with associated normalized eigenvectors 
$\ph^0_k$. The function $F_1$ has poles on the eigenvalues of $-L$, which are non-negative, and is otherwise strictly decreasing. Therefore, the equation $F_1(\tilde\mu)$ can have only one strictly negative eigenvalue, and the claim is proved. 

Consider now the eigenvalue equation $(\widehat{M} + \nu_1)\ph_1 = 0$, which is equivalent to 
\begin{equation}
 \bigpar{\widetilde{M} + \nu_1}\ph_1 = \frac43 \pscal{\hat e_{m-1}}{\ph_1}\hat e_{m-1}\;.
\end{equation} 
If $\nu_1 = \tilde\mu_1 = -\frac43$, the proposition is proved. Otherwise, $\nu_1$ 
is not an eigenvalue of $\widetilde{M}$, and $\pscal{\hat e_{m-1}}{\ph_1} \neq 0$, so that we get 
\begin{equation}
 \frac34 = \pscal{\hat e_{m-1}}{(\widetilde{M} + \nu_1)^{-1}\hat e_{m-1}}
 = \sum_k \frac{\pscal{\hat e_{m-1}}{\tilde\ph_k}^2}{\nu_1-\tilde\mu_k}\;,
\end{equation} 
where the sum ranges over the eigenvalues $\tilde\mu_k$ of $\widetilde{M}$, with associated 
normalized eigenvectors $\tilde\ph_k$. 
Isolating the term $k=1$, we obtain 
\begin{equation}
 \nu_1 + \frac43 = \frac{\pscal{\tilde\ph_1}{\hat e_{m-1}}^2}{\frac34 + S}\;, 
 \qquad 
 S = \sum_{k\neq1} \frac{\pscal{\hat e_{m-1}}{\tilde\ph_k}^2}{\tilde\mu_k-\nu_1}\;.
\end{equation} 
Since $\nu_1 < 0$ and all $\tilde\mu_k$ except $\tilde\mu_1$ are positive, $S$ is positive. 
Furthermore, since $\tilde\ph_1 = v/\norm{v}$ we have 
\begin{equation}
 \pscal{\tilde\ph_1}{\hat e_{m-1}}^2
 = \frac{\pscal{v}{\hat e_{m-1}}^2}{\norm{v}^2}
 = \frac{9}{8(1-9^{-m})3^{2m-2}}
 \leqs \frac{1}{4\cdot 3^{2m-2}}\;.
\end{equation} 
This yields the upper bound on $\nu_1$, using the fact that $2m\geqs n-2$. 

It remains to consider the case where $M$ is odd. Then 
\begin{equation}
 \widehat{M} = 
 \begin{pmatrix}
  1 & 1 & 0 & \dots & \dots & 0 & 0 \\
  1 & -2 & 1 & 0 & \dots & \dots & 0 \\
  0 & 1 & -2 & 1 & 0 & \dots & 0 \\
  \vdots & \ddots & \ddots & \ddots & \ddots & \ddots & \vdots \\
  0 & \dots & 0 & 1 & -2 & 1 & 0 \\
  0 & \dots & \dots & 0 & 1 & -2 & 1 \\
  0 & 0 & \dots & \dots & 0 & 1 & -2
 \end{pmatrix}
 = \widetilde{M} - \frac13 \hat e_{m-1}\hat e_{m-1}^\top\;. 
\end{equation} 
Then a completely analogous argument shows that $\nu_1 \leqs -\frac43 + \frac14 3^{3-2m}
= -\frac43 + \frac14 3^{4-n}$, which is smaller than the stated bound.
\end{proof}

%%%%%%%%%%%%%%%%%%%%%%%%%%%%%%%%%%%%%%%%%%%%%%%%%%%%%%%%%%%%%%%%%%%%%%%%%%%%%%

\subsection{Expected transition time from less stable to more stable states}
\label{ssec:EK_stable} 

The following result follows essentially from~\cite{BEGK2004}, except for the fact that the potential $U$ has the point symmetry $U(-u) = U(u)$. A generalization of the results from~\cite{BEGK2004} to potentials invariant under a discrete symmetry group has been obtained in~\cite{BD15} for continuous-time Markov chains on a finite set, and in~\cite{Dutercq_thesis} for stochastic differential equations. See also~\cite{BD16} for an application to a discretized Allen--Cahn equation with conserved total mass. 

Below, we will write 
\begin{equation}
 \cM_q = \set{u^{(-q)},u^{(-q+1)},\dots,u^{(q)}}
\end{equation} 
for the $q$th metastable set, while the distance between two discrete sets $A, B\subset (\R/\Z)^\Lambda$ is given by $\dist(A,B) = \inf_{a\in A, b\in B} \norm{a-b}$. 

\begin{theorem}
\label{thm:main} 
Fix $q$ such that $0\leqs q < \frac n4$, $\delta > 0$, and let 
\begin{equation}
 \tau_{\cM_q} = \inf\Bigset{t > 0 \colon \dist\bigpar{u_t, \cM_q} < \delta}
\end{equation}
be the first-hitting time of a $\delta$-neighborhood of the set $\cM_q$. 
Then the expectation of $\tau_{\cM_q}$, starting from $u^{(q+1)}$, satisfies 
\begin{equation}
\expecin{u^{(q+1)}}{\tau_{\cM_q}}
 = C(q,n) \e^{H_{q+1}/\eps} \bigbrak{1 + R_n(\eps)}\;, 
\end{equation} 
where 
\begin{align}
\label{e:Casympt}
 C(q,n) &= \frac{3}{4 K n} \left(1 + \frac{\pi^2(4q+3)-4}{4n}\right) + \mathcal{O}\bigpar{n^{-3}}\;, \\
\label{eq:Hasympt}
 H_{q+1} &= U(u^{(q+1/2)}) - U(u^{(q+1)})
 = \frac{K}{\pi} - \biggpar{q + \frac34} \frac{K\pi}{n} 
 + \mathcal{O}\bigpar{n^{-2}}\;, \\
 \lim_{\eps\to 0} R_n(\eps) &= 0\;.
 \label{eq:Rasympt}
\end{align} 
\end{theorem}

\begin{proof}
The proof is almost an application of~\cite[Theorem~3.2]{BEGK2004}, which provides an 
Eyring--Kramers law of the form~\eqref{eq:EK}, except that we have to deal with degeneracies in the potential landscape.
The extra factor $n^{-1}$ in the expected transition time is due to the existence of 
$n$ saddles, having the same height, each providing an optimal pathway between the $q$-twisted 
state $u^{(q+1)}$ and the set $\cM_q$. As a result, the integer 
called $k$ in~\cite[Equation (3.2)]{BEGK2004} is equal to $n$, while the eigenvalues of the 
saddles are the same for all summands. 

The difference between our situation and the one described in~\cite[Theorem~3.2]{BEGK2004} is 
that for any $q\neq0$, the sinks $u^{(q)}$ and $u^{(-q)}$ are at the same height, as are 
the relevant saddles. Such symmetric situations have been analysed in~\cite{BD15} for continuous-time Markov chains, and in~\cite[Chapter 5]{Dutercq_thesis} for diffusion processes. 

In our case, the set of sinks $u^{(q)}$ is invariant under the group $G = \set{\id, I}$, where $\id$ is the identity map on the phase space $(\R/\Z)^\Lambda$, while $I$ is the inversion given by $I(u) = -u$. Note that $G$ is equivalent to $\Z_2 = \Z/2\Z$. There are two important group-theoretic quantities that influence the Eyring--Kramers law:
\begin{itemize}
\item   The \emph{orbit} of a point $u\in(\R/\Z)^\Lambda$ is defined as the set $O_u = \set{g(u) \colon g\in G}$. In particular, we have $O_{u^{(0)}} = \set{u^{(0)}}$, while 
$O_{u^{(q)}} = \set{u^{(q)}, u^{(-q)}}$ for $q\neq0$.

\item   The \emph{stabiliser} of a point $u\in(\R/\Z)^\Lambda$ is defined as the set $G_u = \set{g\in G \colon g(u) = u}$. It is a subgroup of $G$. In particular, we have $G_{u^{(0)}} = G$, while  for $q\neq0$, $G_{u^{(q)}} = \set{\id}$.
\end{itemize}
Here, we are concerned with transitions between the orbit $O_{u^{(q+1)}}$, and the union of the 
orbits from $O_{u^{(0)}}$ to $O_{u^{(q)}}$. For continuous-time Markov chains, \cite[Theorem~3.2]{BD15} states that an Eyring--Kramers law still holds in the symmetric case, but with an 
extra factor 
\begin{equation}
 \frac{\#(G_{u^{(q+1)}}\cap G_{u^{(q)}})}{\#(G_{u^{(q+1)}})}\;.
\end{equation} 
However, in our case, we have $\#(G_{u^{(q+1)}}\cap G_{u^{(q)}})= \#(G_{u^{(q+1)}}) = 1$, 
so that there is no change to the Eyring--Kramers formula. The same holds true for 
diffusions, as discussed in~\cite[Sections 5.2 and 5.3]{Dutercq_thesis}. 

{It remains to specialise~\eqref{eq:EK} to the situation at hand. The form~\eqref{eq:Hasympt} 
of the Arrhenius exponent follows directly from~\eqref{e:Hdiff1}. As for the prefactor, 
using~\eqref{eq:prefactor}, \eqref{e:cosratio} and \eqref{e:nulam0ratio} 
and expanding in powers of $n^{-1}$, we find that it becomes
\begin{align}
 C(q,n) &= \frac{1}{n} \frac{2\pi}{\abs{\mu_1}} \sqrt{\frac{\abs{\mu_1}\mu_2\dots\mu_{n-1}}
 {\lambda_1\dots\lambda_{n-1}}} \\
 &= \frac{1}{n} \frac{2\pi}{\abs{\mu_1}} 
 \Biggpar{\frac{\cos\bigpar{\frac{2\pi\hat q}{n}}}{\cos\bigpar{\frac{2\pi(q+1)}{n}}}}^{n/2}
 \sqrt{\frac{\abs{\nu_1}\nu_2\dots\nu_{n-1}}{\lambda^0_1\dots\lambda^0_{n-1}}} \\
 &= \frac{1}{n} \frac{2\pi}{\abs{\mu_1}}
 \Biggpar{1 + \frac{\pi^2(4q+3)}{4n} + \Order{n^{-2}}}
 \Biggpar{1 - \frac{1}{n} + \Order{n^{-2}}} \\
 &= \frac{1}{n} \frac{2\pi}{\abs{\mu_1}}
 \Biggpar{1 + \frac{\pi^2(4q+3)-4}{4n} + \Order{n^{-2}}}\;.
\end{align}
Since \eqref{eq:prefactor} also implies $\abs{\mu_1} = 2\pi K\abs{\nu_1}(1+\Order{n^{-2}}$, 
the expression~\eqref{e:Casympt} for $C(q,n)$ follows from the bounds~\eqref{e:nu1}
on $\nu_1$ (note that the exponentially small error there is negligible with respect 
to any power of $n^{-1}$).
}
% The remainder of the proof, providing the asymptotics of the constants, is obtained from \eqref{e:Hdiff1},  \eqref{e:cosratio}, \eqref{e:nulam0ratio} and \eqref{e:nu1}, along with a Taylor expansion of $\cos(2\pi \hat{q}/n)$ in the large $n$ limit.  
\end{proof}

\begin{remark} \hfill
\begin{enumerate}
\item Figure \ref{fig:asympt} presents numerical verification of the asymptotics in 
\eqref{e:Casympt} and \eqref{eq:Hasympt}.

\item   The above computations would allow us to determine the next-to-leading 
order of $C(q,n)$.

\item   We do not control how $R_n(\eps)$ depends on $n$. In principle, as 
$n$ gets larger, the convergence of $R_n(\eps)$ to $0$ may become slower. 
However, arguments used in~\cite{BG2013} in a somewhat similar situation suggest 
that this is not the case.
\end{enumerate}
\end{remark}
{%
\begin{remark}
\label{rem:concentration}  
As pointed out in Section~\ref{sec:stoch}, it is known that the  
renormalised variable 
\begin{equation}
\label{eq:tauhat} 
\hat{\tau}_{\cM_q} = \frac{\tau_{\cM_q}}{\expecin{u^{(q+1)}}{\tau_{\cM_q}}} 
\end{equation} 
converges in distribution to an exponential variable of parameter $1$, 
cf.~\eqref{eq:Day}. Since 
\begin{equation}
 \expecin{u^{(q+1)}}{\log(\hat\tau_{\cM_q})}
 = \int_0^\infty \probin{u^{(q+1)}}{\log(\hat\tau_{\cM_q}) > t} \6t 
 = \int_0^\infty \probin{u^{(q+1)}}{\hat\tau_{\cM_q} > \e^t} \6t\;,  
\end{equation} 
it follows from~\eqref{eq:Day} that
\begin{equation}
 \lim_{\eps\to0} \expecin{u^{(q+1)}}{\log(\hat\tau_{\cM_q})} 
 = \int_0^\infty \e^{-\e^t} \6t 
 = \int_1^\infty \frac{\e^{-z}}{z}\6z 
 = \Gamma(0,1)\;.
\end{equation} 
Using~\eqref{eq:tauhat}, one readily obtains 
\begin{equation}
 \lim_{\eps\to0} \eps \expecin{u^{(q+1)}}{\log(\tau_{\cM_q})}
 = H_{q+1}\;.
\end{equation} 
Furthermore, a similar computation shows that the variance of $\log(\hat\tau_{\cM_q})$
converges to a positive quantity of order $1$ as $\eps\to0$. Therefore, the 
variance of $\eps\log(\tau_{\cM_q})$ has order $\eps^2$, showing that this random 
variable converges in probability to its expectation $H_{q+1}$ as $\eps\to0$. 
\end{remark}
}

%%%%%%%%%%%%%%%%%%%%%%%%%%%%%%%%%%%%%%%%%%%%%%%%%%%%%%%%%%%%%%%%%%%%%%%%%%%%%%

\subsection{More general transitions}
\label{ssec:EK_general} 

In this section, we explain how to compute the expected transition time 
between more general states, without giving detailed proofs. The basic idea is that the 
dynamics should be well-approximated by a continuous-time Markov chain on the set of stable 
$q$-twisted states, see Figure~\ref{fig:Markov}, with transition rates given by 
\begin{equation}
 p_{q,q+1} = \frac{1}{C_{q,q+1}} \e^{-\bar H_q/\eps}\;, 
 \qquad 
 p_{q+1,q} = \frac{1}{C_{q+1,q}} \e^{-H_q/\eps}
\end{equation} 
for $0 \leqs q < \frac n4$, while $p_{-q,-q+1} = p_{q,q+1}$, 
$p_{-q+1,-q} = p_{q+1,q}$, and $p_{q,\bar q} = 0$ if $\abs{\bar q - q} \geqs2$. 
Here the coefficients $C_{q,\bar q}$ denote the prefactors in the associated 
Eyring--Kramers laws. 

The validity of such a reduction has been analysed in~\cite{Rezakhanlou_Seo_2021scaling,Seo2020} for diffusions, and in~\cite{B2023}
for continuous-space Markov chains. To justify this approximation in our situation, 
we would need to show that the most likely sinks that can be reached from a saddle 
$u^{(q+1/2)}$ are $u^{(q)}$ and $u^{(q+1)}$, in the sense that the unstable manifolds 
of the saddle converge to these two sinks. While this is indeed the case for 
$u^{(q+1)}$, we have not given a full proof of the fact that $u^{(q)}$ is indeed 
the other sink that the unstable manifold connects to. 

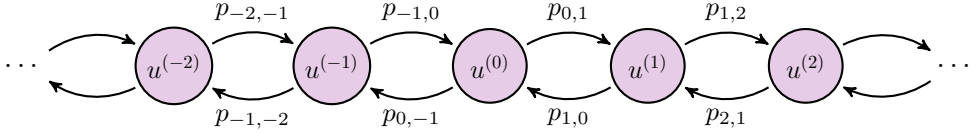
\begin{figure}[t]
\begin{center}
\begin{tikzpicture}[->,>=stealth',shorten >=2pt,shorten <=2pt,auto,node
distance=3.0cm, thick,main node/.style={circle,scale=0.7,minimum size=1.4cm,
fill=violet!20,draw,font=\sffamily\Large}]

  \node[main node] (0) {$u^{(0)}$};
  \node[main node] (1) [right of=0] {$u^{(1)}$};
  \node[main node] (2) [right of=1] {$u^{(2)}$};
  \node[main node] (-1) [left of=0] {$u^{(-1)}$};
  \node[main node] (-2) [left of=-1] {$u^{(-2)}$};
  \node[node distance=2cm] (3) [right of=2] {$\dots$};
  \node[node distance=2cm] (-3) [left of=-2] {$\dots$};

  \path[every node/.style={font=\sffamily\small}]
    (-3) edge [bend left,above] (-2)
    (-2) edge [bend left,above] node {$p_{-2,-1}$} (-1)
    (-1) edge [bend left,above] node {$p_{-1,0}$} (0)
    (0) edge [bend left,above] node {$p_{0,1}$} (1)
    (1) edge [bend left,above] node {$p_{1,2}$} (2)
    (2) edge [bend left,above] (3)
    (3) edge [bend left,below] (2)
    (2) edge [bend left,below] node {$p_{2,1}$} (1)
    (1) edge [bend left,below] node {$p_{1,0}$} (0)
    (0) edge [bend left,below] node {$p_{0,-1}$} (-1)
    (-1) edge [bend left,below] node {$p_{-1,-2}$} (-2)
    (-2) edge [bend left,below] (-3)
    ;
\end{tikzpicture}
\end{center}
\caption{Continuous time Markov chain approximating the dynamics of transitions 
between $q$-twisted states.}
\label{fig:Markov}
\end{figure}

{%
\begin{remark}
Strictly speaking, we have not shown that the sink that is easiest to reach 
from the $q$-twisted state $u^{(q)}$ with $q>0$ is $u^{(q-1)}$. As illustrated by 
Figure~\ref{fig:meta_hierarchy2}, any sink in the metastable set $\cM_{q-1}$ 
may in fact play this role. On the other hand, our analysis shows that it is 
indeed easier to reach the $q$-twisted state $u^{(q+1)}$ from $u^{(q)}$ 
than the $q$-twisted state $u^{(q+2)}$, since this would require crossing 
a higher $1$-saddle.
\end{remark}
\begin{remark}
The inversion symmetry implies that the $q$-twisted states $u^{(q)}$ 
and $u^{(-q)}$ play a symmetric role. One option would therefore be to investigate 
the coarse-grained Markov chain on the set $\cY=\set{0,1,\dots,m}$, in which 
each $q>0$ represents the agglomerated state $\set{u^{(-q)},u^{(q)}}$. 
For this modified chain, some transition rates will be twice as large as for the 
original chain on $\cX$. 
\end{remark}
}

Assuming the approximation is indeed justified, expected transition times can be computed as follows. Denote the Markov chain by $(X_t)_{t\geqs0}$. It takes values in $\cX = \set{-m,\dots,m}$, where $m = \max\set{q\in\N_0\colon q<\frac n4}$. Here $q\in\cX$ represents the state $u^{(q)}$.
Let $L$ denote the infinitesimal generator of the process, 
that is, the $2m+1$ by $2m+1$ matrix with entries 
\begin{equation}
 L_{q,\bar q} = 
 \begin{cases}
  p_{q,\bar q} & \text{if $\bar q \neq q$\;,} \\
  -\sum_{\bar q\neq q} p_{q,\bar q} & \text{if $\bar q = q$\;.}
 \end{cases}
\end{equation} 
For a set $A\subset\cX$, 
if 
\begin{equation}
 w_A(q) = \bigexpecin{q}{\tau_A}\;,
 \qquad 
 \tau_A = \inf\bigset{t>0 \colon X_t\in A}\;,
\end{equation} 
we have the relation 
\begin{equation}
\label{eq:Markov_expected_time} 
 \sum_{\bar q\in A^c} L_{q,\bar q} w_A(q) = -1\;.
\end{equation} 
For a proof, see for instance~\cite[Chapter 3]{Norris}. 

We give a few examples of applications of~\eqref{eq:Markov_expected_time}.

\begin{enumerate}
\item  If $A^c = \set{0}$, then $\tau_A$ denotes the first-hitting time of any 
sink different from $u^{(0)}$. In that case, we obtain
\begin{equation}
 \bigexpecin{0}{\tau_A}
 = \frac{1}{p_{0,1} + p_{0,-1}}
 = \frac{1}{2p_{0,1}}
 = \frac12 C_{01} \e^{\bar H_0/\eps}\;.
\end{equation} 
Here the factor $\frac12$ is due to the fact that there is an equal probability 
to escape towards positive and negative $q$.

\item   If $A^c = \set{q}$ for some $q\neq0$, say $q>0$, then $\tau_A$ denotes 
again the first-hitting time of any sink different from $u^{(q)}$. Then we find
\begin{equation}
 \bigexpecin{q}{\tau_A}
 = \frac{1}{p_{q,q+1} + p_{q,q-1}}
 = \frac{C_{q,q-1} \e^{H_q/\eps}}
 {1 + \frac{C_{q,q-1}}{C_{q,q+1}} \e^{-[\bar H_q - H_q]/\eps}}\;.
\end{equation} 
It follows from Lemma~\ref{lem:Hq} that 
\begin{equation}
 \bar H_q - H_q = \frac{K\pi}{2n} + \mathcal{O}\bigpar{n^{-2}}\;.
\end{equation} 
Therefore, for finite $n$, the expectation is dominated by transitions to 
$u^{(q-1)}$, and hence $\bigexpecin{q}{\tau_A} \simeq C_{q,q-1} \e^{H_q/\eps}$.

\item   As a more complicated example, let us look at the case where 
$A^c = \set{-1,0,1}$. Then we have to solve the system 
\begin{equation}
 \begin{pmatrix}
  -(L_{-1,-2} + L_{-1,0}) & L_{-1,0} & 0 \\
  L_{0,-1} & -(L_{0,-1}+L_{0,1}) & L_{0,1} \\
  0 & L_{1,0} & -(L_{1,0}+L_{1,2}) 
 \end{pmatrix}
 \begin{pmatrix}
  w_A(-1) \\ w_A(0) \\ w_A(1)  
 \end{pmatrix}
= 
\begin{pmatrix}
 -1 \\ -1 \\ -1
\end{pmatrix}\;.
\end{equation} 
The symmetry $L_{q,\bar q} = L_{-q,-\bar q}$ implies $w_A(-1) = w_A(1)$. 
This yields an effective two-dimensional system, whose solution is 
\begin{align}
\bigexpecin{0}{\tau_A} = 
 w_A(0) &= 
 \frac{2L_{0,1} + L_{1,0} + L_{1,2}}{2L_{0,1}L_{1,2}}\;,\\
\bigexpecin{1}{\tau_A} = 
 w_A(1) &= 
 \frac{2L_{0,1} + L_{1,0}}{2L_{0,1}L_{1,2}}\;.
\end{align}
For finite $n$, we have $\bar H_1 > \bar H_0 > H_1$, which implies 
$L_{1,2} \ll L_{0,1} \ll L_{1,0}$. This implies that there exists a constant $\theta>0$,
of order $n^{-1}$, such that 
\begin{align}
 \bigexpecin{0}{\tau_A} 
 &= \frac{C_{0,1}C_{1,2}}{2C_{1,0}}
 \e^{[\bar H_0 - H_1 + \bar H_2]/\eps} \bigbrak{1+\mathcal{O}(\e^{-\theta/\eps})}\;, \\
 \bigexpecin{1}{\tau_A} 
 &= \frac{C_{0,1}C_{1,2}}{2C_{1,0}}
 \e^{[\bar H_0 - H_1 + \bar H_2]/\eps} \bigbrak{1+\mathcal{O}(\e^{-\theta/\eps})}\;.
\end{align} 
In other terms, the expected first-hitting time does not depend on the starting point 
to leading order. What happens is that if the system starts in state $u^{(1)}$, 
it will, with overwhelming probability, visit state $u^{(0)}$ before any other state. 
Note that the exponent $\bar H_0 - H_1 + \bar H_2$ is equal to the potential 
difference $V(u^{(3/2)}) - V(u^{(0)})$, which is precisely the relative communication
height between the states $u^{(0)}$ and $u^{(2)}$.
\end{enumerate}

%%%%%%%%%%%%%%%%%%%%%%%%%%%%%%%%%%%%

\section{Numerical simulations} 
\label{sec:numerical}

In this section, we provide numerical simulations illustrating our main results, 
as well as some extensions to more general couplings. Source code is available at \url{https://github.com/gideonsimpson/KuramotoMetastability}; data is available upon request.

\begin{figure}[!t]
    \centering
    \subfigure[$q=0$]{\includegraphics[width=6.5cm]{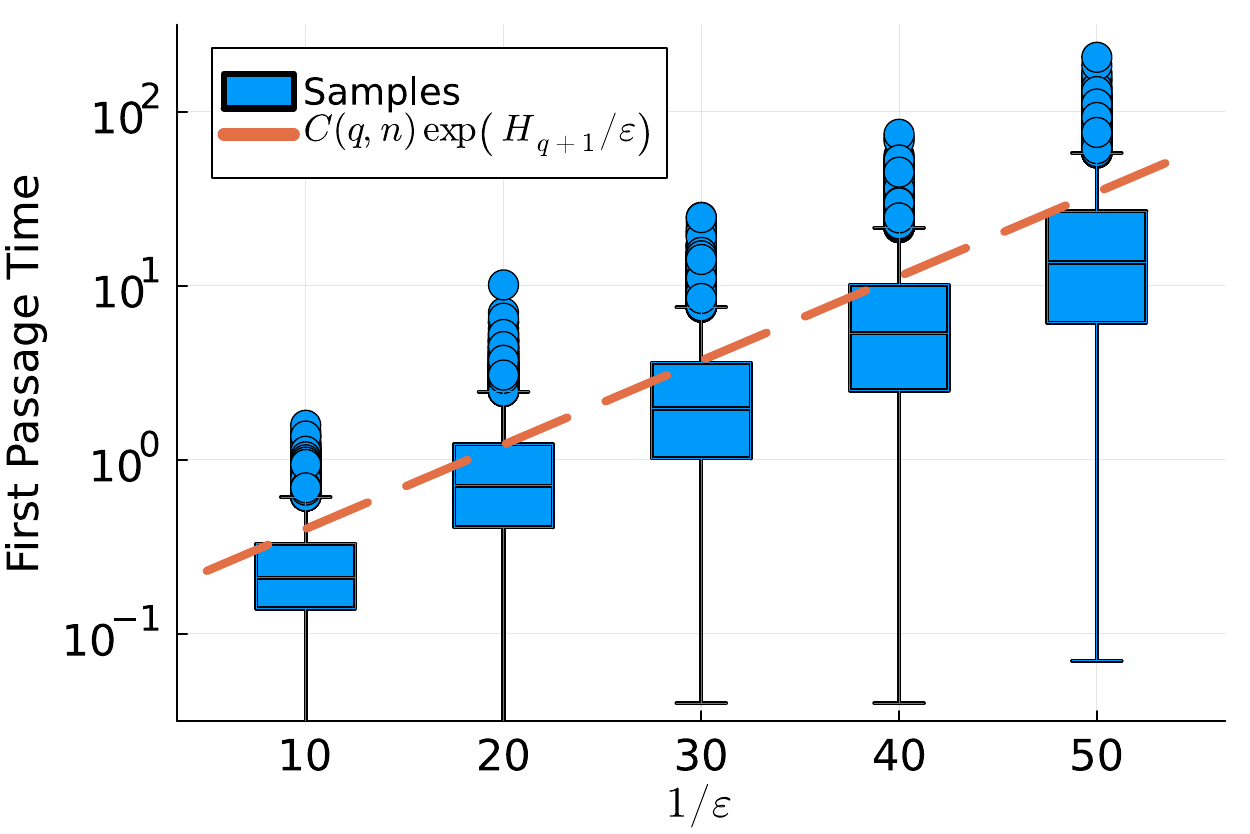}}
    \subfigure[$q=1$]{\includegraphics[width=6.5cm]{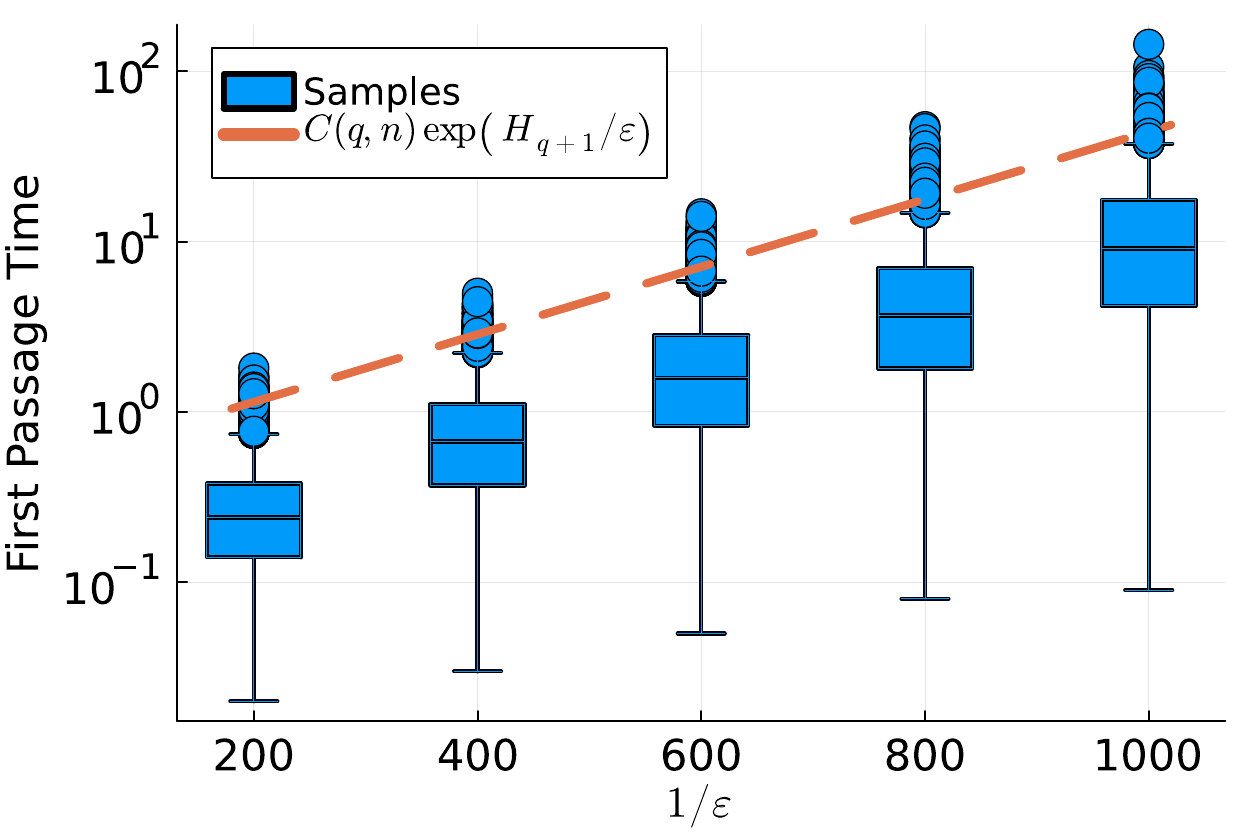}}

    \caption{Empirical first passage time distributions for the $n=10$ system, compared against the Eyring--Kramers formula \eqref{eq:EK}.  For each value of $\eps$, $10^4$ independent trials were performed. {Units are dimensionless with $K=1$.}}
    \label{fig:fptn10}
\end{figure}

\begin{figure}[!t]
    \centering
    \subfigure[$q=0$]{\includegraphics[width=6.5cm]{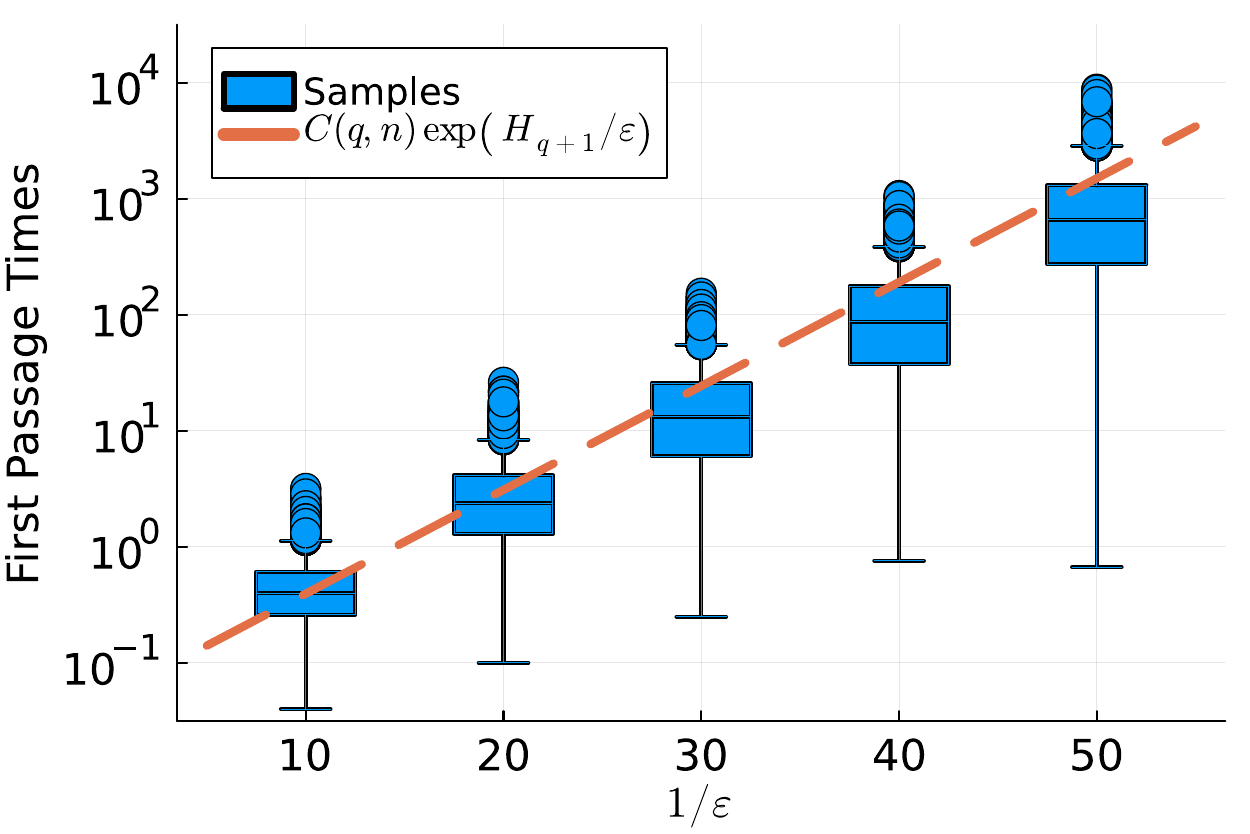}}
    \subfigure[$q=1$]{\includegraphics[width=6.5cm]{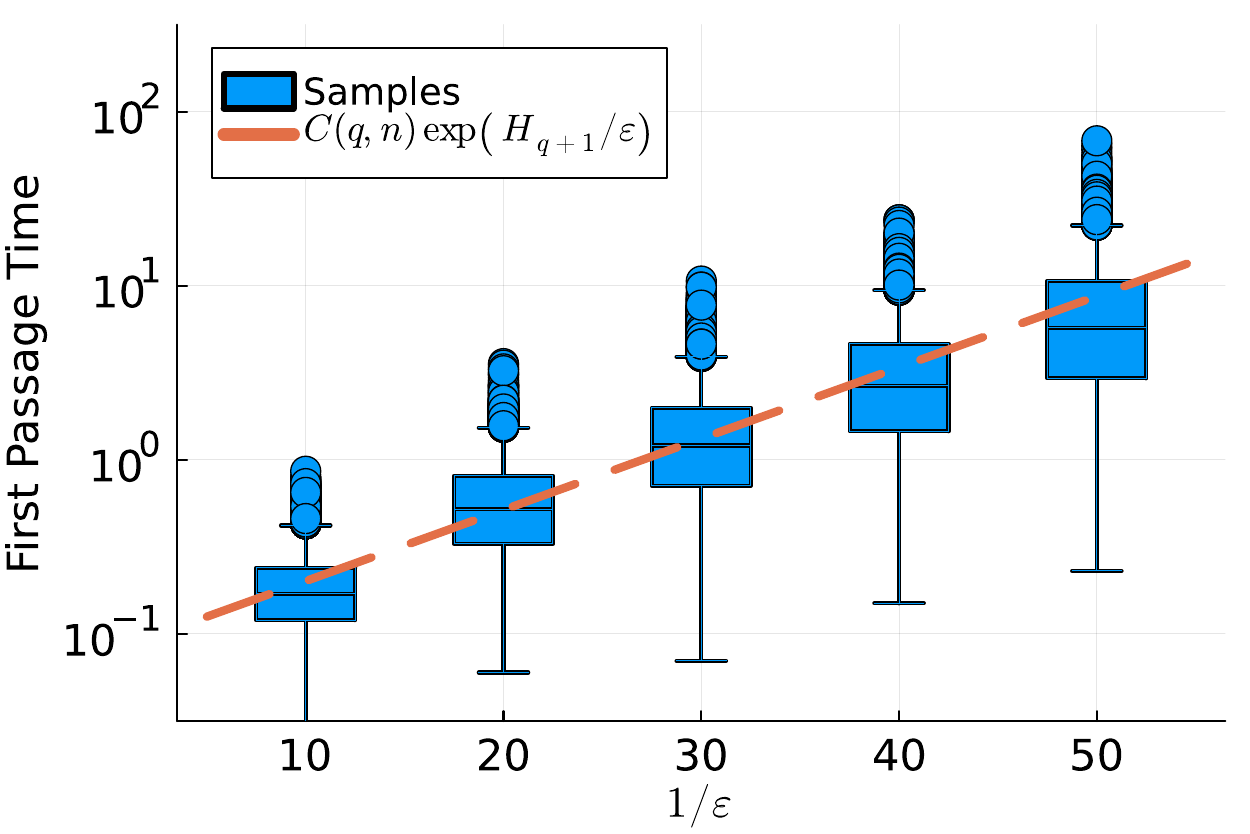}}

    \subfigure[$q=2$]{\includegraphics[width=6.5cm]{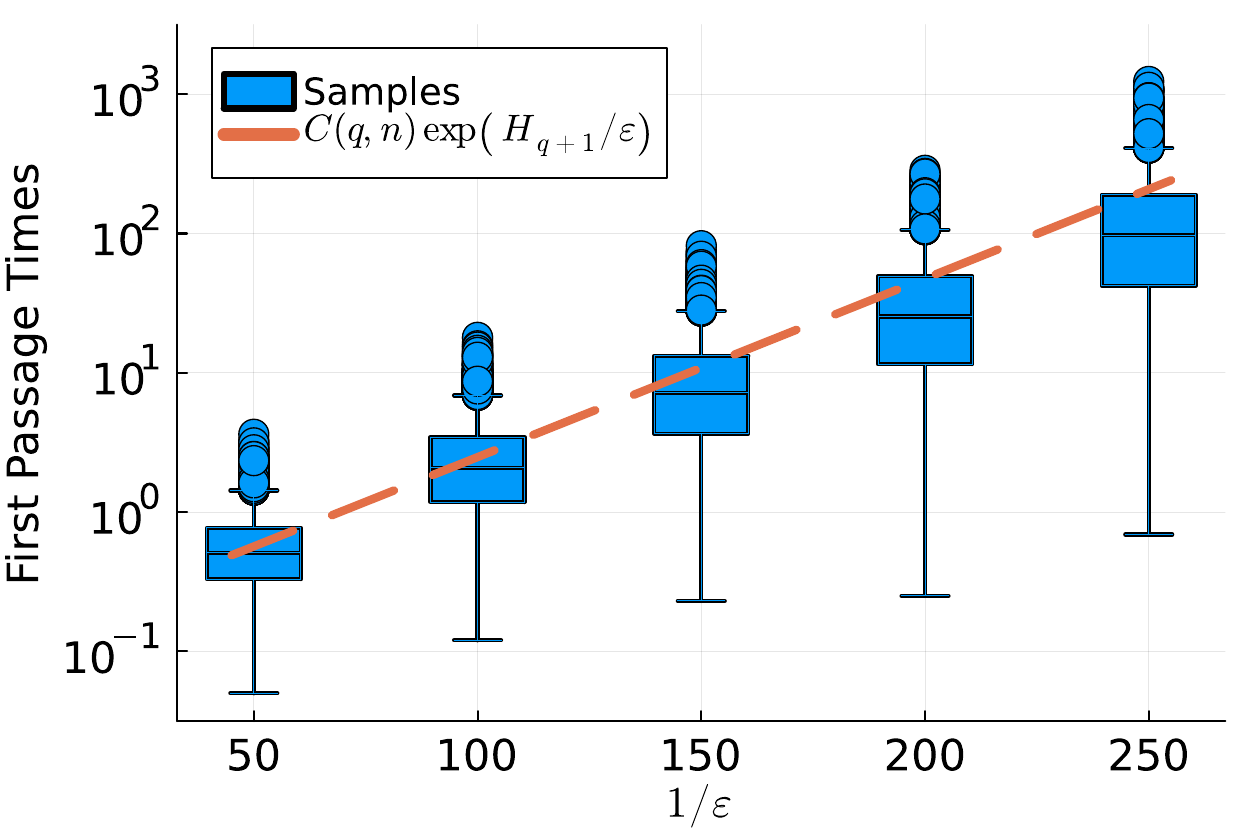}}
    \subfigure[$q=3$]{\includegraphics[width=6.5cm]{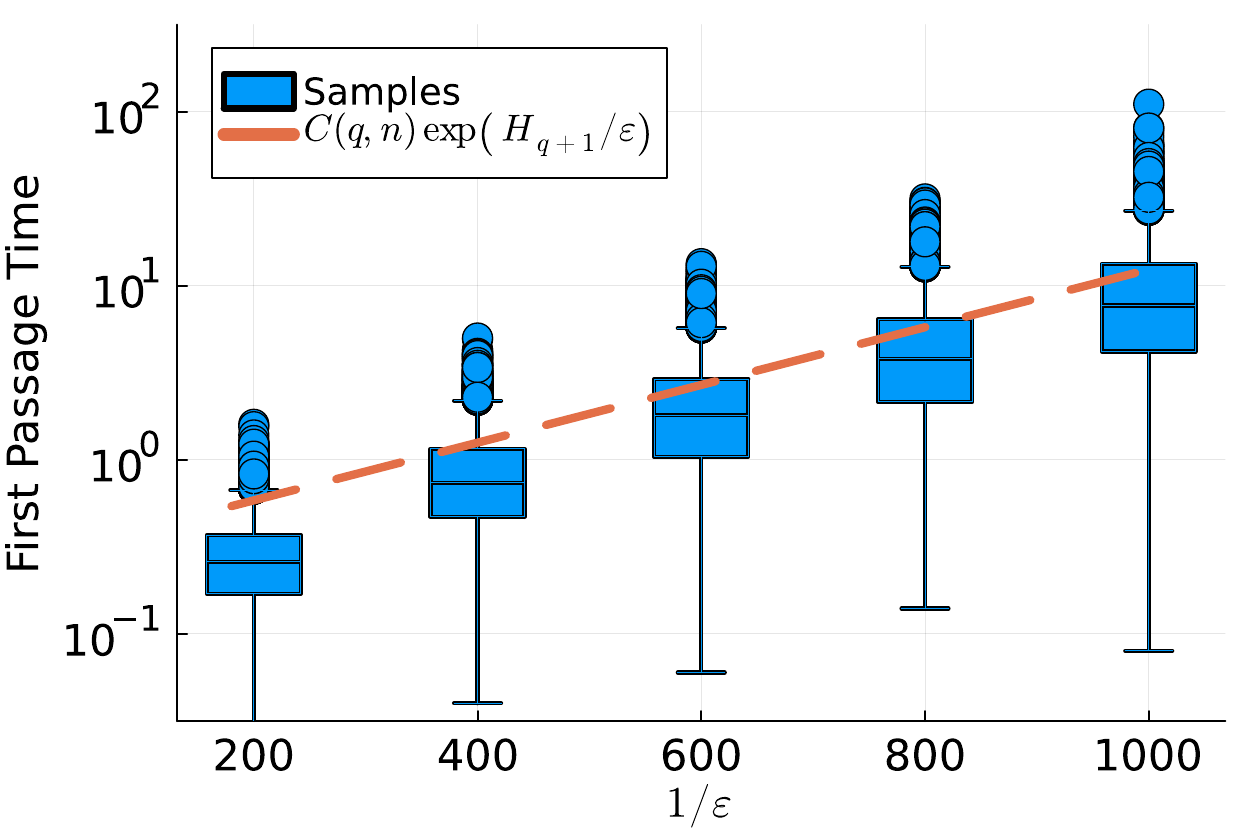}}
    
    \caption{Empirical first passage time distributions for the $n=20$ system, compared against the Eyring--Kramers formula \eqref{eq:EK}. For each value of $\eps$, $10^4$ independent trials were performed. {Units are dimensionless with $K=1$.}}
    \label{fig:fptn20}
\end{figure}

\subsection{Eyring--Kramers asymptotics for the Kuramoto model}

Our first numerical result verifies the Eyring--Kramers law \eqref{eq:EK} for the Kuramoto problem.  This is done directly by running independent trials of \eqref{eq:KM-forced}, with initial conditions given by one of the twisted states.  As shown in Figures \ref{fig:fptn10}, \ref{fig:fptn20}, and \ref{fig:fptn40}, we have good agreement between the data and the law.  Though some cases have closer agreement of the sample means than others, the trends are consistent across all cases and the error is well under an order of magnitude.

We note that we are comparing against the {\it exact} Eyring--Kramers formula, and not the asymptotic expansion obtained in Theorem \ref{thm:main}. In each problem, $10^4$ independent trials were performed with $\Delta t = 10^{-2}$ using Euler--Maruyama time stepping (see \cite{MS2023} for additional details).  Detection of an escape from the basin of attraction was performed by using a BFGS routine (see \cite{heath_scientific_2018,nocedal_numerical_2006})  to minimize the energy $U$ at the current state of the system.  Values of $\eps$ were chosen for each problem such that we had
\begin{equation}
1 \lesssim H_{q+1}/\eps \lesssim 10.
\end{equation}
This makes the problems sufficiently low temperature such that the system is entering the Arrhenius regime, but not so low as to require advanced rare event simulation techniques.

\begin{figure}[!t]
    \centering
    \subfigure[$q=0$]{\includegraphics[width=6.5cm]{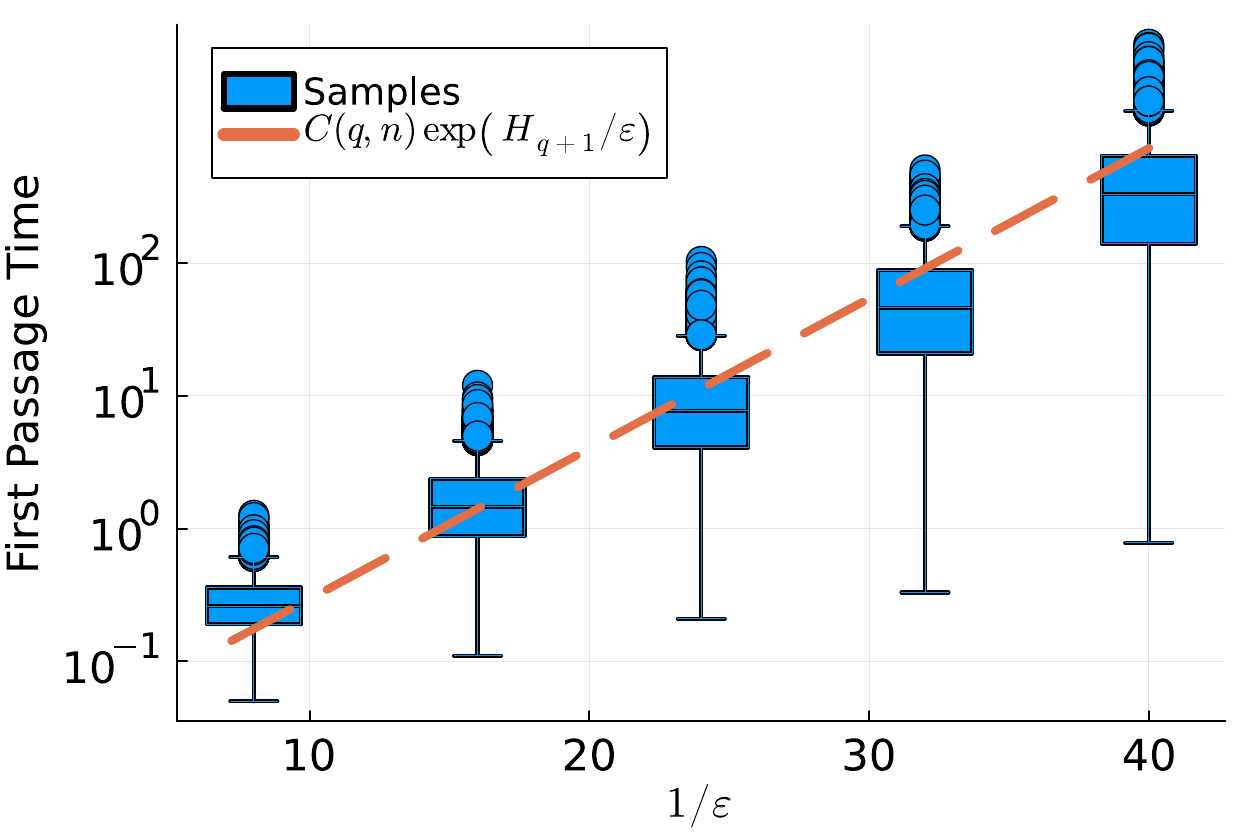}}
    \subfigure[$q=1$]{\includegraphics[width=6.5cm]{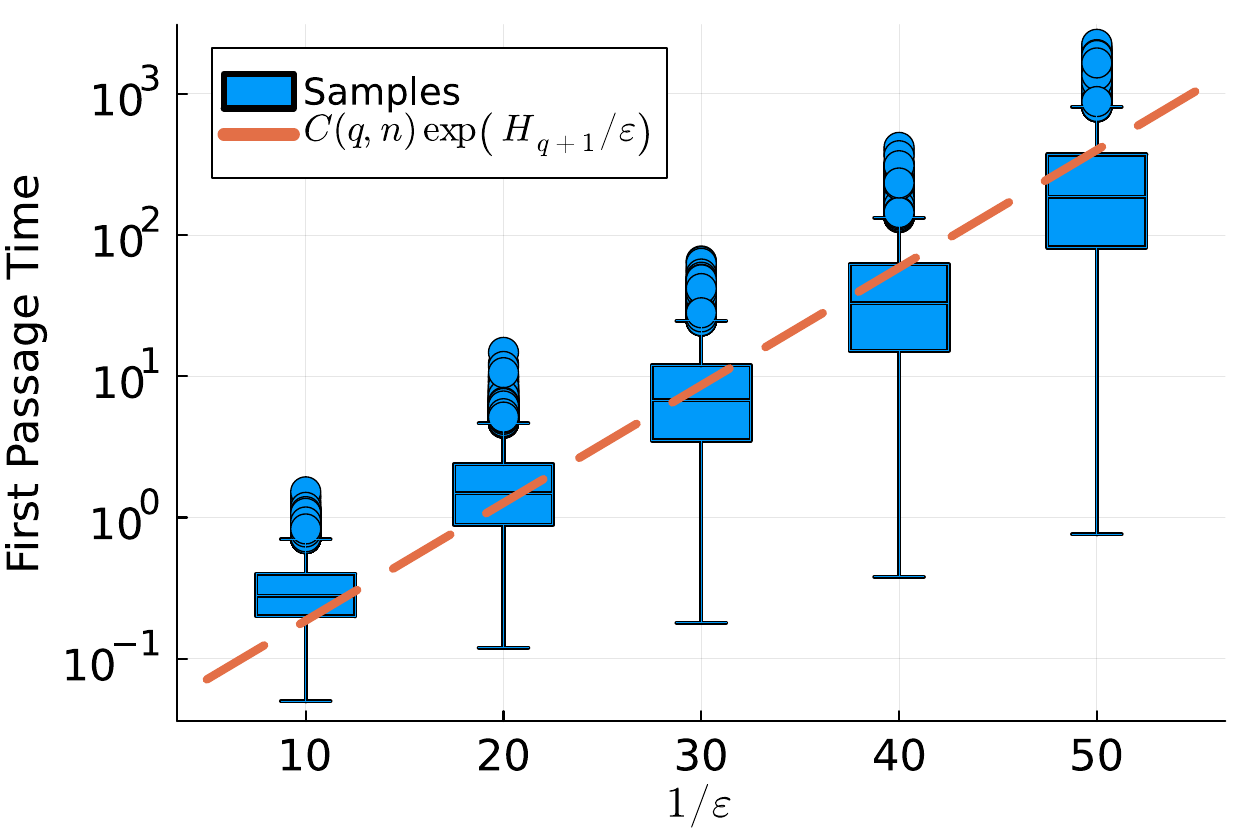}}

    \subfigure[$q=3$]{\includegraphics[width=6.5cm]{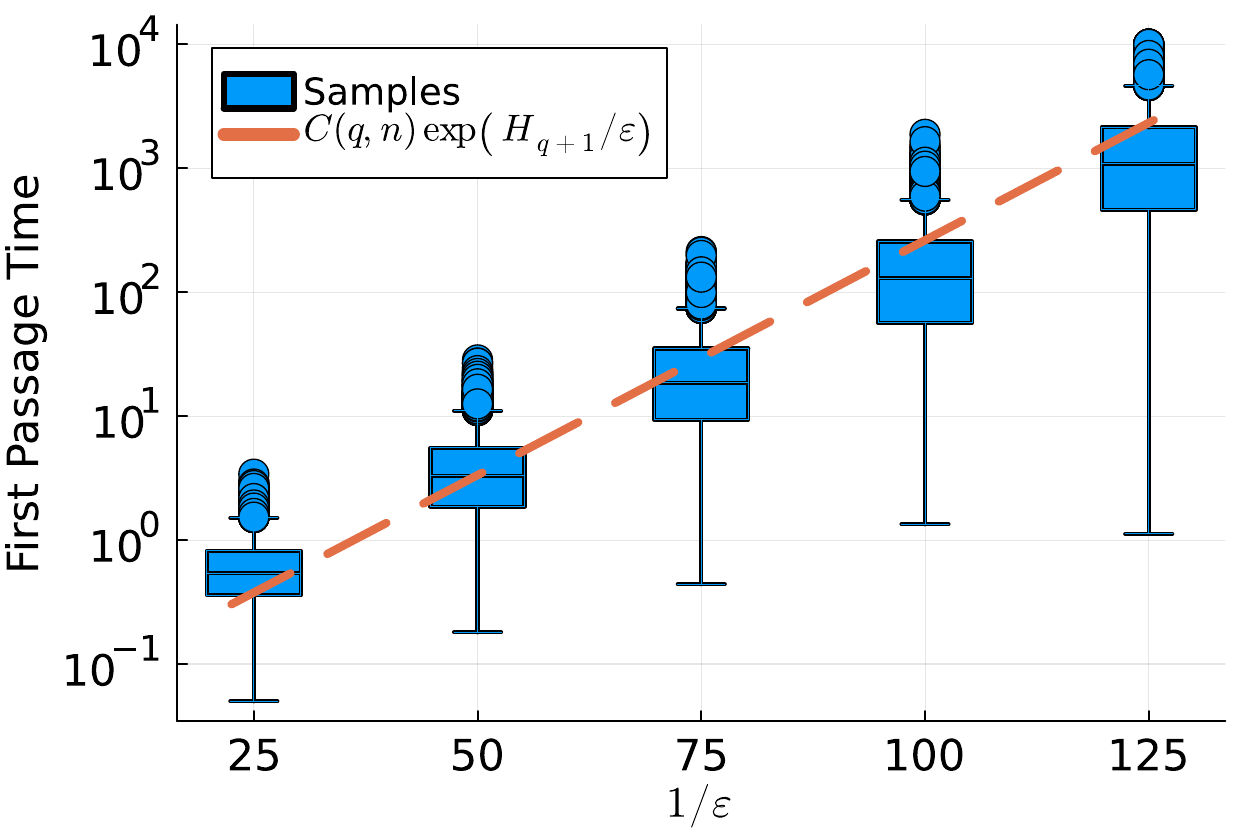}}
    \subfigure[$q=7$]{\includegraphics[width=6.5cm]{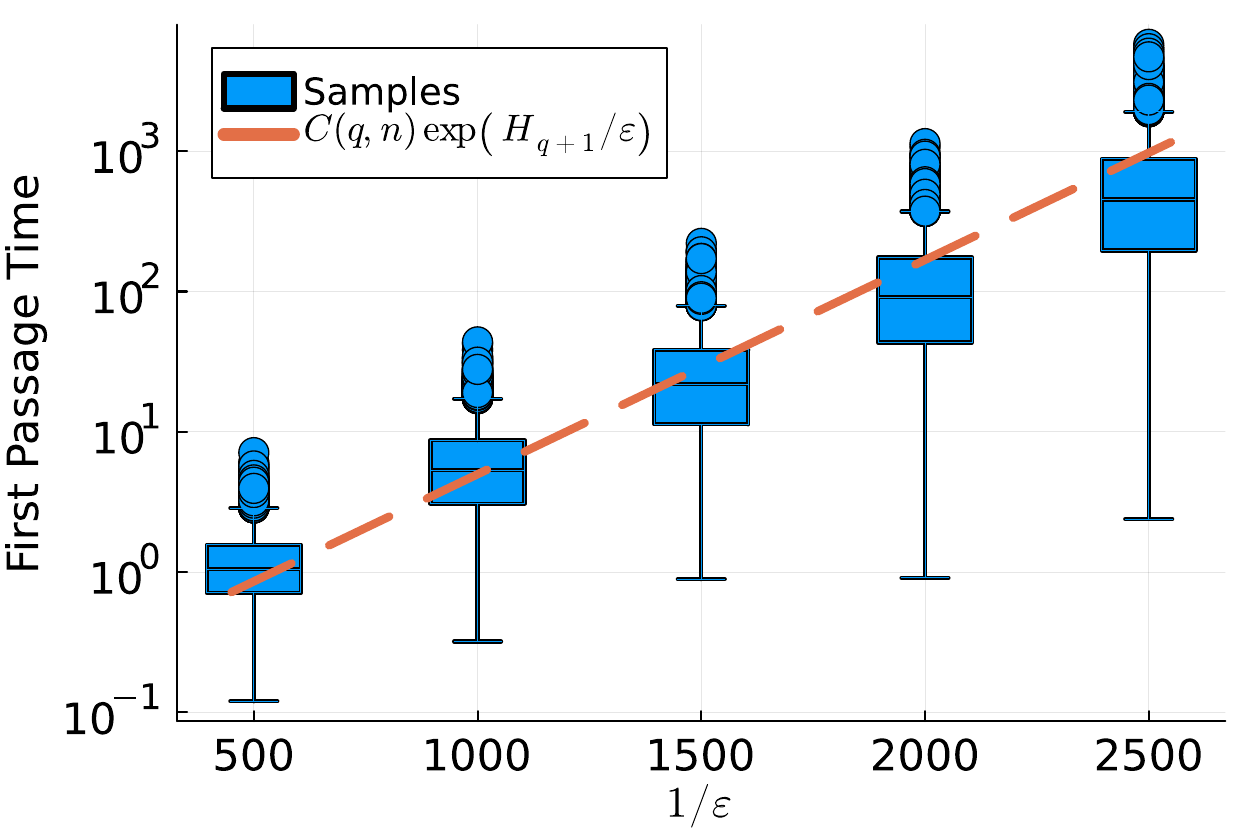}}
    
    \caption{Empirical first passage time distributions for the $n=40$ system, compared against the Eyring--Kramers formula \eqref{eq:EK}. For each value of $\eps$, $10^4$ independent trials were performed. {Units are dimensionless with $K=1$.}}
    \label{fig:fptn40}
\end{figure}

%\clearpage

\subsection{Sharp asymptotics of the prefactor and the energy barrier}

\begin{figure}[!t]
\centering 
\subfigure[$q=0$]{\includegraphics[width=6.5cm]{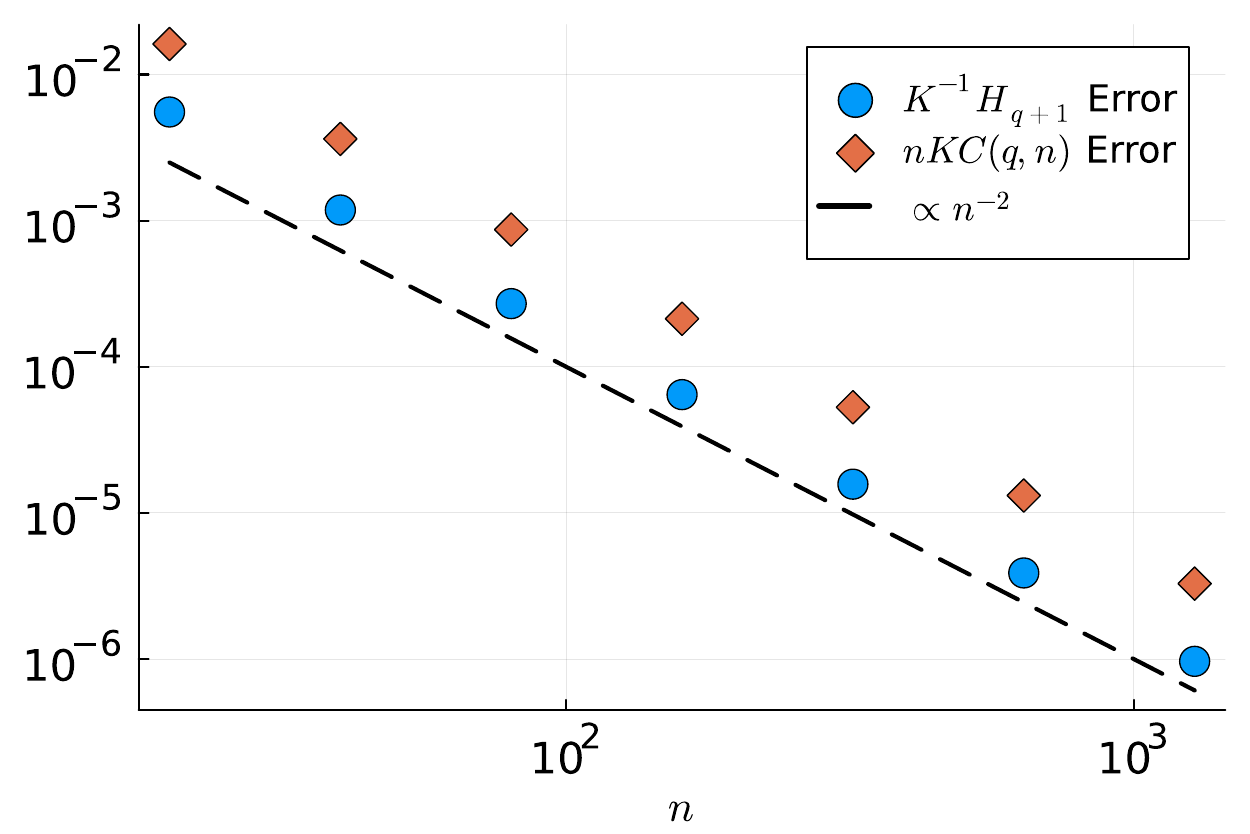}}
\subfigure[$q=1$]{\includegraphics[width=6.5cm]{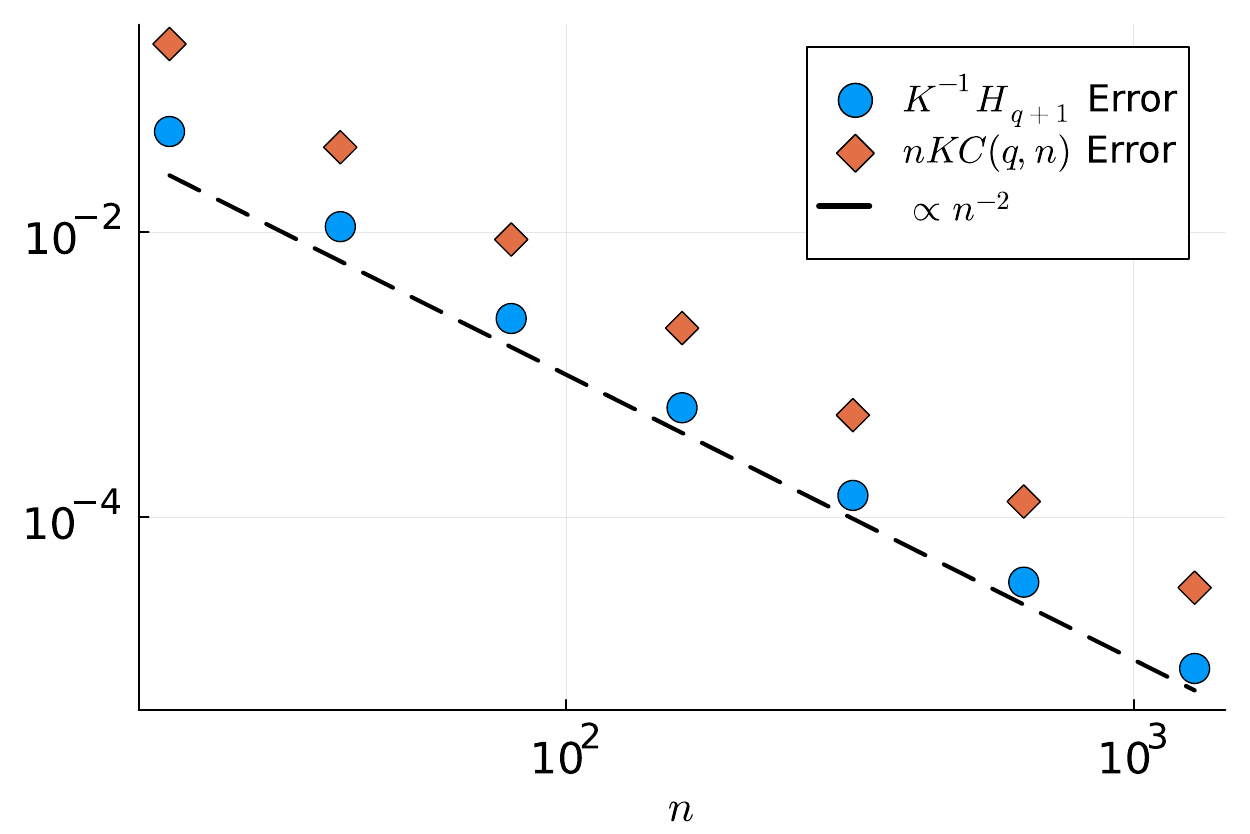}}

\subfigure[$q=2$]{\includegraphics[width=6.5cm]{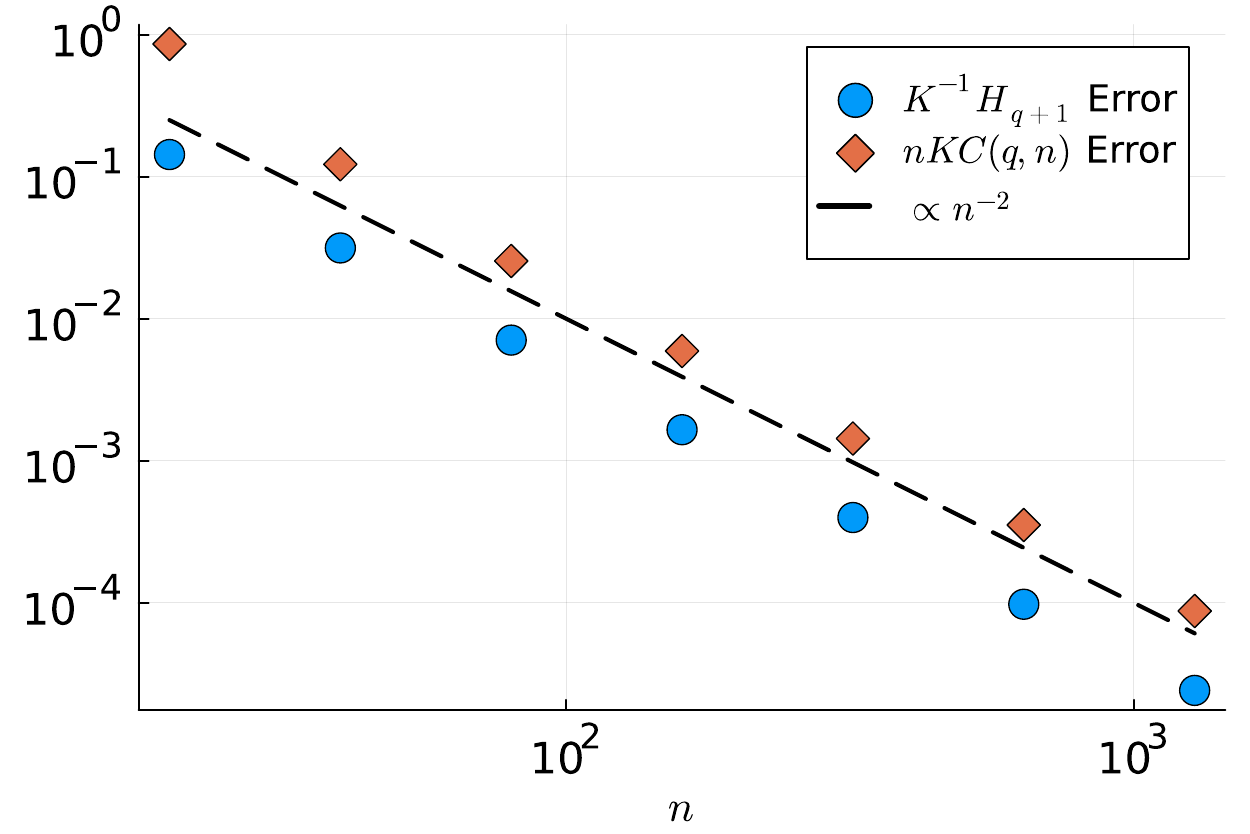}}
\subfigure[$q=3$]{\includegraphics[width=6.5cm]{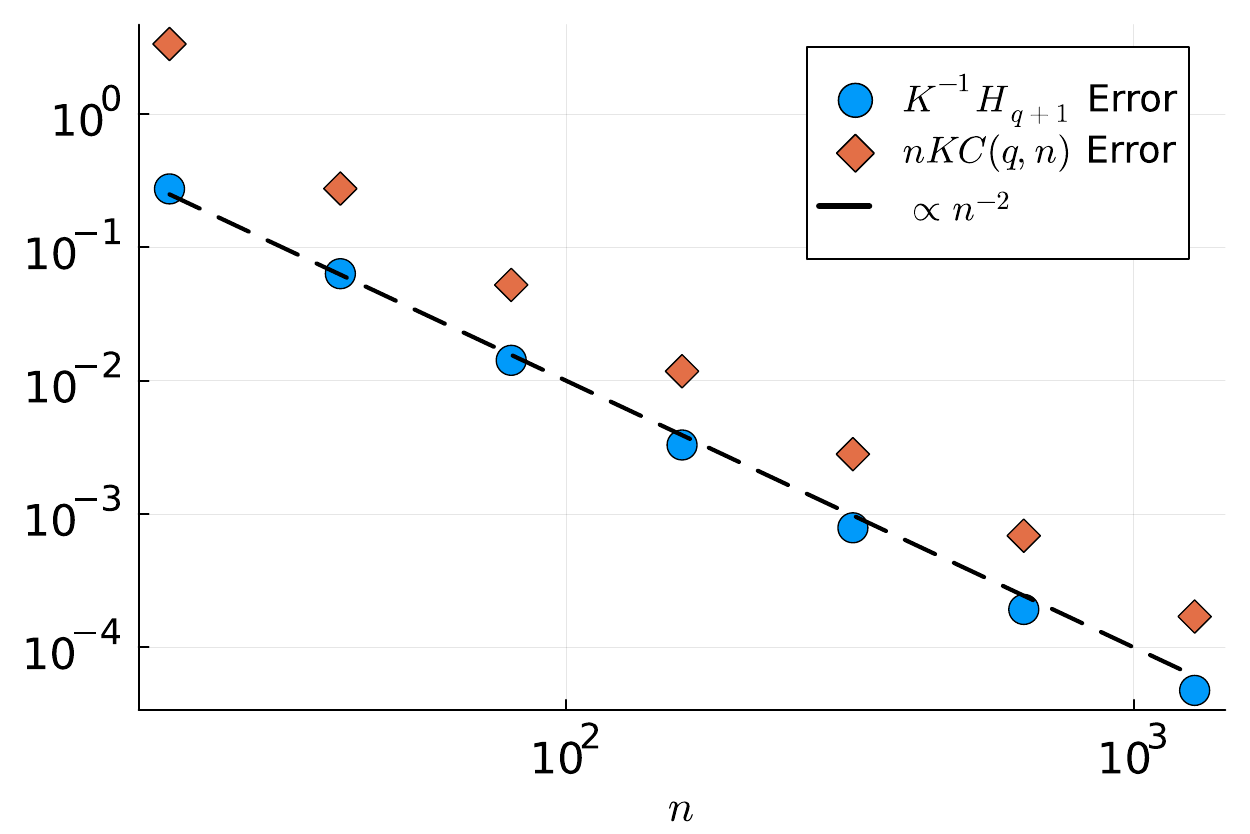}}

    \caption{Confirmation that, after rescaling, the constants $C(q,n)$ and $H_{q+1}$ obey the predictions of \eqref{e:Casympt} and \eqref{eq:Hasympt} in the large $n$ limit.}
    \label{fig:asympt}
\end{figure}

Next, we confirm the predictions of Theorem \ref{thm:main} by numerically computing the prefactor $C(q,n)$ and energy barrier $H_{q+1}$, then comparing them against the asymptotic formulae.  As shown in Figure \ref{fig:asympt}, subject to a rescaling, we have the anticipated behavior for $q=0,1,2,3$ across a broad range of $n$.  These were computed by direct evaluation of the energy and the eigenvalues of the Hessian for the relevant states.

\subsection{Beyond nearest neighbor interaction}
\label{ssec:num_beyond_nn} 

\begin{figure}[!t]
    \centering
    \subfigure[$r=2$]{\includegraphics[width = 6.5cm]{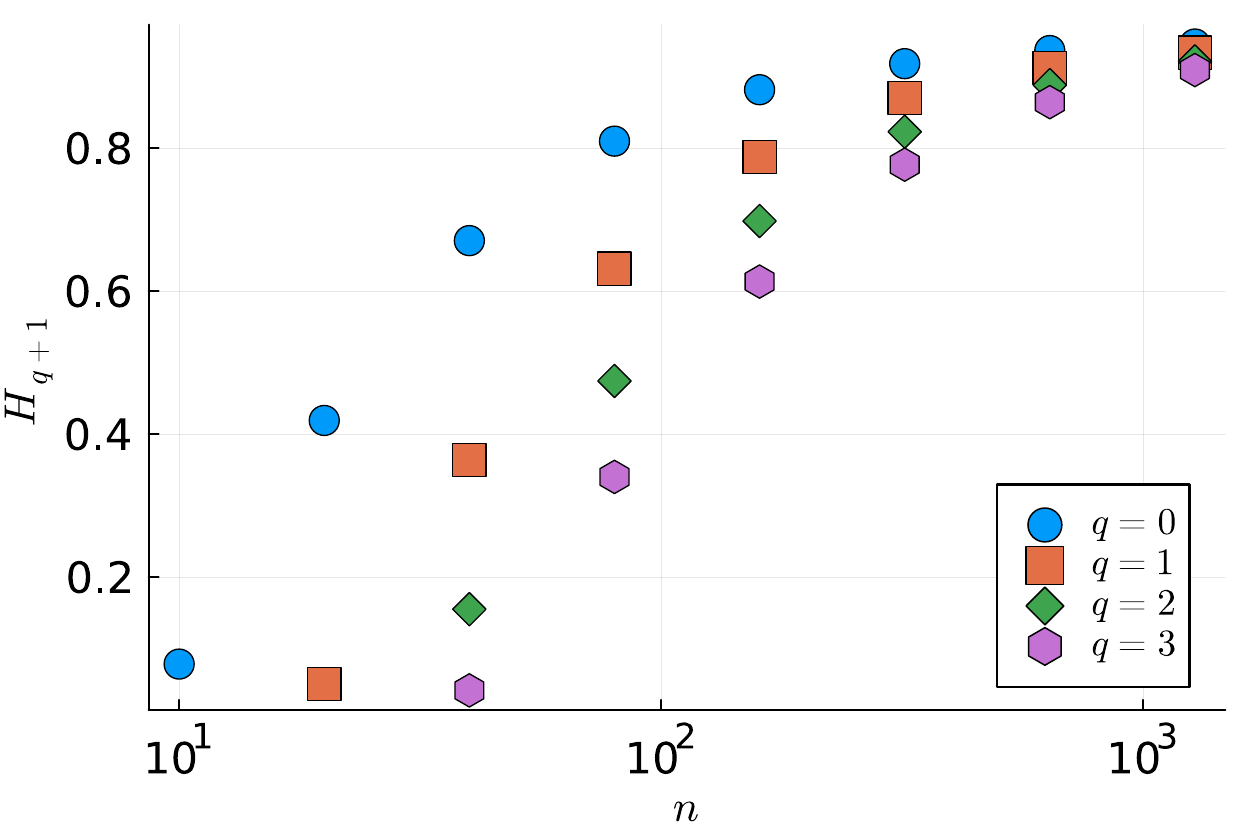}}
    \subfigure[$r=2$]{\includegraphics[width = 6.5cm]{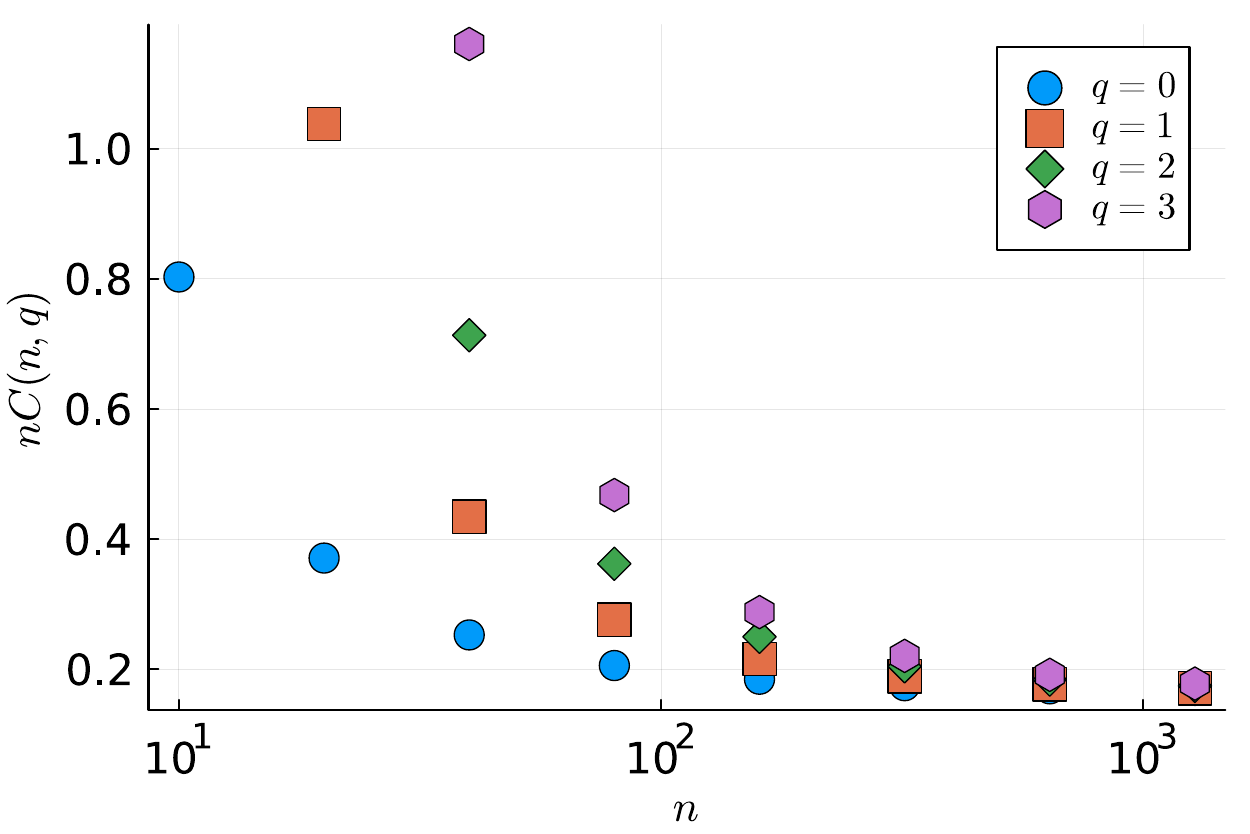}}

    \subfigure[$r=3$]{\includegraphics[width = 6.5cm]{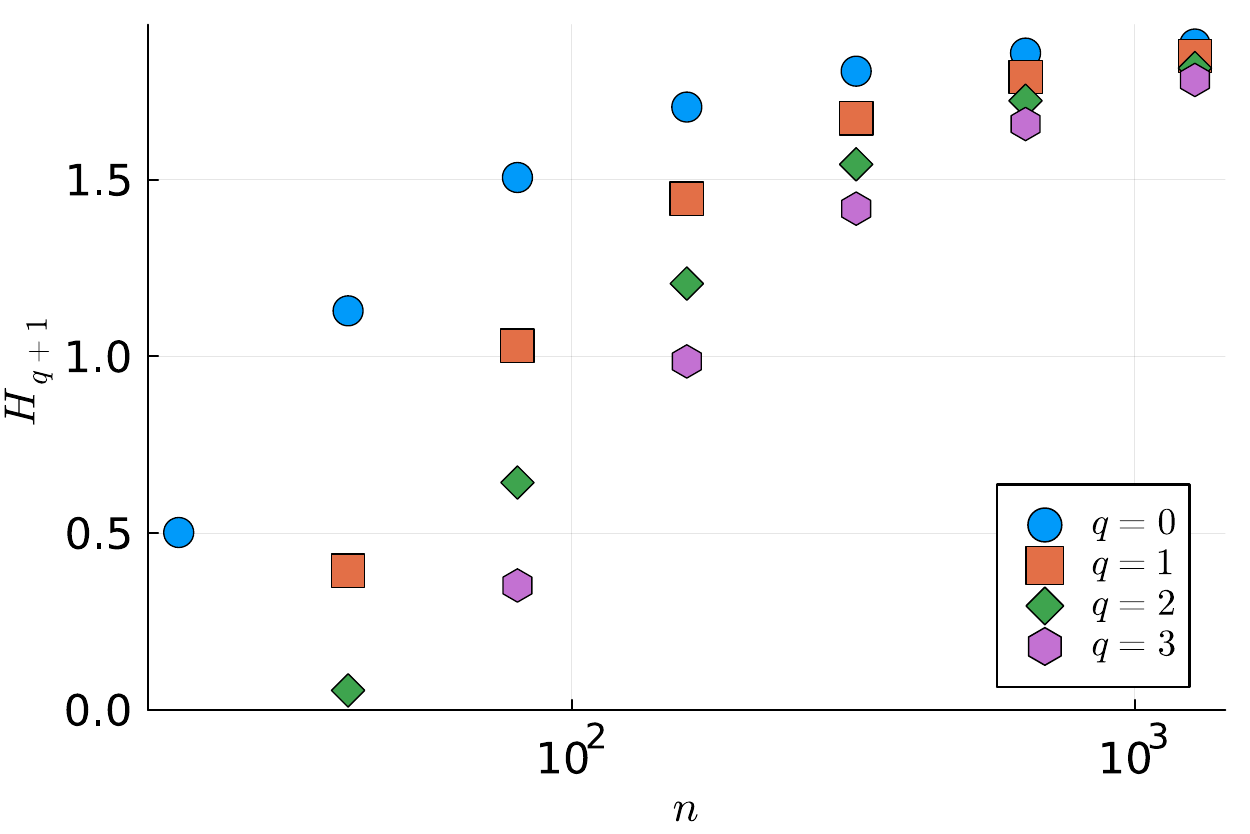}}
    \subfigure[$r=3$]{\includegraphics[width = 6.5cm]{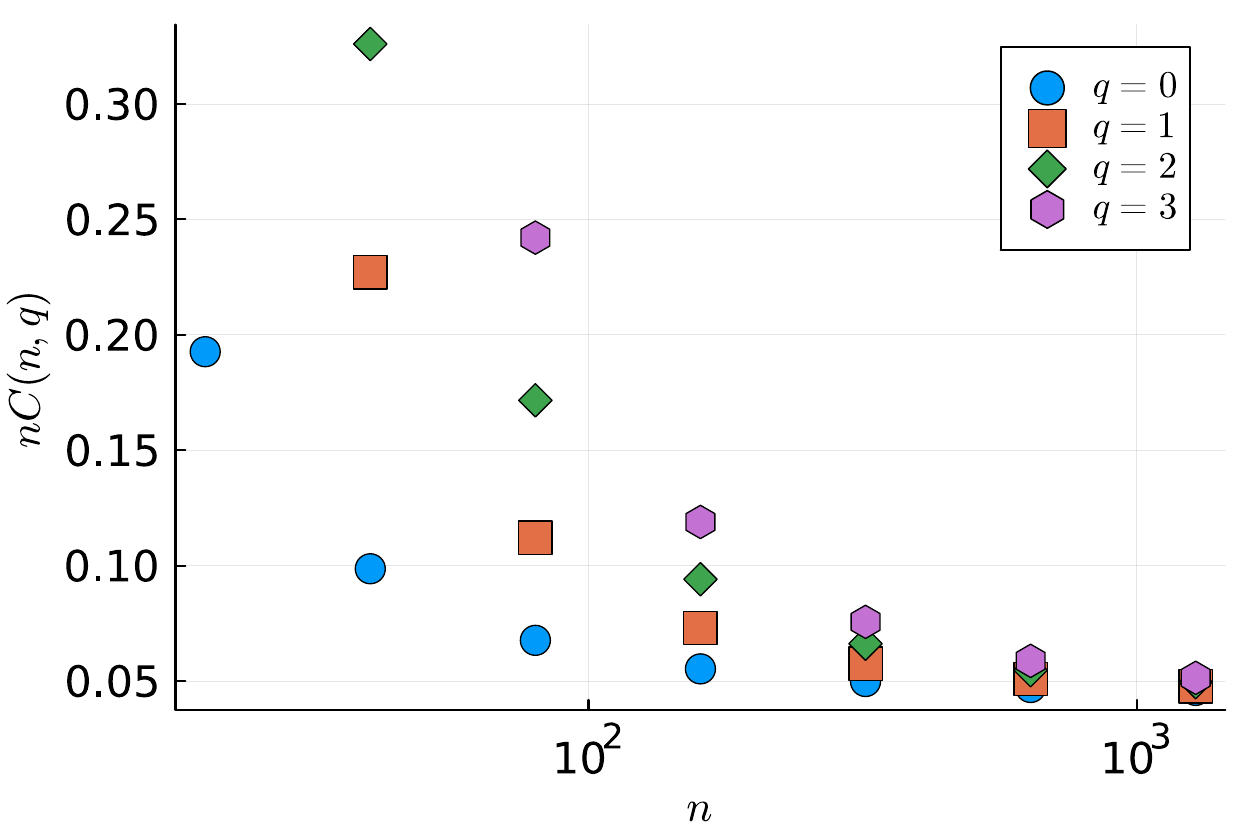}}

        \caption{Energy barriers and prefactors for the Eyring--Kramers formula in the case of $r>1$ nearest neighbors for several twisted states.}
    \label{fig:fixedknnvaryingn}
\end{figure}

Thus far, our study has been focused on the case of nearest neighbor interactions ($r=1$), but part of what makes Kuramoto models so interesting is their behavior when longer range interactions are included \cite{WilStr06, MedTan15b, Ome18, CMM2018, MM2022}. While we do not pursue a full numerical or analytic study here, we conclude with some  computations of the energy barriers and prefactors.  These were computed by finding the saddle point states using the string method and the climbing image method (cf.~\cite{weinan2007simplified}).

We consider the case of the system with $r>1$ nearest neighbors, and varying the system size, $n$, for several values of $q$.  As shown in Figure \ref{fig:fixedknnvaryingn}, we see the following properties. First, as is consistent with  \eqref{e:Casympt} for the $r=1$ case, as $n\to \infty$, $n C(q,n)$ tends to a finite positive constant, independent of $q$.  Additionally, larger $q$ values, at fixed $n$, have larger prefactors.  Analogous to \eqref{eq:Hasympt}, energy barrier $H_{q+1}$ tends to a a fixed constant.  At fixed $n$, larger values of $q$ have lower energy barriers.  We thus conjecture that formulas analogous to \eqref{e:Casympt} and \eqref{eq:Hasympt} are also valid in the nonlocal case.

\section{Discussion} 
\label{sec:discuss} 

In this work, we have described metastable transitions in a Kuramoto model with 
nearest-neighbor interaction, for arbitrary but finite number $n$ of oscillators. 
In particular, we have shown that the most likely transitions are those between 
$q$-twisted states $u^{(q_1)}$ and $u^{(q_2)}$ with $\abs{q_2} < \abs{q_1}$,
and obtained sharp Eyring--Kramers-type asymptotics for the expected transition time.

One interesting question is, what happens in the limit $n\to\infty$. 
Consider a sequence of systems on $\Lambda_n = \Z/n\Z$, given by 
\begin{equation}
 \6u_i = K_n \bigbrak{\sin\bigpar{2\pi(u_{i+1} - u_i)}
 + \sin\bigpar{2\pi(u_{i-1} - u_i)}} \6t + \sqrt{2\eps_n} \6W^i_t\;.
\end{equation} 
Assume that there exists a smooth interpolating function $\phi(t,x)$, such that 
\begin{equation}
 u_i = \phi\biggpar{t,\frac{i}{n}}
 \qquad \forall i\in\Lambda_n\;.
\end{equation} 
Since 
\begin{equation}
 \sin\bigpar{2\pi(u_{i\pm1} - u_i)} 
 = \pm\frac{2\pi}{n} \partial_x\phi\biggpar{t, \frac in \pm \frac1{2n}} 
 + \mathcal{O} \biggpar{\frac{1}{n^3}}\;, 
\end{equation} 
we obtain that $\phi$ satisfies the equation 
\begin{equation}
\label{eq:KM_continuous} 
 \6\phi(t,x) 
 = \biggbrak{\frac{2\pi K_n}{n^2} \partial_{xx}\phi(t,x) + \mathcal{O} \biggpar{\frac{1}{n^3}}} \6t 
 + \sqrt{2\eps_n} \6W(t,x)\;,
\end{equation} 
where $\6W(t,x)$ denotes space-time white noise. One natural scaling regime is obtained by 
setting $K_n = n^2$. If we choose $\eps_n = \eps$ to be constant, \eqref{eq:KM_continuous} 
converges formally to a stochastic heat equation on the torus, where $\phi$ also has values 
in the torus. 

Note however that the relative communication height between $q$-twisted states satisfies 
\begin{equation}
 H_n = \frac{K_n}{\pi} \biggpar{1 + \mathcal{O} \biggpar{\frac{1}{n}}}\;.
\end{equation} 
Therefore, to leading order, the mean transition time between $q$-twisted states 
is given by 
\begin{equation}
 \frac{3\pi}{2n} \e^{K_n/(\pi\eps_n)}\;.
\end{equation} 
This time diverges as $n\to\infty$ if $K_n = n^2$ and $\eps_n = \eps$ (this is of course 
assuming that the error term $R_n(\eps)$ in Theorem~\ref{thm:main} remains bounded, which we 
have not proved). The interpretation of this is that because of the jump discontinuity of
the $1$-saddles $u^{r}$, $r$ a half-integer, these states do not converge to continuous functions 
in the continuum limit. Instead, they have an infinite slope at one point, and this makes 
their energy blow up. As a result, we do not expect to see any metastable transitions in 
the continuum limit, at least for this particular scaling. 

This may change, however, when one goes beyond nearest-neighbor coupling. Assume that 
the coupling set $S$ in~\eqref{eq:Kuramoto} is given by 
\begin{equation}
 S = \bigset{-r(n),\dots,r(n)}
\end{equation} 
so that the range is $r(n)$. As discussed in Section~\ref{ssec:num_beyond_nn}, numerical simulations indicate that similar Eyring--Kramers asymptotics as those obtained for nearest-neighbor coupling hold for general interaction ranges. If $r(n)$ scales like a constant times $n$, $1$-saddles visited during transitions between $q$-twisted states may have a smoother space dependence, in which the jump is replaced by a boundary layer. For suitable parameter values, this could imply that these transition states have a finite energy in the $n\to\infty$ limit, which would result in a finite value for the expectation of metastable transition times. 

%%%%%%%%%%%%%%%%%%%%%%%%%%%%%%%%%%%%

\subsection*{Acknowledgement}

NB was supported by the ANR project PERISTOCH, ANR--19--CE40--0023. 
The work of GM was partially supported by National Science Foundation (NSF) through grants  DMS--2009233 and DMS--2406941.
GS was supported by NSF grant number DMS--2111278. 
Work reported here was run on hardware supported by Drexel’s University Research Computing Facility. 

This work was motivated by numerical experiments conducted by a group of undergraduate students during two NSF supported summer REUs at Drexel University. John Bonnes, Nicholas De Filippis, Zachary Lee, and Jacob Woods contributed at the early stages of this project.

Finally, we thank the two referees for their constructive comments that led to an improved presentation.

%%%%%%%%%%%%%%%%%%%%%%%%%%%%%%%%%%%%

\bibliographystyle{amsplain}
% \bibliography{meta}
\bibliography{metastab}
%%%%%%%%%%%%%%%%%%%%%%%%%%%%%%%%%%%%

\bigskip\bigskip\noindent
{\small
Nils Berglund \\
Institut Denis Poisson (IDP) \\ 
Universit\'e d'Orl\'eans, Universit\'e de Tours, CNRS -- UMR 7013 \\
B\^atiment de Math\'ematiques, B.P. 6759\\
45067~Orl\'eans Cedex 2, France \\
{\it E-mail address: }
{\tt nils.berglund@univ-orleans.fr}

\bigskip\bigskip\noindent
{\small
  Georgi S. Medvedev\\
  Department of Mathematics\\
  Drexel University\\
  3141 Chestnut Street\\
  Philadelphia, PA 19104, USA\\
{\it E-mail address: }
{\tt medvedev@drexel.edu}

\bigskip\bigskip\noindent
{\small
  Gideon Simpson\\
  Department of Mathematics\\
  Drexel University\\
  3141 Chestnut Street\\
  Philadelphia, PA 19104, USA\\
{\it E-mail address: }
{\tt  grs53@drexel.edu}

\end{document}